\documentclass{article}
\usepackage[margin=3cm]{geometry}

\usepackage{microtype}
\usepackage[dvipsnames]{xcolor}
\usepackage{amsmath,amssymb,amsthm,mathtools}
\usepackage{bm}
\usepackage{mathrsfs}
\usepackage{enumitem}
\usepackage[colorlinks=true,allcolors=blue]{hyperref}
\usepackage[nameinlink,capitalise,noabbrev]{cleveref}
\usepackage{xspace}

\numberwithin{equation}{section}

\theoremstyle{plain}
\newtheorem{theorem}{Theorem}[section]
\newtheorem{proposition}[theorem]{Proposition}
\newtheorem{lemma}[theorem]{Lemma}
\newtheorem{corollary}[theorem]{Corollary}

\theoremstyle{definition}
\newtheorem{assumption}[theorem]{Assumption}
\newtheorem{definition}[theorem]{Definition}

\theoremstyle{remark}
\newtheorem{remark}[theorem]{Remark}
\crefname{assumption}{Assumption}{Assumptions}
\Crefname{assumption}{Assumption}{Assumptions}

\newcommand{\restatementname}{}
\theoremstyle{plain}
\newtheorem*{restatementinner}{\restatementname}
\newenvironment{restatement}[3][]
  {%
    \renewcommand{\restatementname}{#2~\ref{#3}}%
    \begin{restatementinner}[#1]%
  }
  {%
    \end{restatementinner}%
  }

\newcommand\AverageSmallMatrix[1]{{%
\footnotesize\arraycolsep=0.22\arraycolsep\ensuremath{\begin{bmatrix}#1\end{bmatrix}}}}
\newcommand\SmallMatrix[1]{{%
\tiny\arraycolsep=0.5\arraycolsep\ensuremath{\begin{bmatrix}#1\end{bmatrix}}}}

\newcommand{\R}{\mathbb{R}}
\newcommand{\Prob}{\mathbb{P}}
\newcommand{\E}{\mathbb{E}} 
\newcommand{\N}{\mathbb{N}}
\newcommand{\F}{\mathcal{F}}
\newcommand{\dd}{\mathop{}\!\mathrm{d}}

\DeclarePairedDelimiterX{\ip}[2]{\langle}{\rangle}{#1,#2}

\newcommand\nablaof[1]{\nabla_{\hspace{-1pt}#1}}

\newcommand{\Int}{\textrm{Int}} 

\renewcommand{\P}{\mathcal{P}}

\newcommand{\U}{\mathcal{U}}

\newcommand{\C}{\mathcal{C}}
\newcommand{\sC}{\mathscr{C}}
\newcommand{\sD}{\mathscr{D}}

\newcommand{\bX}{\mathbb{X}}

\newcommand{\bB}{\mathbb{B}}
\newcommand{\mbX}{\mathbf{X}}

\newcommand{\mbB}{\mathbf{B}}

\newcommand\OCP[1]{\hyperref[OCP]{$\textbf{OCP}_{#1}$}\xspace}
\newcommand\SOCP[1]{\hyperref[SOCP]{$\textbf{SOCP}_{#1}$}\xspace}
\newcommand\PMP{\hyperref[PMP]{$\textbf{PMP}$}\xspace}
\newcommand{\ocp}{\hyperref[OCP]{\textbf{OCP}}\xspace}

\newcommand{\pmp}{\hyperref[PMP]{\textbf{PMP}}\xspace}

\newcommand{\ellone}[1]{#1}

\usepackage{fix-cm} 
\usepackage{multicol}
\usepackage{cite}
\usepackage[many]{tcolorbox}

\makeatletter
\newcommand{\customlabel}[2]{%
\protected@write\@auxout{}{\string \newlabel{#1}{{#2}{\thepage}{#2}{#1}{}}}%
\hypertarget{#1}{#2}
}
\makeatother

\usepackage[titles]{tocloft}
\title{
\bf
\fontsize{17pt}{24pt}\selectfont
\vspace{-8mm}
Unified Optimality Conditions for Stochastic Optimal Control in the Rough Path and It\^o Frameworks}

\author{Thomas Lew\footnote{Toyota Research Institute. Email: thomas.lew@tri.global}}

\begin{document}

\maketitle

\begin{abstract}
Stochastic differential equations (SDEs) can be studied via It\^o calculus and rough path theory. For stochastic optimal control, these two frameworks give distinct Pontryagin Maximum Principle (PMP) optimality conditions 
with forward-backward SDEs (FBSDEs) or rough differential equations. 
We show that the adjoint equations of the It\^o and rough PMPs  are connected via the conditional expectation
$$
p_t^{\textrm{It\^o}}=\E\big[p_t^{\textrm{rough}} \mid \mathcal{F}_t \big],
$$
where $\mathcal{F}_t$ represents information available at time $t$. 
First, we derive a rough stochastic PMP for problems with adapted controls that does not use FBSDEs. 
Its proof extends the rough stochastic PMP over deterministic controls by considering stochastic needle variations. 
Second, we derive a unified PMP connecting the It\^o and rough PMPs, 
using  It\^o-Stratonovich conversion formulas and duality identities between the forward tangent and backward adjoint SDEs. 
As a first application, we rederive the adjoint matching method for fine-tuning generative models. 
As a second application, we  propose an indirect shooting method for a class of feedback problems. 
Overall, these results give a new conditional bridge connecting two popular frameworks for stochastic optimal control.
\end{abstract}

\setcounter{tocdepth}{2}
\tableofcontents

\section{Introduction and Main Results}
Stochastic optimal control has many applications such as in 
finance, robotics, and generative modeling,  
and provides insights on dynamical systems and biological phenomena \cite{Berret2020}.
Studying the structure of optimal solutions can reveal low-dimensional representations and %
more efficient algorithms, such as dynamic programming \cite{Yong1999}, indirect shooting methods \cite{Lew2026}, and data-driven approximations \cite{E2017}.

Using It\^o calculus, %
solutions to stochastic optimal control problems can be characterized through Hamilton-Jacobi-Bellman  partial differential equations and the stochastic Pontryagin maximum principle (PMP) \cite{Yong1999}. A limitation of the classical It\^o framework is that it does not apply to non-semimartingale processes such as fractional Brownian motion (fBM). Also, the It\^o PMP is formulated using forward-backward stochastic differential equations (FBSDEs), where the forward equation describes the state trajectory, and the backward adjoint equation encodes the effect of future costs while remaining adapted to the information available at each time. This adapted backward representation is fundamental to the It\^o formulation, but the resulting FBSDE structure complicates the analysis and numerical resolution of nonlinear problems.

Stochastic optimal control has also been studied using rough path theory \cite{Lyons2002}, which enables a pathwise analysis that can handle non-semimartingale processes such as fBM and problems with anticipative controls \cite{Diehl2016,Horst2026}, with applications to robust filtering \cite{Allan2020} and reinforcement learning \cite{Ashkarian2026}. 
A difficulty, however, is that adaptedness and non-anticipativity  are not naturally encoded by a pathwise formulation. 
To handle problems with non-anticipative controls, duality methods add  and  optimize over suitable penalties for anticipative controls \cite{HADavis1992,Buckdahn2007,Rogers2007,Diehl2016,Bank2026}.   
Alternatively, problems with deterministic open-loop controls are studied in \cite{Lew2026}. This approach enables optimizing over parameterized non-anticipative controls, but it restricts the search space. %
\textit{Can we derive a PMP for adapted controls using rough path theory only, without restricting the control parameterization or penalizing anticipative controls?}

A second question arises when both It\^o and rough path theories apply. 
Despite their common goal, results obtained via these approaches appear in disjoint forms: It\^o and rough path PMPs have different Hamiltonians and adjoint equations  despite describing the same problem. 
\textit{In such settings, how are the optimality conditions from these two frameworks related?}

Recent work \cite{FrizControlled2024,Bank2026,Horst2026} combines both theories to study systems described by rough-It\^o SDEs
$
\dd x_t=b(t,x_t,u_t)\dd t+\sigma_1(t,x_t,u_t)\dd B_t+\sigma_2(t,x_t)\dd W_t
$,  
providing a useful generalization that gives It\^o or rough path optimality conditions if either $\sigma_1=0$ or $\sigma_2=0$. 
In this work, we do not consider SDEs driven simultaneously by rough and It\^o signals. 
Instead, we study problems that admit both an It\^o and an equivalent rough path formulation, and look for unified optimality conditions that relate the two frameworks. 
Compared to rough path duality methods \cite{Diehl2016,Bank2026}, we do not  penalize anticipative controls, and instead derive optimality conditions with a conditional expectation.

\paragraph{Problem setting and main results.} 
Let $B$ be a $d$-dimensional Brownian motion on a filtered probability space $(\Omega,\F,(\F_t)_{0\leq t\leq T},\Prob)$, and  the state  process solve the  It\^o SDE
\begin{equation}\label{eq:sde:ito}
\dd x_t=b(t,x_t,u_t)\dd t+\sigma(t,x_t)\dd B_t,
\qquad x_0=\bar x_0. 
\end{equation}
where $\sigma$ is independent of the control $u$. %
This state process also solves 
the rough SDE
\begin{equation}\label{eq:sde:rough}
\dd x_t=\mathbf b(t,x_t,u_t)\dd t+\sigma(t,x_t)\dd\mbB_t,
\qquad x_0=\bar x_0,
\end{equation} 
where $\mbB$ is the Stratonovich lift of $B$, and $\mathbf b$ is the corresponding rough (or Stratonovich) drift
\begin{equation}
\mathbf b(t,x,u)
:=
b(t,x,u)
-
\tfrac12 \nabla_x\sigma(t,x)\sigma(t,x),
\label{eq:rough-drift}
\end{equation}
see Section \ref{sec:background} for details.

We consider  stochastic optimal control problems (OCPs) 
\begin{align}
\label{OCP} 
\tag{\ocp}
&\begin{cases}
\min\limits_{u\in \mathcal U} 
	&\E\left[
\int_0^T f(t,x_t,u_t)\dd t+g(x_T)\right]
\\
\ \textrm{s.t.}   
&
\dd x_t = \mathbf b(t,x_t,u_t)\dd t +\sigma(t,x_t)\dd\mbB_t,
\quad t\in[0,T],
\\
&x_0=\bar{x}_0,
\\
&\E\left[h(x_T)\right]=0,
\end{cases}
\end{align}
where the control $u$ is in the class of non-anticipative, $\F$-adapted controls 
\begin{equation}
u\in \mathcal U := L^2_{\F}([0,T]\times\Omega,U)
\end{equation}
with values in $U\subseteq\R^m$, $f$ and $g$ are cost functions, and $h$ encodes terminal state constraints.

\paragraph{Rough PMP.}  %
Our first contribution is the following rough path optimality conditions for \OCP{}. 

\begin{tcolorbox}[
title={%
\textbf{Theorem \customlabel{thm:pmp:rough}{1.1}}\textbf{(Rough Stochastic Pontryagin Maximum Principle)}\phantom{$A^1_1$}
}, 
coltitle=black, 
boxrule=3pt,
colframe=blue!6, %
enhanced, colback=blue!2, boxsep=1pt, left=5pt, right=5pt] 
Define $T,\Omega,\F,\Prob,\mbB,U,\mathbf b,\sigma,f,g,h,\bar x_0$ as in  Assumption \ref{assum:pmp:rough}, 
where $\mbB$ is an enhanced Gaussian process. 
Define the Hamiltonian 
$\mathbf H:[0,T]\times\R^n\times\R^m\times\R^n\times\R\to\R$ by
$$
\mathbf H(t,x,u,\mathbf p,\mathfrak p_0)
=
\mathbf p^\top\mathbf b(t,x,u)
+
\mathfrak p_0 f(t,x,u).
$$ 
Let $(x,u)$ be an optimal solution to  \OCP{}. 
Then, there exist a rough
adjoint stochastic process $\mathbf p\in L^2(\Omega,C([0,T],\R^n))$  
and nontrivial multipliers $(\mathfrak p_0,\mathfrak p_1,\dots,\mathfrak p_r)\neq0$
with $\mathfrak p_0\in\{-1,0\}$, such that:
\begin{enumerate}[label=(\roman*)]
\item \textbf{Adjoint Equation}: 
for some initial conditions $\mathbf p_0\in L^2(\Omega,\R^n)$, 
$\mathbf p$ solves the rough SDE
\begin{align}
\label{eq:pmp:rough:adjoint}
\hspace{-6mm}
\dd \mathbf p_t
=
-\nabla_x\mathbf H(t,x_t,u_t,\mathbf p_t,\mathfrak p_0)\dd t
-
\nabla_x\sigma(t,x_t)^\top\mathbf p_t\dd\mbB_t,
\quad t\in[0,T].
\end{align}
\item \textbf{Transversality Condition}: 
almost surely, 
\begin{equation}
\label{eq:pmp:rough:transversality}
\mathbf p_T
=
\mathfrak p_0\nabla g(x_T)
+
\smash{
\sum_{i=1}^r
\mathfrak p_i\nabla h_i(x_T).
}
\end{equation}
\item \textbf{Maximality Condition}: 
for $(\dd t\otimes\Prob)$-almost every $(t,\omega)$, 
\begin{equation}
\label{eq:pmp:rough:maximality}
u_t
\in
\mathop{\arg\max}_{v\in U}\, 
\mathbf H
(
t,x_t,v,
\E[\mathbf p_t\mid\F_t],
\mathfrak p_0
).
\end{equation}
\end{enumerate} 
\end{tcolorbox} 

\noindent This result is proved using rough path theory only, so it avoids It\^o FBSDEs. 
Its proof extends the proof of the open-loop PMP  \cite{Lew2026} using stochastic needle variations. The result extends previous rough PMPs to  adapted controls and to problems with constraints $\E[h(x_T)]=0$, 
without penalizing anticipative controls like duality methods.  
The rough adjoint $\mathbf p$ is defined pathwise, backwards-in-time from $\mathbf p_T$, and encodes information about the entire path of $(x,\mbB)$. 
\textit{Taking the conditional expectation $\E[\mathbf p_t\mid\F_t]$ in the maximality condition resolves the anticipativeness of the rough adjoint $\mathbf p$} to give adapted controls.

\textit{Anticipative controls}: If we instead optimize over anticipative controls
$$
u\in \mathcal U^{\textrm{anticipative}} := L^2([0,T]\times\Omega,U),
$$
then Theorem \ref{thm:pmp:rough} holds with the maximality condition replaced with
\begin{equation}\label{eq:pmp:rough:maximality:anticipative}
u_t
\in
\mathop{\arg\max}\limits_{v\in U}
\mathbf H
\left(
t,x_t,v,\mathbf p_t,\mathfrak p_0
\right).
\end{equation}
Such pathwise optimality conditions were proven in \cite{Diehl2016}. Theorem \ref{thm:pmp:rough} states how to change the maximality condition to yield adapted controls, by replacing  $\mathbf p$ with its conditional expectation.

\paragraph{Unified PMP.} Our second contribution is showing that, for problems where both It\^o calculus and rough path theory apply, the It\^o and rough path PMPs are connected by a conditional projection.

\begin{tcolorbox}[
title={%
\textbf{Theorem \customlabel{thm:pmp}{1.2}}\textbf{(Unified Stochastic Pontryagin Maximum Principle (\PMP))}\phantom{$A^1_1$}
},
coltitle=black, 
boxrule=3pt,
colframe=blue!6, %
enhanced, colback=blue!2, boxsep=1pt, left=5pt, right=5pt] 
\label{PMP}
Define $T,\Omega,\F,\Prob,B,\mbB,U,b,\sigma,f,g,h,\bar x_0$ as in  Assumption \ref{assum:pmp} and  
the It\^o and rough Hamiltonians $H:[0,T]\times\R^n\times\R^m\times\R^n\times\R^{n\times d}\times\R\to\R$ 
and 
$\mathbf H:[0,T]\times\R^n\times\R^m\times\R^n\times\R\to\R$ by
\vspace{-3mm}
\begin{align*}
H(\scalebox{0.92}{$t,x,u,p,q,\mathfrak p_0$})
&=
p^\top b(\scalebox{0.92}{$t,x,u$})
+
\mathfrak p_0 f(\scalebox{0.92}{$t,x,u$})
+
\sum_{j=1}^d q^{j\top} \sigma^j(\scalebox{0.92}{$t,x$}),
\\[-1mm]
\mathbf H(\scalebox{0.92}{$t,x,u,\mathbf p,\mathfrak p_0$})
&=
\mathbf p^\top \mathbf b(\scalebox{0.92}{$t,x,u$})
+
\mathfrak p_0 f(\scalebox{0.93}{$t,x,u$}).
\end{align*}
Let $(x,u)$ be an optimal solution to \ocp.  
Then, there exist stochastic processes
$$(p,q)\in L^2_\F(\Omega,C([0,T],\R^n)) \times  L^2_\F([0,T]\times\Omega,\R^{n\times d}),
\qquad \qquad 
\mathbf p\in L^2(\Omega,C([0,T],\R^n)),
$$ 
called adjoint vectors, and nontrivial multipliers $(\mathfrak p_0,\mathfrak p_1,\dots,\mathfrak p_r)\neq0$
with $\mathfrak p_0\in\{-1,0\}$, such that:
\begin{enumerate}[label=(\roman*),leftmargin=6mm]

\item \textbf{Adjoint Equations}: 
$(p,q)$ and $\mathbf p$ solve the It\^o and rough SDEs
$$
\dd p_t
=
-\nabla_xH(\scalebox{0.9}{$t,x_t,u_t,p_t,q_t,\mathfrak p_0$})\dd t
+q_t\dd B_t,
\quad
\dd \mathbf p_t
=
-\nabla_x\mathbf H(\scalebox{0.9}{$t,x_t,u_t,\mathbf p_t,\mathfrak p_0$})\dd t
-
\nabla_x\sigma(\scalebox{0.9}{$t,x_t$})^\top\mathbf p_t\dd\mbB_t
$$
over $t\in[0,T]$. 
\textcolor{gray}{(see Theorem  \ref{thm:sdes:ito} and Theorem \ref{thm:sdes:rough})}

\item \textbf{Transversality Conditions}: almost surely,
\begin{multicols}{2}
\noindent
$$
p_T%
=
\mathfrak p_0\nabla g(x_T)
+
\smash{\sum_{i=1}^{r}}\,\mathfrak p_i\nabla h_i(x_T),
$$

\columnbreak

\noindent
$$ 
\mathbf p_T%
=
\mathfrak p_0\nabla g(x_T)
+
\smash{\sum_{i=1}^{r}}\,\mathfrak p_i\nabla h_i(x_T).
$$
\end{multicols}

\item \textbf{Maximality Conditions}: 
for $(\dd t\otimes\Prob)$-almost every $(t,\omega)$, 
\begin{multicols}{2}
\noindent
$$
u_t
\in
\smash{\mathop{\arg\max}_{v\in U}}\,
H(t,x_t,v,p_t,q_t,\mathfrak p_0),
$$

\columnbreak

\noindent
$$
u_t
\in
\smash{\mathop{\arg\max}_{v\in U}}\,
\mathbf H(t,x_t,v,\E[\mathbf p_t\, |\, \F_t],\mathfrak p_0).
$$
\end{multicols}
\item  \textbf{Conditional Bridge}:
for $(\dd t\otimes\Prob)$-almost every $(t,\omega)$, 
$$
p_t=\E\!\left[ \mathbf p_t \ \big|\ \F_t \right].
$$
\end{enumerate} 
\end{tcolorbox} 
\noindent 

\noindent
The conditional bridge (iv) %
shows that the conditional expectation $\E\!\left[\mathbf p_t \mid \F_t\right]$ in the rough PMP is precisely the It\^o adjoint $p_t$ when both frameworks apply, %
although the It\^o PMP defines an adapted adjoint pair $(p,q)$ through FBSDEs, %
and the rough PMP defines the adjoint $\mathbf p$ pathwise and depends on future information.  
As will be visible from the proof, the remaining differences between the Hamiltonians $H$ and $\mathbf H$ are due to the different chain rules of It\^o calculus and Stratonovich or rough path theory.

 We prove \pmp under Assumption \ref{assum:pmp}, which intersects standard assumptions for It\^o and rough path PMPs. The smoothness assumptions on the drift $b$ and diffusion $\sigma$ are more restrictive than usual assumptions for the It\^o PMP, due to the use of rough path theory. 
Conversely, the Brownian motion assumption restricts the use of non-martingale processes like fBM due to the use of It\^o calculus. These assumptions are only required to prove a unified PMP with the conditional bridge (iv).

\textit{Deterministic controls}: 
If we instead optimize over deterministic open-loop controls  $u\in L^2([0,T],U)$, then the optimality conditions   correspond to the PMP above with the same adjoint equations (i), the same transversality conditions (ii), and the maximality condition from \cite{Lew2026}
\begin{equation}\label{eq:maximality:openloop}
u_t\in
\mathop{\arg\max}_{v\in U}
\E\left[
\mathbf H(t,x_t,v,\mathbf p_t,\mathfrak p_0)
\right].
\end{equation}
On problems where Assumption \ref{assum:pmp} holds, \PMP{} applies with the same conditional bridge (iv) and $u_t\in
\operatorname*{arg\,max}_{v\in U}
\E\left[
H(t,x_t,v,p_t,q_t,\mathfrak p_0)
\right]$.

\paragraph{Paper outline.} The paper is organized as follows. %
\begin{itemize}[leftmargin=4mm]
\item In Section \ref{sec:background}, we provide background on It\^o SDEs and rough SDEs.
\item In Section \ref{sec:pmp_unified}, we prove the unified \PMP (Theorem \ref{thm:pmp}) using It\^o calculus and rough path theory. 
We consider the It\^o and rough tangent vectors $v$ and $\mathbf v$ that solve the It\^o and rough linearized SDEs
\begin{align*}
\dd v_t
=\nabla b(\scalebox{0.95}{$t,x_t,u_t$})v_t\dd t
+
\nabla\sigma(\scalebox{0.95}{$t,x_t$})v_t\dd B_t
\quad\text{and}\quad
\dd \mathbf{v}_t
=\nabla \mathbf b(\scalebox{0.95}{$t,x_t,u_t$})\mathbf{v}_t\dd t
+
\nabla\sigma(\scalebox{0.95}{$t,x_t$})\mathbf{v}_t\dd\mbB_t,
\end{align*}
that define the same forward tangent process $v=\mathbf v$. 
Then, we pair them with their respective adjoint processes $p$ and $\mathbf p$ via the duality identities
$$
\ip{p_t}{v_t} =  \E\left[\ip{p_T}{v_T}\,\big|\,\F_t\right] 
\quad\text{and}\quad
\ip{\mathbf p_t}{\mathbf v_t} =  \ip{\mathbf p_T}{\mathbf v_T}.
$$
Since the tangent processes $v$ and $\mathbf v$ are the same, we deduce the conditional bridge 
$$p_t =  \E\left[\mathbf p_t\,\big|\,\F_t\right],$$ which is the identity (iv) of \PMP. 
The remaining identities follow from the classical It\^o PMP. 
\item In Section \ref{sec:pmp_rough}, we prove the  rough stochastic PMP (Theorem \ref{thm:pmp:rough}) without using It\^o calculus. This result applies under weaker assumptions than Assumption \ref{assum:pmp}, and allows handling more general stochastic processes that cannot be handled by It\^o calculus like fBM. It also shows how to extend the rough stochastic PMP for deterministic open-loop controls \cite{Lew2026} to the setting with non-anticipative controls via stochastic needle variations, and where the proof differs from the standard It\^o PMP.

\item In Section \ref{sec:applications}, we provide  applications. 
\begin{enumerate}[leftmargin=5mm]
\item In Section \ref{sec:applications:example},   an example shows the difference between adapted and anticipative solutions.
\item In Section \ref{sec:applications:adjoint_matching}, we rederive the adjoint matching method \cite{Domingo2025} for generative modeling. 
These derivations, inspired by \cite{Domingo2026}, show that adjoint matching is a clever application of the method of successive approximations \cite{Chernousko1982} using \PMP{} that allows replacing  $\E[\mathbf p_t\,|\,\F_t]$ with $\mathbf p_t$.
\item In Section \ref{sec:applications:shooting}, we explore the design of a shooting method for solving feedback problems. The method resembles the method of \cite{Lew2026} for optimizing over parameterized feedback controls, but   empirically shows improved robustness thanks to a simpler adjoint equation.
\end{enumerate}
\item Appendix \ref{sec:appendix:additional_proofs}, we provide additional details about It\^o-Stratonovich  conversions.  In Appendix \ref{sec:pmp_ito}, we give a standard proof of the It\^o stochastic PMP, for completeness and    to highlight differences with the rough path approach. In Appendix \ref{sec:proofs:rough_paths}, we provide additional proofs for rough SDEs.
\end{itemize}

\paragraph{Outlook.} The conditional bridge of \PMP{} connects FBSDE optimality conditions from It\^o calculus and pathwise optimality conditions from rough path theory, raising interesting questions. 
First, future work could explore how to  approximate  the conditional expectation $\E[\mathbf p_t\mid \F_t]$ (as done initially in Section \ref{sec:applications:shooting}), or find other methods and problems that do not require evaluating it (as in adjoint matching, see Section \ref{sec:applications:adjoint_matching}), leading to new algorithms and applications. 
Second, deriving guarantees for the proposed shooting method (e.g., asymptotic convergence as the sample size increases) remains an open problem. 
Third,  extending the rough PMP to the case where
the diffusion $\sigma$ depends on the control remains challenging, as it can lead to degenerate formulations with irregular
controls \cite{Diehl2016,Allan2020} and rough path theory relies on smooth-in-time coefficients.

\paragraph{Acknowledgments.} We thank Riccardo Bonalli for his helpful feedback on the manuscript.

\paragraph{Use of Large Language Models.}
LLMs were used to review grammar and proofs, help extend a proof of the bridge for the case with state-independent diffusion $\sigma(t)$ to the case with diffusion $\sigma(t,x)$, 
find the counter-example in Remark \ref{remark:rough_adjoint_integrability}, 
refine the definition of Lebesgue times, 
sketch the argument in Remark \ref{remark:pmp:rough:localization}, 
iterate on the example in Section \ref{sec:applications:example}, and 
help extend code from the open-loop shooting method \cite{Lew2026}.
The author wrote every statement and takes full responsibility for the paper and code.

\section{Background}\label{sec:background}
We  summarize useful results in It\^o calculus \cite{Yong1999,LeGall2016} and rough path theory \cite{Allan2021,Friz2010,Friz2020}.

\textbf{Notations}. %
Given $a,b\in\R^n$, the inner product $\ip{a}{b}=a^\top b:=\sum_{i=1}^na_ib_i$. 
The norm $\|\cdot\|$ denotes the Euclidean norm for vectors and the corresponding operator norm for linear maps. The Kronecker product is denoted by $\otimes$. %

Given a continuously differentiable map $F:\R^n\to\R^N$, we denote by $\nabla_xF(x)\in\R^{N\times n}$ its Jacobian, with 
$
[\nabla_xF(x)]_{ij}
=
\frac{\partial F_i}{\partial x_j}(x)
$. 
Given $\sigma(t,x)=\left(
\sigma^1(t,x),\cdots,\sigma^d(t,x)
\right)\in\R^{n\times d}$, for $v,p\in\R^n$, we define
\begin{align}
\nabla_x\sigma(t,x)v
&:=
\big(
\nabla_x\sigma^1(t,x)v,\cdots,\nabla_x\sigma^d(t,x)v
\big)
\in\R^{n\times d},
\\
\nabla_x\sigma(t,x)^\top p
&:=
\big(
\nabla_x\sigma^1(t,x)^\top p,\cdots,\nabla_x\sigma^d(t,x)^\top p
\big)
\in\R^{n\times d},
\\
\nabla_x\sigma(t,x)\sigma(t,x) 
&:=
\Bigg(
\sum_{j=1}^d
\sum_{k=1}^n
\frac{
\partial \sigma^{1j}}{\partial x^k}\sigma^{kj}
,\cdots,
\sum_{j=1}^d
\sum_{k=1}^n
\frac{
\partial \sigma^{nj}}{\partial x^k}\sigma^{kj}
\Bigg)
\in\R^n,
\label{eq:nablasigmasigma}
\end{align}
where the last equation is used to concisely define the rough drift $\mathbf b:=b-\tfrac12\nabla\sigma\sigma$ in \eqref{eq:rough-drift}. 
The space $C(E,F)$ consists of the continuous maps from $E$ to $F$. The space $C_b^k(E,F)$ consists of the $k$-times continuously differentiable maps from $E$ to $F$ that are bounded and whose derivatives up to order $k$ are bounded, with corresponding norm $\|f\|_{C_b^n}:=\|f\|_\infty+\|\nabla f\|_\infty+\dots+\|\nabla^n f\|_\infty<\infty$.

Let $T>0$. For any interval $I=[s,t]\subseteq[0,T]$, we write $|I|:=|t-s|$ and $I^2:=I\times I$. Given a path $X:[0,T]\to\R^n$, its increments are denoted by $X_{s,t}:=X_t-X_s$ for any $s,t\in[0,T]$. 

\textbf{Stochastic processes}. Throughout the paper,  $(\Omega,\F,(\F_t)_{0\leq t\leq T},\Prob)$ denotes a filtered probability space whose filtration $(\F_t)_{0\leq t\leq T}$ is complete and
right-continuous. 
For $\ell\geq1$, we denote by
\begin{itemize}[leftmargin=4mm]
\item $L^\ell(\Omega,\R^n)$ the space of measurable maps $\xi:\Omega\to\R^n$ (called random variables) satisfying $\E[\|\xi\|^\ell]{<}\infty$,
and $L_{\F_t}^\ell(\Omega,\R^n)$ the space of random variables $\xi\in L^\ell(\Omega,\R^n)$ that are $\F_t$-measurable,
\item $L^\ell(\Omega,C([0,T],\R^n))$ the space of measurable maps
$X:\Omega\to C([0,T],\R^n)$ (called stochastic processes) satisfying
$
\|X\|_{L^\ell(\Omega,C([0,T],\R^n))}:=\E[\sup_{0\leq t\leq T}\|X_t\|^\ell]^{1/\ell}<\infty,
$  %
\item $L^\ell_\F(\Omega,C([0,T],\R^n))$ the space of adapted processes in
$L^\ell(\Omega,C([0,T],\R^n))$, which are progressively measurable,
\item $L^\ell_\F([0,T]\times\Omega,\R^n)$ the space of progressively measurable processes $X$ satisfying
$
\E\big[\int_0^T\|X_t\|^\ell\dd t\big]<\infty
$.
\end{itemize}
Given a Borel set $U\subseteq\R^m$, the spaces of $U$-valued processes like $L^\ell_\F([0,T]\times\Omega,U)$  are defined similarly. 

\textbf{Conditional expectation}. 
Given a random variable $X\in L^1(\Omega,\R^n)$, the conditional expectation $\E[X\mid\F_t]$ is defined componentwise. If $X\in L^2(\Omega,\R^n)$ and
$Z\in L^2_{\F_t}(\Omega,\R^n)$, the conditional expectation satisfies $\E[\ip{X}{Z}]
=
\E\left[\ip{\E[X\mid\F_t]}{Z}\right]$. 
An $\F$-adapted stochastic process $X$ satisfying
$X_t\in L^1(\Omega,\R^n)$ for all $t\in[0,T]$ is called a martingale if $$\E\left[X_t \,|\, \F_s\right]=X_s$$ for all $0\leq s\leq t\leq T$ almost surely.

\subsection{It\^o and Stratonovich SDEs} 
Let $(\Omega,\F,(\F_t)_{0\le t\le T},\Prob)$ be a filtered probability space with a $d$-dimensional Brownian
motion $B=(B^1,\dots,B^d)$, where 
$(\F_t)_{0\leq t\leq T}$ is the usual augmentation of the filtration
generated jointly by $B$ and by an initial sigma algebra %
independent of $B$. %

Let $\ell\geq2$ and $X=(X^1,\dots,X^d)\in L^\ell_\F([0,T]\times\Omega,\R^{n\times d})$. 
The It\^o integral of $X$ against $B$ is the stochastic process $y$ denoted as
\begin{equation}
y_t=\int_0^tX_s\dd B_s=\sum_{i=1}^d
\int_0^tX_s^i\dd B^i_s.
\end{equation}
The It\^o integral satisfies\footnote{By the Burkholder-Davis-Gundy inequality, $
\E[\sup_{0\leq t\leq T}\|y_t\|^\ell]
\leq
C_\ell
\E[(
\int_0^T\|X_t\|^2\dd t)^{\ell/2}]
\leq
C_{\ell,T}\E[\int_0^T\|X_t\|^\ell\dd t]<\infty$.} $y\in L^\ell_\F(\Omega,C([0,T],\R^n))$ and is a martingale. 

Conversely, by the martingale representation theorem %
\cite[Theorem 5.18]{LeGall2016}, any martingale $y\in L^\ell_\F([0,T]\times\Omega,\R^n)$ 
can be represented as an It\^o integral
\begin{equation}\label{eq:martingale_representation}
y_t=y_0+\int_0^t Q_s\dd B_s
\end{equation}
for some process $Q\in L^2_\F(\Omega\,{\times}\,[0,T],\R^{n\times d})$. 
This representation is used to derive the It\^o adjoint FBSDE. %
Given initial conditions $\bar y_0\in L^2_{\F_0}(\Omega,\R^n)$ and coefficients $(b,\sigma)$, we say that a stochastic process %
$y$ solves the It\^o SDE
\begin{equation*}
\dd y_t
=
b(t,y_t)\dd t
+
\sigma(t,y_t)\dd B_t,
\qquad
y_0=\bar y_0,
\end{equation*}
if $y_t=\bar y_0+\int_0^tb(s,y_s)\dd s+\int_0^t\sigma(s,y_s)\dd B_s$ for all $t\in[0,T]$ almost surely.

The assumption below is standard to prove the It\^o  PMP (see e.g.  \cite{BonalliLewESAIM2022}).
\begin{assumption}[Assumptions for It\^o PMP]
\label{assum:pmp:ito}
Let 
$T>0$, 
$U\subseteq\R^m$,    
$(\Omega,\F,(\F_t)_{0\le t\le T},\Prob)$ be a filtered probability space and $B$ be a Brownian
motion as defined above.
\begin{itemize}[leftmargin=6mm]
\item The initial conditions $\bar x_0$ satisfy $\bar x_0\in L^2_{\F_0}(\Omega,\R^n)$.
\item The drift  $b:[0,T]\times\R^n\times U\to\R^n$ and diffusion $\sigma:[0,T]\times\R^n\to\R^{n\times d}$ satisfy, for some $C>0$:
\begin{enumerate}[label=(\roman*),leftmargin=6mm]
\item 
$b(\cdot, x, u):[0,T]\to\R^n$  and $\sigma(\cdot,x):[0,T]\to\R^{n\times d}$ are  measurable for all $(x,u)\in\R^n\times U$,  
\item  
$b(t,\cdot,\cdot):\R^n\times U\to\R^n$   and $\sigma(t,\cdot):\R^n\to\R^{n\times d}$
are continuous for almost every $t\in[0,T]$,
\item $\|b(t,0,u)\|+\|\sigma(t,0)\|\leq k(t)$ for some map $k\in L^2([0,T],\R_+)$  and  all $u\in U$.
\item 
$b(t,\cdot,u):\R^n\to\R^n$ and $\sigma(t,\cdot):\R^n\to\R^{n\times d}$ are continuously differentiable  for almost every $t\in[0,T]$  and  all $u\in U$,  
\item 
$\big\|\nablaof{x}b(t,x,u)\big\|+\big\|\nablaof{x}\sigma(t,x)\big\| \leq C$
and 
$\big\|\nablaof{x}b(t,x,u)-\nablaof{x}b(t,\tilde{x},u)\big\|
+
\big\|\nablaof{x}\sigma(t,x)-\nablaof{x}\sigma(t,\tilde{x})\big\|
\leq C\|x-\tilde{x}\|$  
for almost every $t\in[0,T]$, all $x,\tilde{x}\in\R^n$, and all $u\in U$,
\item %
$\|b(t,x,u)-b(t,x,\tilde{u})\|\leq C\|u-\tilde{u}\|$ for almost every $t\in[0,T]$, all $x\in\R^n$, and all $u,\tilde{u}\in U$. 
\end{enumerate} 
\end{itemize}
\end{assumption} 

In particular, Assumption \ref{assum:pmp:ito} ensures the existence and uniqueness of solutions to It\^o SDEs. 
\begin{theorem}[It\^o SDEs 
{(\hspace{-1pt}\cite[Chap. 1, Thm. 6.3]{Yong1999} 
\cite[Chap. 1, Thm. 6.14]{Yong1999}  
\cite[Chap. 7, Thm. 2.2]{Yong1999})}]
\label{thm:sdes:ito}

Under Assumption \ref{assum:pmp:ito}, let $u\in L^2_\F([0,T]\times\Omega,U)$.

\textit{(a) Nonlinear SDEs}:
The SDE
\begin{equation}
\tag{\ref{eq:sde:ito}}
\dd x_t
=
b(t,x_t,u_t)\dd t
+
\sigma(t,x_t)\dd B_t,
\qquad
x_0=\bar x_0,
\end{equation}
has a unique solution (up to indistinguishability), and
$x\in L^2_\F(\Omega,C([0,T],\R^n))$.

\textit{(b) Linear SDEs}:
Let $A\in L^\infty_\F([0,T]\times\Omega,\R^{n\times n})$, $\Sigma\in L^\infty_\F([0,T]\times\Omega,\R^{n\times d\times n})$, 
$t_1\in[0,T]$,  
$\bar v_{t_1}\in L^2_{\F_{t_1}}(\Omega,\R^n)$. 
The SDE\footnote{With $\Sigma=(\Sigma^1,\dots,\Sigma^d)$ where $\Sigma^1,\dots,\Sigma^d\in L^\infty_\F([0,T]\times\Omega,\R^{n\times n})$, the It\^o integral term is $\Sigma_tv_t\dd B_t=\sum_{i=1}^d\Sigma^i_tv_t\dd B^i_t$.}
\begin{equation}
\label{eq:sde:ito_linear}
\dd v_t
=
A_tv_t\dd t
+
\Sigma_tv_t\dd B_t,
\qquad
v_{t_1}=\bar v_{t_1},
\end{equation}
has a unique solution, and
$v\in L^2_\F(\Omega,C([t_1,T],\R^n))$. Moreover, 
$
v_t = 
\phi_t\psi_{t_1}\bar v_{t_1}
$ for any $t\in[t_1,T]$,
where $\phi,\psi\in L^\ell_\F(\Omega,C([0,T],\R^{n\times n}))$ for any $\ell\in [2,+\infty)$ with $\phi(t)=\psi^{-1}(t)$  solve the matrix-valued SDEs
\begin{equation}\label{eq:ito_sde:matrix}
\dd \phi_t=A_t\phi_t\dd t+\Sigma_t\phi_t\dd B_t, 
\ \  
\phi_0=I,
\qquad\quad
\dd \psi_t=-\psi_t(A_t-\Sigma_t^2)\dd t-\psi_t\Sigma_t\dd B_t,
\ \ 
\psi_0=I.
\end{equation}

\textit{(c) Linear BSDEs}: 
Let $\bar p_T\in L^2_{\F_T}(\Omega,\R^n)$ and $r\in L^2_\F([0,T]\times\Omega,\R^n)$. 
The backward SDE (BSDE)\footnote{With $q=(q^1,\dots,q^d)$ where $q^1,\dots,q^d\in L^2_\F([0,T]\times\Omega,\R^n)$, the drift term contains $\Sigma_t^\top q_t=\sum_{j=1}^d \Sigma_t^j q_t^j$ and the It\^o integral term is $q_t\dd B_t=\sum_{i=1}^dq^i_t\dd B^i_t$.}
\begin{align}
\label{eq:sde:ito_bsde}
\dd p_t
&=
-\left(
A_t^\top p_t
+r_t
+
\Sigma_t^\top q_t
\right)\dd t
+
q_t\dd B_t,
\qquad
p_T=\bar p_T
\end{align}
has a unique solution $(p,q)\in  L^2_\F(\Omega,C([0,T],\R^n)) \times  L^2_\F([0,T]\times\Omega,\R^{n\times d})$.
\end{theorem}

\textbf{It\^o's formula}.
Let $x$ and $y$ solve the It\^o SDEs 
$
x_t=x_0+\int_0^ta_s\dd s+\int_0^t\sum_{j=1}^d\alpha_s^j\dd B_s^j
$ and $
y_t=y_0+\int_0^tb_s\dd s+\sum_{j=1}^d\beta_s^j\dd B_s^j
$. Then, It\^o's formula \cite[Theorem 5.10]{LeGall2016} states that 
\begin{equation}
\label{eq:ito_formula}
\ip{x_t}{y_t}
=
\ip{x_0}{y_0}
+
\int_0^t
\bigg(
\ip{a_s}{y_s}
+
\ip{x_s}{b_s}
+
\sum_{j=1}^d\ip{\alpha_s^j}{\beta_s^j}
\bigg)\dd s
+
\int_0^t
\sum_{j=1}^d
\left(
\ip{\alpha_s^j}{y_s}
+
\ip{x_s}{\beta_s^j}
\right)\dd B_s^j.
\end{equation}

\textbf{It\^o to Stratonovich conversion}. 
Let $\ell\geq 2$ and $X=(X^1,\ldots,X^d)\in L^\ell_\F([0,T]\times\Omega,\R^{n\times d})$ be a continuous semimartingale \cite[Definition 4.19]{LeGall2016}. The Stratonovich
integral of $X$ against $B$ is defined by
\begin{equation*}
y_t
=
\int_0^tX_s\circ\dd B_s
:=
\sum_{j=1}^d
\left(
\int_0^tX_s^j\dd B_s^j
+\frac12[X^j,B^j]_t
\right),
\end{equation*}
where the quadratic covariation $[X^j,B^j]_t:=\lim_{|\pi|\to 0}\sum_{[u,v]\in\pi}X^j_{u,v}B^j_{u,v}$ %
is a well-defined limit in probability over partitions $\pi$ of $[0,t]$.  
This integral is generally not a martingale.  
Given initial conditions $\bar y_0\in L^2_{\F_0}(\Omega,\R^n)$ and coefficients $(b,\sigma)$, we say that a stochastic process %
$y$ solves the Stratonovich SDE
\begin{equation}
\label{eq:sde:background_stratonovich}
\dd y_t
=
b(t,y_t)\dd t
+
\sigma(t,y_t)\circ\dd B_t,
\qquad
y_0=\bar y_0,
\end{equation}
if $y_t=\bar y_0+\int_0^tb(s,y_s)\dd s+\int_0^t\sigma(s,y_s)\circ\dd B_s$ for all $t\in[0,T]$ almost surely. 
The following standard result connects solutions to It\^o and Stratonovich SDEs. Its proof is in Appendix \ref{sec:ito_stratonovich_conversion:details} for completeness.

\begin{proposition}[It\^o-Stratonovich conversion]
\label{prop:ito_stratonovich_conversion}
Under Assumption \ref{assum:pmp:ito}, let $u\in L^2_\F([0,T]\times\Omega,U)$. 
Assume also that
$\sigma\in C_b^2([0,T]\times\R^n,\R^{n\times d})$. Define  the Stratonovich drift $\mathbf b:[0,T]\times\R^n\times U\to\R^n$ by 
\begin{equation}
\tag{\ref{eq:rough-drift}}
\mathbf b(t,x,u)
:=
b(t,x,u)
-
\tfrac12 \nabla_x\sigma(t,x)\sigma(t,x),
\end{equation}
with 
$\nabla_x\sigma(t,x)\sigma(t,x)$ as in \eqref{eq:nablasigmasigma}.
 Then, $x\in L^2_\F(\Omega,C([0,T],\R^n))$  solves the It\^o SDE 
\begin{equation}
\tag{\ref{eq:sde:ito}}
\dd x_t
=
b(t,x_t,u_t)\dd t
+
\sigma(t,x_t)\dd B_t,
\qquad\ \ \ 
x_0=\bar x_0,
\end{equation}
if and only if it solves the Stratonovich SDE
\begin{equation}
\label{eq:background:stratonovich_sde}
\dd x_t
=
\mathbf b(t,x_t,u_t)\dd t
+
\sigma(t,x_t)\circ\dd B_t,
\qquad
x_0=\bar x_0.
\end{equation}
\end{proposition}

\subsection{Rough SDEs}\label{sec:background:rough_sdes}
We recall concepts in rough path theory \cite{Allan2021,Friz2010,Friz2020}, following \cite{Lew2026}.  Unless specified,  $p\in[2,3)$.

\subsubsection{Rough paths and rough integration}

\textbf{$p$-variations}. 
We quantify the regularity of a path via the notion of $p$-variation. 
For example, the sample paths $B(\omega)$ of Brownian motion have finite $p$-variation for $p>2$ almost surely. 
Given $p\geq 1$, the $p$-variation  of a path $X:[0,T]\to\R^n$  is defined as
$$
\|X\|_p
:=
\|X\|_{p,[0,T]},
\ \ \text{where }\ \,
\|X\|_{p,[s,t]}
:=
\bigg(
\sup_{\pi\in\mathcal{P}([s,t])}\sum_{[u,v]\in\pi}\|X_{u,v}\|^p
\bigg)^{\frac{1}{p}}
\  
\text{for any $[s,t]\subseteq[0,T]$},
$$
where  $\P([s,t])$ denotes the set of all partitions of $[s,t]$. Given a %
map $\bX:[0,T]^2\to\R^n$, let
$$
\|\bX\|_{\frac{p}{2}}:=\|\bX\|_{\frac{p}{2},[0,T]},
\ \ \text{where }\ \,
\|\bX\|_{\frac{p}{2},[s,t]}:=\bigg(\sup_{\pi\in\mathcal{P}([s,t])}\sum_{[u,v]\in\pi}
\|\bX_{u,v}\|^{\frac{p}{2}}\bigg)^{\frac{2}{p}}
\  
\text{for any $[s,t]\subseteq[0,T]$}.
$$
We denote by $\C^p=\C^p([0,T],\R^n)$ the space of continuous paths $X:[0,T]\to\R^n$ of finite $p$-variation. 
\textbf{Geometric rough paths}.
Let $p\in[2,3)$. A \textit{geometric $p$-rough path} (or simply \textit{rough path}) is a pair $\mbX=(X,\bX)$ with a path $X:[0,T]\to\R^d$ and its enhancement $\bX:[0,T]^2\to\R^{d\times d}$ that satisfy Chen's relation and the integration by parts identity
\begin{equation}
\label{eq:chen's_relation}
\bX_{s,t}=\bX_{s,r}+\bX_{r,t}+X_{s,r}\otimes X_{r,t}
\qquad \ \text{and}\qquad \ 
\bX_{s,t}+\bX^\top_{s,t}=X_{s,t}\otimes X_{s,t}
\end{equation}
for all $s,r,t\in[0,T]$, 
and that has finite inhomogeneous $p$-variation rough path norm:
$$
\|\mbX\|_p:=\|X\|_p+\|\bX\|_{\frac{p}{2}}<\infty.
$$  
We denote by  $\sC^p_g=\sC^p_g([0,T],\R^d)$ the space of geometric rough paths. We write $\|\mbX\|_{p,[s,t]}:=\|X\|_{p,[s,t]}+\|\bX\|_{\frac{p}{2},[s,t]}$ for any $[s,t]\subseteq[0,T]$.
\textbf{Controlled rough paths}. 
Let %
$\mbX=(X,\bX)\in\sC_g^p([0,T],\R^d)$. A \textit{controlled rough path} (with respect to $X$) is a pair
$$
(Y,Y')\in\C^p([0,T],\R^n)\times\C^p([0,T],\R^{n\times d}),
$$
such that the remainder term $R^Y:[0,T]^2\to\R^n,(s,t)\mapsto R^Y_{s,t}:=Y_{s,t}-Y'_sX_{s,t}$ 
satisfies $\|R^Y\|_{\frac{p}{2}}<\infty$.  The path $Y'$ is called the Gubinelli derivative of $Y$.

We denote by $\sD^p_X=\sD^p_X([0,T],\R^n)$ the set of controlled rough paths with respect to $X$. 

\begin{proposition}[Rough integration {\cite[Proposition 2.6]{Friz2018}}]
\label{prop:rough_integral_welldefined:error_bound}
Let %
$\mbX=(X,\bX)\in\sC_g^p([0,T],\R^d)$ 
and $(Y,Y')\in\sD^p_X([0,T],\R^{n\times d})$. 
Then, for every $0\leq s\leq t\leq T$, the limit\footnote{The limit is independent of the choice of partition of $[s,t]$ with vanishing mesh size, and we identify $\R^{n\times d\times d}$ with the space of linear maps from $\R^{d\times d}$ to $\R^n$ %
to make sense of the last term $Y'_s\bX_{s,t}$.} %
\begin{equation}
\label{eq:rough_int}
\int_s^tY_r\dd\mbX_r
:=
\lim_{|\pi|\to0}
\sum_{[u,v]\in\pi}
\left(
Y_uX_{u,v}+Y'_u\bX_{u,v}
\right)
\end{equation}
exists, and is called the rough integral of $(Y,Y')$ against $\mbX$. 
Moreover,  
$\|\int_s^tY_r\dd\mbX_r-Y_sX_{s,t}-Y'_s\bX_{s,t}\|\leq C_p(\|R^Y\|_{\frac{p}{2},[s,t]}\|X\|_{p,[s,t]}+\|Y'\|_{p,[s,t]}\|\bX\|_{\frac{p}{2},[s,t]})$  
for a constant $C_p$ that only depends on $p$.
\end{proposition}

\textbf{It\^o's formula for geometric rough paths}. %
We have the following chain rule.

\begin{lemma}[Chain rule for geometric rough paths {\cite[Theorem 7.7]{Friz2020}}]
\label{lem:rough_path:ito_formula}
Let %
$\mbX\in\sC_g^p([0,T],\R^d)$, 
$(Y,Y')\in\sD_X^p([0,T],\R^n)$, 
and $f\in C^3(\R^n,\R)$. 
Assume that
$$
Y_t
=
Y_0
+
\Gamma_t
+
\int_0^tY'_s\dd\mbX_s,
$$
where $\Gamma\in\C^{\frac{p}{2}}([0,T],\R^n)$ and
$
(Y',Y'')
\in
\sD_X^p([0,T],\R^{n\times d})$. 
Then,
\begin{equation}\label{eq:rough_path:ito_formula}
f(Y_t)=f(Y_0)
+\int_0^t\nabla f(Y_u)\dd\Gamma_u
+\int_0^t\nabla f(Y_u)Y'_u\dd\mbX_u
\end{equation}
for all $t\in[0,T]$, where the first integral is a Young integral, and the second is a rough integral.
\end{lemma}

Applying Lemma \ref{lem:rough_path:ito_formula} to the inner product gives the following product rule: If
$$
Y_t=Y_0+\int_0^ta_s\dd s+\int_0^tY'_s\dd\mbX_s
\quad\text{and}\quad
Z_t=Z_0+\int_0^tb_s\dd s+\int_0^tZ'_s\dd\mbX_s,
$$ 
where $a,b\in L^1([0,T],\R^n)$ and the controlled paths satisfy the assumptions of Lemma \ref{lem:rough_path:ito_formula}, then
\begin{align}
\label{eq:rough_product_rule}
\ip{Y_t}{Z_t}
&=
\ip{Y_0}{Z_0}
+
\int_0^t
\left(
\ip{a_s}{Z_s}
+
\ip{Y_s}{b_s}
\right)\dd s +
\int_0^t
\left(
Z_s^\top Y'_s
+
Y_s^\top Z'_s
\right)\dd\mbX_s.
\end{align}
Equivalently, $
\dd\ip{Y_t}{Z_t}
=
\ip{Y_t}{\dd Z_t}
+
\ip{Z_t}{\dd Y_t}$. 

\subsubsection{Rough differential equations (RDEs)}

\begin{assumption}[RDE drift]
\label{assum:rde_drift}
Let $U\subseteq\R^m$ and $b:[0,T]\times\R^n\times U\to\R^n$ satisfy:
\begin{enumerate}[label=(\roman*),leftmargin=6mm]
\item $b(\cdot,x,u):[0,T]\to\R^n$ is measurable for every $(x,u)\in\R^n\times U$;
\item $b(t,\cdot,\cdot):\R^n\times U\to\R^n$ is continuous for almost every $t\in[0,T]$;
\item  
$
\|b(t,x,u)\|\leq C$ and $
\|b(t,x,u)-b(t,\tilde x,u)\|
\leq
C\|x-\tilde x\|
$ 
for some constant $C > 0$, 
almost every $t\in[0,T]$, all $x,\tilde x\in\R^n$, and all $u\in U$.
\end{enumerate}
\end{assumption}

\begin{theorem}[Nonlinear RDEs {\cite[Theorem 3.9]{Lew2026}}]
\label{thm:rdes:existence_uniqueness}
Let %
$u\in L^1([0,T],U)$, 
 $b$ satisfy Assumption \ref{assum:rde_drift}, 
$\sigma\in C_b^3([0,T]\times\R^n,\R^{n\times d})$, $\bar y\in\R^n$, and 
$\mbX\in\sC_g^p([0,T],\R^d)$. 
Then, there exists a unique
$(Y,Y')\in\sD_X^p([0,T],\R^n)$ with $Y'_t=\sigma(t,Y_t)$ that solves the RDE
\begin{equation}
\label{eq:nonlinear_rde}
Y_t
=
\bar y
+
\int_0^tb(s,Y_s,u_s)\dd s
+
\int_0^t\sigma(s,Y_s)\dd\mbX_s.
\end{equation}
\end{theorem}

If $(\bar y,u,\mbX)$ are random, the random RDE \eqref{eq:nonlinear_rde} is understood pathwise. 
Obtaining moment bounds requires additional integrability assumptions on the (e.g., Gaussian) rough paths, see Section \ref{sec:background:gaussian_paths}.

\begin{theorem}[Linear RDEs {\cite[Theorem 3.10]{Lew2026}}]
\label{thm:rde:linear:existence_uniqueness}
Let %
$\mbX\in\sC_g^p([0,T],\R^d)$, 
 $A\in L^\infty([0,T],\R^{n\times n})$,  
$
(\Sigma,\Sigma')
\in
\sD_X^p\!\left([0,T],\R^{n\times d\times n}\right)$, and 
$\bar v,\bar p\in\R^n$.

Then, there exists a unique
$(V,V')\in\sD_X^p([0,T],\R^n)$ with $V'_t=\Sigma_tV_t$ that solves the RDE\footnote{We identify $\R^{n\times d\times n}$ with the space of linear maps from $\R^n$  to $\R^{n \times d}$
to make sense of the last term $\Sigma_sV_s$, so with $\Sigma_t=(\Sigma_t^1,\dots,\Sigma_t^d)$ where
$\Sigma_t^j\in\R^{n\times n}$, we have
$
\Sigma_tV_t
:=
(\Sigma_t^1V_t,\dots,\Sigma_t^dV_t
)
\in\R^{n\times d}
$.}
\begin{equation}
\label{eq:linear_rde}
V_t
=
\bar v
+
\int_0^tA_sV_s\dd s
+
\int_0^t\Sigma_sV_s\dd\mbX_s.
\end{equation}

Also, there exists a unique
$(P,P')\in\sD_X^p([0,T],\R^n)$ with $P'_t=-\Sigma_t^\top P_t$ that solves the RDE
\begin{equation}
\label{eq:linear_rde_terminal}
P_t
=
\bar p
+
\int_t^T
A_s^\top P_s
\dd s
+
\int_t^T
\Sigma_s^\top P_s\dd\mbX_s.
\end{equation} 
Equivalently, $(P,P')\in\sD_X^p([0,T],\R^n)$ has Gubinelli derivative 
$P'_t=-\Sigma_t^\top P_t$,  
final value $P_T=\bar p$, and satisfies 
$
P_t
=
P_0
-
\int_0^t
A_s^\top P_s
\dd s
-
\int_0^t
\Sigma_s^\top P_s\dd\mbX_s
$.
\end{theorem}
The second part of  Theorem \ref{thm:rde:linear:existence_uniqueness} follows from the first part applied to the time-reversed augmented equation, see \cite[Proposition 5.12]{Friz2020}. The result can be applied to the linearization of nonlinear RDEs  (also called the Jacobian flow), see Theorem \ref{thm:sdes:rough}.

\subsubsection{Gaussian rough paths and rough SDEs}\label{sec:background:gaussian_paths}  

\textbf{Gaussian rough paths}. 
Let $B=(B^1,\dots,B^d)$ be a centered, continuous, $\R^d$-valued Gaussian process with independent components that satisfies the regular covariance  condition \cite[Condition 10]{Bayer2016}
\begin{equation}
\label{eq:gaussian_covariance_regular}
\sup_{\pi,\widetilde\pi\in\mathcal P([s,t])}
\left(
\sum_{[t_j',t_{j+1}']\in\widetilde\pi}
\left(
\sum_{[t_i,t_{i+1}]\in\pi}
\left|
\E\left[B_{t_i,t_{i+1}}^kB_{t_j',t_{j+1}'}^k\right]
\right|
\right)^\rho
\right)^\frac{1}{\rho}
\leq
K|t-s|^\frac{1}{\rho}
\end{equation}
for some $\rho\in[1,3/2)$ and $K<\infty$, for every $k=1,\dots,d$ and $0\leq s<t\leq T$. Then, for  $p\in(2\rho,3)$, $B$ admits a measurable natural lift
$$
\mbB=(B,\bB):
\Omega
\to
\sC_g^p([0,T],\R^d),
$$
unique up to modification, and whose $p$-variation rough path norm $\|\mbB\|_p$ satisfies useful moment bounds  
(see \cite{Friz2013,Bayer2016} and Theorem \ref{thm:gaussian_rough_paths}). We call $\mbB$ an enhanced Gaussian process, and its sample paths $\mbB(\omega)$ are called Gaussian rough paths.

The Stratonovich lift of a Brownian motion $B$ is a particular example of enhanced Gaussian process. In this case, the natural lift is the Stratonovich lift (see
\cite[Theorem 15.33]{Friz2010} or \cite[Theorem 10.4]{Friz2020})
\begin{equation}
\label{eq:brownian_stratonovich_lift}
\bB_{s,t}^{ij}
:=
\int_s^t
(B_r^i-B_s^i)\circ\dd B_r^j,
\qquad
1\leq i,j\leq d.
\end{equation}
However, rough path theory allows handling more general Gaussian rough paths that cannot be handled via It\^o calculus, such as the natural lift of fractional Brownian motion.

\textbf{Gaussian rough SDEs}. 
Solutions to random RDEs, also called rough SDEs, are obtained by composing the random inputs with the It\^o-Lyons solution map:
\begin{gather*}
\hspace{-18mm}
\Omega \longrightarrow
\left(\R^n,L^1([0,T],U),\sC_g^p([0,T],\R^d)\right)
\ \longrightarrow \ 
C([0,T],\R^n)
\\
\omega
\ \longmapsto \
\hspace{9mm}
\hspace{0.5pt}
\big(\bar x_0(\omega),u(\omega),\mbB(\omega)\big)
\hspace{1pt}
\hspace{10mm}
\ \longmapsto\ 
x(\omega)
=
\operatorname{SolveRDE}
\big(\bar x_0(\omega),u(\omega),\mbB(\omega)\big).
\hspace{-2.1cm}
\end{gather*}

The following result gives measurability and moment bounds for this composition.

\begin{theorem}[Rough SDEs]
\label{thm:sdes:rough}
Let $\rho\in[1,3/2)$, $p\in(2\rho,3)$, 
$\mbB$ be an enhanced Gaussian process as defined above, 
 $\ell\geq1$, 
$U\subseteq\R^m$, 
$u\in L^\ell([0,T]\times\Omega,U)$, 
$\bar x_0\in L^\ell(\Omega,\R^n)$, 
 $b$ satisfy Assumption \ref{assum:rde_drift} and be Lipschitz in the control $u$, 
and $\sigma\in C_b^3([0,T]\times\R^n,\R^{n\times d})$.

\textit{(a) Rough nonlinear SDEs}: For almost every $\omega\in\Omega$, let $(x(\omega),x'(\omega))\in\sD_{B(\omega)}^p([0,T],\R^n)$ with $x'_t(\omega)=\sigma(t,x_t(\omega))$ be the solution to the RDE
\begin{equation}
\label{eq:random_gaussian_rde}
x_t(\omega)
=
\bar x_0(\omega)
+
\int_0^t
b(s,x_s(\omega),u_s(\omega))\dd s
+
\int_0^t
\sigma(s,x_s(\omega))\dd\mbB_s(\omega).
\end{equation}
Then, $x\in L^\ell(\Omega,C([0,T],\R^n))$. 

If $\bar x_0\in L^\ell_{\F_0}$, $u$ is progressively measurable, and $\mbB|_{[0,t]}:\Omega\to\sC^p_g([0,t],\R^d)$ is $\F_t$-measurable  for every $t\in[0,T]$, then $x$ is also $\F$-adapted. %

\textit{(b) Rough linearized SDEs}: Assume also that for some $C>0$:
\begin{enumerate}[label=(\roman*),leftmargin=6mm]
\item
$b(t,\cdot,u):\R^n\to\R^n$ is continuously differentiable
for almost every $t\in[0,T]$ and every $u\in U$,
\item
$\nabla_xb:[0,T]\times\R^n\times U\to\R^{n\times n}$
is a Carath\'eodory map,
\item $\|\nabla_xb(t,x,u)\|\leq C$ for almost every $t\in[0,T]$, all $x\in\R^n$, and all $u\in U$,
\item $\|\nabla_xb(t,x,u)-\nabla_xb(t,\tilde x,u)\|
\leq
C\|x-\tilde x\|
$ 
for almost every $t\in[0,T]$, all $x,\tilde x\in\R^n$, and all
$u\in U$,
\item $\sigma\in C_b^4([0,T]\times\R^n,\R^{n\times d})$.
\end{enumerate}
Let $x$ solve the rough SDE %
\eqref{eq:random_gaussian_rde}, 
and define
$
A_t
:=
\nabla_xb(t,x_t,u_t)$ and $\Sigma_t^j
:=
\nabla_x\sigma^j(t,x_t)$ for $j=1,\dots,d$. 
Let $\ell>1$, $t_1\in[0,T]$,  and 
$\bar v_{t_1},\bar p_T\in L^\ell(\Omega,\R^n)$. 
For almost every $\omega\in\Omega$, let $(v(\omega),v'(\omega)),(p(\omega),p'(\omega))\in\sD_{B(\omega)}^p([t_1,T],\R^n)$ with $v'_t(\omega)=\Sigma_tv_t$ and $p'_t(\omega)=-\Sigma_t^\top p_t$ be the solutions to the RDEs
\begin{align}
\label{eq:tangent_rde_background}
v_t(\omega)
&=
\bar v_{t_1}(\omega)
+
\int_{t_1}^tA_s(\omega)v_s(\omega)\dd s
+
\int_{t_1}^t\Sigma_s(\omega)v_s(\omega)\dd\mbB_s(\omega),
&&\hspace{-10mm}t\in[t_1,T],
\\
\label{eq:adjoint_rde_background}
p_t(\omega)
&=
\bar p_T(\omega)
+
\int_t^T
A_s(\omega)^\top p_s(\omega)
\dd s
+
\int_t^T
\Sigma_s(\omega)^\top p_s(\omega)\dd\mbB_s(\omega),
&&\hspace{-10mm}t\in[t_1,T].
\end{align}
Then, $v,p\in L^{\ell'}(\Omega,C([t_1,T],\R^n))$ for any $\ell'<\ell$. 
If $\bar x_0\in L^\ell_{\F_0}$, $\bar v_{t_1}$ is $\F_{t_1}$-measurable, $u$ is progressively measurable, and $\mbB|_{[0,t]}:\Omega\to\sC^p_g([0,t],\R^d)$ is $\F_t$-measurable  for every $t\in[0,T]$, then $v$ is also $\F$-adapted.
\end{theorem}
This result extends \cite[Theorem 3.22]{Lew2026} to handle stochastic controls, as described in Appendix \ref{sec:proofs:rough_paths}.

\subsubsection{Brownian rough paths and consistency with Stratonovich and It\^o SDEs}

\begin{proposition}[Consistency with Stratonovich and It\^o SDEs]
\label{prop:brownian_rde_consistency}
Let $B$ be a $d$-dimensional Brownian motion, and $\mbB=(B,\bB)$ be its Stratonovich lift \eqref{eq:brownian_stratonovich_lift}. 
Let the assumptions of Theorem \ref{thm:sdes:rough} (a) hold, with $\bar x_0\in L^2_{\F_0}$ and progressively measurable $u$. 
Define
$
\mathbf b(t,x,u)
=
b(t,x,u)
-
\frac12
\nabla_x\sigma(t,x)\sigma(t,x)
$. 
Then, the It\^o, Stratonovich, and rough SDEs
\begin{align}
\hspace{20mm}
\dd x_t
&=
b(t,x_t,u_t)\dd t
+
\sigma(t,x_t)\dd B_t,
&&
\hspace{-10mm}
x_0=\bar x_0,
\hspace{15mm}
\\
\dd x_t
&=
\mathbf b(t,x_t,u_t)\dd t
+
\sigma(t,x_t)\circ\dd B_t,
&&
\hspace{-10mm}
x_0=\bar x_0,
\\
\dd x_t
&=
\mathbf b(t,x_t,u_t)\dd t
+
\sigma(t,x_t)\dd\mbB_t,
&&
\hspace{-10mm}
x_0=\bar x_0,
\end{align}
have unique solutions, that are indistinguishable, and satisfy $x\in L^2_\F(\Omega,C([0,T],\R^n))$.
\end{proposition}
\begin{proof}
From  Proposition \ref{prop:ito_stratonovich_conversion}, the It\^o and Stratonovich SDEs have the same solution $x\in L^2_\F$. The rough and Stratonovich SDEs also have the same solutions by \cite[Theorem 9.1]{Friz2020}.
\end{proof}

\begin{corollary}
\label{cor:brownian_tangent_consistency}
Under the assumptions of Proposition \ref{prop:brownian_rde_consistency} and Theorem \ref{thm:sdes:rough} (b), 
let $t_1\in[0,T]$ and $\bar v_1\in L^2_{\F_{t_1}}(\Omega,\R^n)$.
Then, the rough and Stratonovich SDEs
\begin{align*}
\dd v_t
&=
\nabla_x\mathbf b(t,x_t,u_t)v_t\dd t
+
\nabla_x\sigma(t,x_t)v_t\dd\mbB_t,
&&
\hspace{-20mm}
v_{t_1}=\bar v_1,
\\
\dd v_t
&=
\nabla_x\mathbf b(t,x_t,u_t)v_t\dd t
+
\nabla_x\sigma(t,x_t)v_t\circ\dd B_t,
&&
\hspace{-20mm}
v_{t_1}=\bar v_1,
\end{align*}
have unique solutions, that are indistinguishable, and satisfy
$v\in L^2_\F(\Omega,C([t_1,T],\R^n))$.
\end{corollary}
\noindent
We will show in Lemma \ref{lem:tangent:conversion} that the It\^o SDE 
$
\dd v_t
=
\nabla_x b(t,x_t,u_t)v_t\dd t
+
\nabla_x\sigma(t,x_t)v_t\dd B_t
$ with $v_{t_1}=\bar v_1$ also has the same solution.

\section{Proof of the Unified Stochastic PMP}
\label{sec:pmp_unified}

We prove \PMP under  the following assumptions.
\begin{assumption}[Definitions and Assumptions for \ocp and \pmp (Theorem \ref{thm:pmp})]
\label{assum:pmp}
Let 
$T>0$, 
$\rho\in[1,\frac{3}{2})$, $p\in(2\rho,3)$, 
and 
$U\subseteq\R^m$.
\begin{itemize}
\item Let $(\Omega,\F,(\F_t)_{0\le t\le T},\Prob)$ be a filtered probability space with a $d$-dimensional Brownian
motion $B=(B^1,\dots,B^d)$, where $(\F_t)_{0\leq t\leq T}$ is the usual augmentation of the filtration generated by $B$ and by an initial sigma algebra independent of $B$. Let $\mbB=(B,\bB)$ be its Stratonovich lift.  
\item The initial conditions $\bar x_0$ satisfy $\bar x_0\in L^2_{\F_0}(\Omega,\R^n)$.
\item The drift  $b:[0,T]\times\R^n\times U\to\R^n$ and cost  $f:[0,T]\times\R^n\times U\to\R$   
satisfy:
\begin{itemize} 
\item 
$b(\cdot, x, u):[0,T]\to\R^n$  is  measurable for all $(x,u)\in\R^n\times U$,  
\item  
$b(t,\cdot,\cdot):\R^n\times U\to\R^n$ 
is continuous for almost every $t\in[0,T]$,
\item 
$b(t,\cdot,u):\R^n\to\R^n$  is  continuously differentiable  for almost every $t\in[0,T]$  and  all $u\in U$,  
\item 
$\|b(t,x,u)\|+\big\|\nablaof{x}b(t,x,u)\big\| \leq C_{b,f}$ 
and 
$\big\|\nablaof{x}b(t,x,u)-\nablaof{x}b(t,\tilde{x},u)\big\|\leq C_{b,f}\|x-\tilde{x}\|$  
for almost every $t\in[0,T]$, all $x,\tilde{x}\in\R^n$, and all $u\in U$,
\item %
$\|b(t,x,u)-b(t,x,\tilde{u})\|\leq C_{b,f}\|u-\tilde{u}\|$ for almost every $t\in[0,T]$, all $x\in\R^n$, and all $u,\tilde{u}\in U$,
\end{itemize}
and similarly for $f$. 
\item The diffusion $\sigma:[0,T]\times\R^n\to\R^{n\times d}$ satisfies $\sigma\in C^4_b([0,T]\times\R^n,\R^{n\times d})$.
\item The terminal cost $g:\R^n\to\R$ and constraints $h:\R^n\to\R^r$ are continuously differentiable and 
satisfy 
$\big\|\nabla g(x)\big\|+\big\|\nabla h(x)\big\|
\leq C_{g,h}$
and 
$\big\|\nabla g(x)-\nabla g(\tilde{x})\big\|
+
\big\|\nabla h(x)-\nabla h(\tilde{x})\big\|
\leq C_{g,h}\|x-\tilde{x}\|$ for all $x,\tilde{x}\in\R^n$ 
for a constant $C_{g,h}>0$.
\end{itemize}
\end{assumption}
\noindent These assumptions are the intersection of the usual assumptions of the It\^o and rough stochastic PMPs. Specifically, the smoothness assumptions on the drift $b$ and diffusion $\sigma$ are more restrictive than usual assumptions for the It\^o PMP, due to the use of rough path theory. 
Conversely, the Brownian motion assumption restricts the use of non-martingale processes like fBM due to the use of It\^o calculus. 
These assumptions are only required to prove a unified PMP with the conditional bridge (iv), while individual It\^o- or rough path-only PMPs rely on weaker assumptions.

Under Assumption \ref{assum:pmp}, given a non-anticipative control $u\in\U := L^2_{\F}([0,T]\times\Omega,U)$, and 
the drift 
$
\mathbf b(t,x,u)
:=
b(t,x,u)
-
\tfrac12\nabla_x\sigma(t,x)\sigma(t,x),
$ 
the SDEs
\begin{align}
\tag{\ref{eq:sde:ito}}
\textcolor{gray}{\text{(It\^o SDE)}}
\hspace{13mm}
\dd x_t
&=
b(t,x_t,u_t)\dd t
+
\sigma(t,x_t)\dd B_t,
&& 
x_0=\bar x_0,
\\
\hspace{-10mm}\textcolor{gray}{\text{(Stratonovich SDE)}}
\hspace{13mm}
\dd x_t
&=
\mathbf b(t,x_t,u_t)\dd t
+
\sigma(t,x_t)\circ\dd B_t,
&& 
x_0=\bar x_0,
\nonumber
\\
\tag{\ref{eq:sde:rough}}
\textcolor{gray}{\text{(Rough SDE)}}
\hspace{13mm}
\dd x_t
&=
\mathbf b(t,x_t,u_t)\dd t
+
\sigma(t,x_t)\dd\mbB_t,
&& 
x_0=\bar x_0,
\end{align}
have the same unique solution $x\in L^2_\F(\Omega,C([0,T],\R^n))$, by Proposition \ref{prop:brownian_rde_consistency}.

Let $t_1\in[0,T)$, $\bar v_1\in L^2_{\F_{t_1}}(\Omega,\R^n)$, and $\bar p_T\in L^{2+\epsilon}_{\F_T}(\Omega,\R^n)$ for some $\epsilon\geq 0$. 
The linear SDEs
\begin{align}
\label{eq:tangent:ito}
\textcolor{gray}{\text{(It\^o SDE)}}
\hspace{14mm}
\dd v_t
&=\nabla b_tv_t\dd t
+
\nabla\sigma_tv_t\dd B_t,
&&v_{t_1}=\bar v_1,
\\
\label{eq:tangent:rough}
\hspace{2mm}\textcolor{gray}{\text{(Rough SDE)}}
\hspace{13mm}\hspace{1.5pt}
\dd \mathbf{v}_t
&=\nabla \mathbf b_t\mathbf{v}_t\dd t
+
\nabla\sigma_t\mathbf{v}_t\dd \mbB_t,
&&\mathbf{v}_{t_1}=\bar v_1,
\\
\label{eq:adjoint:ito}
\textcolor{gray}{\text{(It\^o BSDE)}}
\hspace{14mm}
\dd p_t
&=
-\left(
\nabla b^\top_t p_t 
+
\nabla\sigma^\top_t q_t
\right)
\dd t
+
q_t\dd B_t,
&&p_T=\bar p_T,
\\
\label{eq:adjoint:rough}
\hspace{2mm}\textcolor{gray}{\text{(Rough SDE)}}
\hspace{13mm}\hspace{1.5pt}
\dd \mathbf p_t
&=
-\nabla \mathbf b^\top_t \mathbf p_t 
\dd t
-\nabla\sigma^\top_t
\mathbf p_t
\dd \mbB_t,
&&\mathbf p_T=\bar p_T,
\end{align}
have unique solutions on $[t_1,T]$ and $[0,T]$, 
thanks to  Theorem \ref{thm:sdes:ito} and Theorem \ref{thm:sdes:rough} (b), 
where 
$$
\left(b_t,\mathbf b_t,\sigma_t,\nabla b_t,\nabla\mathbf b_t,\nabla\sigma_t\right):=\left(b(t,x_t,u_t),\mathbf b(t,x_t,u_t),\sigma(t,x_t),\nabla_xb(t,x_t,u_t),\nabla_x\mathbf b(t,x_t,u_t),\nabla_x\sigma(t,x_t)\right)
$$ 
for conciseness.  

This section is organized as follows.
\begin{itemize}
\item In Section \ref{sec:pmp_unified:v=v}, we prove that  the It\^o and rough tangent equations define the same process: $\mathbf v = v$.
\item In Section \ref{sec:pmp_unified:<p,v>}, we prove the It\^o and rough duality identities
$$
\ip{p_t}{v_t} =  \E\left[\ip{p_T}{v_T}\,\big|\,\F_t\right] 
\quad\text{and}\quad
\ip{\mathbf p_t}{\mathbf v_t} =  \ip{\mathbf p_T}{\mathbf v_T}.
$$
\item In Section \ref{sec:pmp_unified:bridge}, we use the duality identities above to prove the conditional bridge 
\begin{align*}
 \hspace{4cm}
 &p_t =  \E\left[\mathbf p_t\,\big|\,\F_t\right] \quad\text{for all $t\in[t_1,T]$}.
 &&\hspace{1cm}\textcolor{gray}{\text{(condition (iv) of the PMP)}}
\end{align*}
\item In Section \ref{sec:pmp_unified:conclusion}, we conclude the proof of \PMP using this conditional bridge and the It\^o PMP.
\end{itemize}

\subsection{Forward tangent processes}
\label{sec:pmp_unified:v=v}

\begin{lemma}[It\^o-rough forward tangent conversion]\label{lem:tangent:conversion}
Under Assumption \ref{assum:pmp}, let $x\in L_\F^2(\Omega,C([0,T],\R^n))$ be the solution to the It\^o SDE
\begin{equation}\tag{\ref{eq:sde:ito}}
\dd x_t
=
b(t,x_t,u_t)\dd t
+
\sigma(t,x_t)\dd B_t,
\qquad
x_0=\bar x_0.
\end{equation}
Let $\bar v_1\in L^2_{\F_{t_1}}(\Omega,\mathbb R^n)$, and $v,\mathbf v\in L_\F^2(\Omega,C([t_1,T],\R^n))$ be the solutions to the It\^o and
rough SDEs
\begin{align*}
\tag{\eqref{eq:tangent:ito},\eqref{eq:tangent:rough}} 
\dd v_t
&=\nabla b_tv_t\dd t
+
\nabla\sigma_tv_t\dd B_t,
\ 
v_{t_1}=\bar v_1,
\qquad
\dd \mathbf v_t
=\nabla \mathbf b_t\mathbf v_t\dd t
+
\nabla\sigma_t\mathbf v_t\dd \mbB_t,
\ 
\mathbf v_{t_1}=\bar v_1,
\end{align*}
Then, $v=\mathbf v$ almost surely. 
\end{lemma}

\begin{proof} 
By Proposition \ref{prop:ito_stratonovich_conversion} and Corollary \ref{cor:brownian_tangent_consistency}, the state trajectory $x$ and tangent process $\mathbf v$ solve the Stratonovich SDEs
\begin{align*}
\dd x_t
&= 
\mathbf b(t,x_t,u_t)\dd t
+
\sigma(t,x_t)\circ \dd B_t,
&&\hspace{-1cm}
x_0=\bar x_0,
\\
\dd \mathbf v_t
&=
\nabla_x\mathbf b(t,x_t,u_t)\mathbf v_t\dd t
+
\sum_{j=1}^d
\nabla_x\sigma^j(t,x_t)\mathbf v_t\circ \dd B_t^j,
&&\hspace{-1cm}
\mathbf v_{t_1}=\bar v_1.
\end{align*}

Next, we convert this SDE into It\^o form. 
The joint process
$z:=(x,\mathbf v)$ solves the Stratonovich SDE
\begin{align*}
\dd z_t
&=
\begin{bmatrix}
\mathbf b(t,x_t,u_t)
\\
\nabla_x\mathbf b(t,x_t,u_t)\mathbf v_t
\end{bmatrix}
\dd t
+
\sum_{j=1}^d
\underbrace{
\begin{bmatrix}
\sigma^j(t,x_t)
\\
\nabla_x\sigma^j(t,x_t)\mathbf v_t
\end{bmatrix}
}_{=:V^j(t,x_t,\mathbf v_t)}
\circ \dd B_t^j,
\quad
z_{t_1}=\begin{bmatrix}
x_{t_1}
\\
\bar v_1
\end{bmatrix},
\end{align*}
where 
$
V^j(t,x,v)
:=
\begin{bmatrix}
\sigma^j(t,x)\\
\Sigma^j(t,x)v
\end{bmatrix}
\in\R^{2n}$ with $\Sigma^j(t,x):=\nabla_x\sigma^j(t,x)\in\mathbb R^{n\times n}$ 
for $j=1,\dots,d$. 
By Proposition \ref{prop:ito_stratonovich_conversion}, $z$ solves the equivalent It\^o SDE
\begin{align*}
\dd 
\begin{bmatrix}
x_t
\\
\mathbf v_t
\end{bmatrix}
&=
\bigg(
\begin{bmatrix}
\mathbf b(t,x_t,u_t)
\\
\nabla_x\mathbf b(t,x_t,u_t)\mathbf v_t
\end{bmatrix}
+
\frac12 
\sum_{j=1}^d\nabla_{(x,v)}V^j(t,x_t,\mathbf v_t)V^j(t,x_t,\mathbf v_t)
\bigg)
\dd t
+
\sum_{j=1}^d
V^j(t,x_t,\mathbf v_t)
\dd B_t^j.
\end{align*}
Each $j$-th correction term is
\begin{small}
\begin{align*}
\nabla_{(x,v)}V^j(t,x,v)V^j(t,x,v)
&=
\begin{bmatrix}
\Sigma^j(t,x) &0
\\
\nabla_x\left(\Sigma^j(t,x)v\right) & \Sigma^j(t,x)
\end{bmatrix}
\begin{bmatrix}
\sigma^j(t,x)\\
\Sigma^j(t,x)v
\end{bmatrix}
 =
\begin{bmatrix}
\Sigma^j(t,x)\sigma^j(t,x)
\\
\nabla_x\Sigma^j(t,x)\sigma^j(t,x)v + \Sigma^j(t,x)\Sigma^j(t,x)v
\end{bmatrix},
\end{align*}
\end{small}

\noindent where
\begin{equation}\label{eq:tangent:strat_to_ito_conversion}
\big(\nabla_x \Sigma^j(t,x)\sigma^j(t,x)\big)v
+
\Sigma^j(t,x)\Sigma^j(t,x)v
=
\nabla_x\!\left(\Sigma^j(t,x)\sigma^j(t,x)\right)v.
\end{equation}
The identity \eqref{eq:tangent:strat_to_ito_conversion} follows from direct computations, 
see Section \ref{sec:tangent:strat_to_ito_conversion:details}. 
Thus, by Proposition \ref{prop:ito_stratonovich_conversion}, $\mathbf v$ solves the equivalent SDEs
\begin{align*}
\dd \mathbf v_t
&=
\nabla_x\mathbf b(t,x_t,u_t)\mathbf v_t\dd t
+
\sum_{j=1}^d
\nabla_x\sigma^j(t,x_t)\mathbf v_t\circ \dd B_t^j
\\[-1mm]
&=
\bigg(\!
\nabla_x\mathbf b(t,x_t,u_t)\mathbf v_t
+
\frac12
\sum_{j=1}^d
\nabla_x\!\left(\Sigma^j(t,x_t)\sigma^j(t,x_t)\right)\!\mathbf v_t
\!\bigg)
\dd t
+
\sum_{j=1}^d
\nabla_x\sigma^j(t,x_t)\mathbf v_t\dd B_t^j
\\[-1mm]
&=
\nabla_x b(t,x_t,u_t)\mathbf v_t
\dd t
+
\sum_{j=1}^d
\nabla_x\sigma^j(t,x_t)\mathbf v_t\dd B_t^j,
\end{align*}
which is the It\^o SDE solved by $v_t$. 
By uniqueness of solutions,
$\mathbf v=v$ almost surely.
\end{proof}

\subsection{Duality identities}\label{sec:pmp_unified:<p,v>}

\begin{lemma}[It\^o duality identity]
\label{lem:duality:ito}
Let $A\in L^\infty_\F([0,T]\times\Omega,\R^{n\times n})$ ,  
$\Sigma\in L^\infty_\F([0,T]\times\Omega,\R^{n\times d\times n})$,  
$\bar v_1\in L^2_{\F_{t_1}}(\Omega,\R^n)$, and $\bar p_T\in L^2_{\F_T}(\Omega,\R^n)$.   
Let $v\in  L^2_\F(\Omega,C([0,T],\R^n))$ and $(p,q)\in  L^2_\F(\Omega,C([0,T],\R^n)) \times  L^2_\F([0,T]\times\Omega,\R^{n\times d})$ solve the It\^o SDE and BSDE
\begin{align*}
\dd v_t
&=
A_tv_t\dd t
+
\Sigma_tv_t\dd B_t,
\hspace{20mm}
v_{t_1}=\bar v_1,
\\
\dd p_t
&=
-\left(
A_t^\top p_t
+
\Sigma_t^\top q_t
\right)\dd t
+
q_t\dd B_t,
\hspace{4mm}\hspace{1pt}
p_T=\bar p_T.
\end{align*}
Then, for every $t\in[t_1,T]$, almost surely,
$$
\ip{p_t}{v_t}
=
\E\left[
\ip{p_T}{v_T}
\,\big|\,
\mathcal F_t
\right].
$$
\end{lemma}

\begin{proof}

By Theorem \ref{thm:sdes:ito}, the solutions to the SDEs are well-defined. By the Cauchy-Schwarz inequality,
$\ip{p_t}{v_t}\in L_{\F_t}^1(\Omega,\R)$ for each $t\in[0,T]$.
By It\^o's formula \eqref{eq:ito_formula},
\begin{align*}
\ip{p_t}{v_t}
=
\ip{p_{t_2}}{v_{t_2}}
&+
\int_{t_2}^t
\overbrace{
\bigg(
-\ip{A^\top_s p_s+\Sigma^\top_sq_s}{v_s}
+
\ip{p_s}{A_sv_s}
+\sum_{j=1}^d\ip{q_s^j}{\Sigma^j_sv_s}
\bigg)
}^{=0}
\dd s
\\[-2mm]
&+
\int_{t_2}^t
\sum_{j=1}^d
\left(
\ip{q_s^j}{v_s}
+
\ip{p_s}{\Sigma^j_sv_s}
\right)\dd B_s^j,
\end{align*}
for any $t_1\leq t_2\leq t\leq T$. 
Thus, $\ip{p}{v}$ is a martingale, and it follows that
$
\ip{p_t}{v_t} = 
\E\left[\ip{p_T}{v_T}\,\big|\,\F_t\right]
$
for all $t\in[t_1,T]$, which concludes the proof. 
\end{proof}

\begin{corollary}[Duality identity for tangent and adjoint It\^o SDEs]
\label{lem:duality:ito:cost_augmented}

Under Assumption \ref{assum:pmp}, let
\begin{itemize}
\item $x\in L_\F^2(\Omega,C([0,T],\R^n))$ be the solution to the It\^o state SDE \eqref{eq:sde:ito}, 
\item $v\in L_\F^2(\Omega,C([t_1,T],\R^n))$ be the solution to the It\^o tangent SDE \eqref{eq:tangent:ito}, 
\item  $p\in L_\F^2(\Omega,C([0,T],\R^n))$ be the solution to the  It\^o adjoint BSDE \eqref{eq:adjoint:ito}.
\end{itemize}
Then, for every $t\in[t_1,T]$, almost surely,
$$
\ip{p_t}{v_t}
=
\E\left[
\ip{p_T}{v_T}
\,\big|\,
\mathcal F_t
\right].
$$
\end{corollary}

\begin{proof}
Apply Lemma \ref{lem:duality:ito} with 
$A_t=\nabla b_t$ and $
\Sigma_t=\nabla\sigma_t$. 
\end{proof}

\begin{lemma}[Rough duality identity]
\label{lem:duality:rough}
Let $p\in[2,3)$, 
$\mbX\in\sC_g^p([0,T],\R^d)$, 
$A\in L^\infty([0,T],\R^{n\times n})$, 
$
(\Sigma,\Sigma')
\in
\sD_X^p\left([0,T],\R^{n\times d\times n}\right),
$
and $\bar v_1, \bar p_T\in\R^n$. 
Let $\mathbf v$ and $\mathbf p$ solve the RDEs
\begin{align*}
\dd \mathbf v_t
&=
A_t\mathbf v_t\dd t
+
\Sigma_t\mathbf v_t\dd\mbX_t,
\hspace{11mm}
\mathbf v_{t_1}=\bar v_1,
\\
\dd\mathbf p_t
&=
-A_t^\top \mathbf p_t\dd t
-\Sigma^\top_t\mathbf p_t\dd\mbX_t, 
\hspace{6mm}
\mathbf p_T=\bar p_T.
\end{align*}
Then, for every $t\in[t_1,T]$,
$$
\ip{\mathbf p_t}{\mathbf v_t}
=
\ip{\mathbf p_T}{\mathbf v_T}.
$$
\end{lemma}

\begin{proof}
The solutions to the RDEs are well-defined, by Theorem \ref{thm:rde:linear:existence_uniqueness}. 
By the chain rule \eqref{eq:rough_product_rule},
\begin{align*} 
\ip{\mathbf p_t}{\mathbf v_t}
&\mathop{=}^{\eqref{eq:rough_product_rule}}
\ip{\mathbf p_s}{\mathbf v_s}
+
\int_s^t
\left(
\ip{-A^\top_\tau\mathbf p_\tau}{\mathbf v_\tau}
+
\ip{\mathbf p_\tau}{A_\tau\mathbf v_\tau}
\right)\dd \tau +
\int_s^t
\left(
\mathbf p_\tau^\top \Sigma_\tau\mathbf v_\tau
-
\mathbf v_\tau^\top \Sigma^\top_\tau\mathbf p_\tau
\right)\dd\mbX_\tau
=
\ip{\mathbf p_s}{\mathbf v_s}
\end{align*}
for any $s,t\in[t_1,T]$. 
Thus,  $\ip{\mathbf p_t}{\mathbf v_t}$ is constant on $[t_1,T]$, and the claim follows.
\end{proof}

\begin{corollary}[Duality identity for tangent and adjoint rough SDEs]
\label{cor:duality:rough:cost_augmented}
Under Assumption \ref{assum:pmp}, let $\epsilon>0$, $\bar p_T\in L^{2+\epsilon}(\Omega,\R^n)$, and
\begin{itemize}
\item $x\in L_\F^2(\Omega,C([0,T],\R^n))$ be the solution to the It\^o state SDE \eqref{eq:sde:ito}, 
\item $\mathbf v\in L_\F^2(\Omega,C([t_1,T],\R^n))$ be the solution to the rough tangent SDE \eqref{eq:tangent:rough}, 
\item  $\mathbf p\in L^2(\Omega,C([0,T],\R^n))$ be the solution to the  rough adjoint SDE \eqref{eq:adjoint:rough}.
\end{itemize}
Then, for every $t\in[t_1,T]$, almost surely,
$$
\ip{\mathbf p_t}{\mathbf v_t}
=
\ip{\mathbf p_T}{\mathbf v_T}.
$$
\end{corollary}

\begin{proof}
First, $x\in L_\F^2(\Omega,C([0,T],\R^n))$ by Theorem \ref{thm:sdes:ito}, $\mathbf v\in L_\F^2(\Omega,C([t_1,T],\R^n))$ by Corollary \ref{cor:brownian_tangent_consistency}, and $\mathbf p\in L^2(\Omega,C([0,T],\R^n))$ by Theorem \ref{thm:sdes:rough}. The equality $\ip{\mathbf p_t}{\mathbf v_t}
=
\ip{\mathbf p_T}{\mathbf v_T}$ follows by applying Lemma \ref{lem:duality:rough} with 
$A_t=\nabla \mathbf b_t$ and $
\Sigma_t=\nabla\sigma_t$. 
\end{proof}

\subsection{Conditional bridge}\label{sec:pmp_unified:bridge}

\begin{lemma}[It\^o-rough conditional bridge]
\label{lem:ito_rough_bridge}
Under Assumption \ref{assum:pmp}, let $\epsilon>0$, $\bar p_T\in L^{2+\epsilon}_{\F_T}(\Omega,\R^n)$, and
\begin{itemize}
\item $x\in L_\F^2(\Omega,C([0,T],\R^n))$ be the solution to the It\^o state SDE \eqref{eq:sde:ito}, 
\item  $p\in L_\F^2(\Omega,C([0,T],\R^n))$ be the solution to the  It\^o adjoint BSDE \eqref{eq:adjoint:ito},
\item  $\mathbf p\in L^2(\Omega,C([0,T],\R^n))$ be the solution to the  rough adjoint SDE \eqref{eq:adjoint:rough}.
\end{itemize}
Then, %
for $(\dd t\otimes\Prob)$-almost every $(t,\omega)$, 
$$
p_t
=
\E\left[
\mathbf p_t\,\big|\,\mathcal F_t
\right].
$$
\end{lemma}

This identity does not rely on optimality of the control $u$, and applies to any control $u\in\U$. Also, this identity relates the two adjoint processes $p$ and $\mathbf p$, whose equations do not involve the tangent processes $v$ and $\mathbf v$. The tangent processes are only used in the proof: They provide  test directions $v^k$ to compare $p$ and $\mathbf p$. Thus, the bridge relies on the assumption that the It\^o and rough forward tangent equations define the same process $v=\mathbf{v}$, but  $p$ and $\mathbf p$  are defined independently.

\begin{proof}
Let $t_1\in[0,T]$, $e^k=(0,\dots,1,\dots,0)\in\R^n$ for $k\in\{1,\dots,n\}$, and
\begin{itemize}
\item $v^k\in L_\F^2(\Omega,C([t_1,T],\R^n))$ solve the It\^o tangent SDE \eqref{eq:tangent:ito} with initial conditions $\bar v_1=e^k$, 
\item $\mathbf v^k\in L_\F^2(\Omega,C([t_1,T],\R^n))$ solve the rough tangent SDE \eqref{eq:tangent:rough} with initial conditions $\bar v_1=e^k$.
\end{itemize}
By Lemma \ref{lem:tangent:conversion}, $v^k=\mathbf v^k$ almost surely. By the It\^o and rough duality identities,
\begin{align}
\hspace{2cm}\langle p_{t_1},e^k\rangle
&=
\E\left[
\langle \bar p_T,v_T^k\rangle
\,\big|\,
\mathcal F_{t_1}
\right],
&&\hspace{-1.5cm}\textcolor{gray}{\text{(Lemma \ref{lem:duality:ito})}}
\nonumber
\\
\label{eq:p_t1ek=p_Tvk_T}
\langle \mathbf p_{t_1},e^k\rangle
&=
\langle \bar p_T,v_T^k\rangle
&&\hspace{-1.5cm}\textcolor{gray}{\text{(Lemma \ref{lem:duality:rough})}}
\end{align}
almost surely.
Therefore, almost surely,
$$
\langle p_{t_1},e^k\rangle
=
\E\left[
\langle \mathbf p_{t_1},e^k\rangle
\,\big|\,
\mathcal F_{t_1}
\right].
$$
Since this holds for any basis vector $e^k\in\R^n$ and every $t_1\in[0,T]$,
we obtain $
p_t
=
\E\left[
\mathbf p_t\,\big|\,
\mathcal F_t
\right]
$ 
for all $t\in[0,T]$ almost surely. 

To conclude  for $(\dd t\otimes\Prob)$-almost every $(t,\omega)$, we define the set $A=\{(t,\omega)\in[0,T]\times\Omega:p_t(\omega)\neq \E[\mathbf p_t\mid\mathcal F_t](\omega)\}$, where $(t,\omega)\mapsto\E[\mathbf p_t\mid\F_t](\omega)$ denotes an optional (hence measurable) projection of $\mathbf p$ \cite[Chapter 2, Theorem 4.2]{EthierKurtz1986} so that $A$ is measurable, 
 and the set $A_t=\{\omega\in\Omega:(t,\omega)\in A\}$, which satisfies $\Prob(A_t)=0$ for all $t\in[0,T]$ as shown  above. 
 Then, we apply Fubini's theorem to obtain $(\dd t\otimes\Prob)(A)=\int_{[0,T]}\int_\Omega\mathbf{1}_A(t,\omega)\dd\Prob\dd t=\int_{[0,T]}\Prob(A_t)=0$.
\end{proof}

 \begin{remark}[Integrability and anticipativity of the rough adjoint $\mathbf p$]
\label{remark:rough_adjoint_integrability}
Since the forward tangent vector satisfies $\mathbf v=v\in L^2_\F(\Omega,C([t_1,T],\R^n))$ if $\bar v_1\in L^2_{\F_{t_1}}(\Omega,\R^n)$, 
the It\^o-rough bridge could seem to imply that 
\begin{equation}\label{eq:incorrect:prough_L2}
\bar p_T\in L^2_{\F_T}(\Omega,\R^n)
\ \mathop{\implies}^{\text{(Incorrect)}}\ 
\mathbf p\in L^2(\Omega,C([0,T],\R^n)).
\end{equation}
However, the claim  \eqref{eq:incorrect:prough_L2} is incorrect, because the rough adjoint is generally anticipative. 
We do have
\begin{equation}\label{eq:bpTinL2+eps=>bpinL2}
\bar p_T\in L^{2+\epsilon}_{\F_T}(\Omega,\R^n)\text{ for some $\epsilon>0$}
\implies 
\mathbf p\in L^2(\Omega,C([0,T],\R^n)),
\end{equation}
which can be shown using rough path theory  only (see Theorem \ref{thm:sdes:rough}), 
or using the characterization
\begin{equation}\label{eq:bpt_psiphibarpT}
\mathbf p_t=\psi_t^\top \phi_T^\top \bar p_T
\end{equation}
where $\phi,\psi\in L^\ell_\F(\Omega,C([0,T],\R^{n\times n}))$ for any $\ell\in [2,+\infty)$  solve the matrix-valued  It\^o SDEs \eqref{eq:ito_sde:matrix}. The characterization  \eqref{eq:bpt_psiphibarpT} is derived by following the proof of Lemma \ref{lem:ito_rough_bridge} and deducing that
$$
\ip{\mathbf p_{t_1}}{e^k}
\mathop{=}^{\eqref{eq:p_t1ek=p_Tvk_T}}
\ip{\bar p_T}{v_T^k}
\mathop{=}^{\text{(Theorem \ref{thm:sdes:ito})}}
\ip{\bar p_T}{\phi_T\psi_{t_1}e^k}
=
\ip{\psi_{t_1}^\top\phi_T^\top \bar p_T}{e^k}
$$
for any basis vector $e^k=(0,\dots,1,\dots,0)$ and $t_1\in[0,T]$. 
In the proof of \PMP, the boundedness of $\nabla g$ and $\nabla h$
gives $\bar p_T\in L^\infty(\Omega,\R^n)$, so that $\mathbf p\in L^2(\Omega,C([0,T],\R^n))$.

As a counter-example to \eqref{eq:incorrect:prough_L2}, consider the scalar case with $b(t,x,u)=0$, $\sigma(t,x)=\sin(x)$, and $\bar x_0=0$, so that $x_t=0$, $\nabla\sigma(t,x)=1$, 
$\phi_t=\exp(B_t-t/2)=\psi_t^{-1}$, and 
$$
\mathbf p_t
=
\bar p_T\phi_T\phi_t^{-1}
=
\bar p_T
\exp\left(B_T-B_t-(T-t)/2\right).
$$
Take
$
\bar p_T
=
\frac{1}{\sqrt{1+B_T^2}}\exp\left(\frac{B_T^2}{4T}\right)$. Since $B_T\sim \mathcal N(0,T)$, 
\begin{small}
$$
\E\!\left[\bar p_T^2\right]
=
\E\left[
\frac{1}{1+B_T^2}\exp\left(\frac{B_T^2}{2T}\right)
\right]
=
\frac{1}{\sqrt{2\pi T}}\int_\R\frac{1}{1+x^2}
e^{x^2/(2T)}
e^{-x^2/(2T)}\dd x
=
\frac{1}{\sqrt{2\pi T}}\int_\R\frac{1}{1+x^2}\dd x<\infty,
$$
\end{small}
so $\bar p_T\in L^2_{\F_T}(\Omega,\R)$. 
However, 
\begin{small}
$$
\E\!\left[\mathbf p_0^2\right] = 
\E\left[
\left(\frac{1}{1+B_T^2}\exp\left(\frac{B_T^2}{2T}\right)\right)
\exp\left(2B_T-T\right)
\right]
=
\frac{e^{-T}}{\sqrt{2\pi T}}\int_\R\frac{1}{1+x^2}
e^{x^2/(2T)}e^{2x}
e^{-x^2/(2T)}\dd x
=
+\infty,
$$
\end{small}
which contradicts \eqref{eq:incorrect:prough_L2}.

Nevertheless, the conditional bridge $\E[\mathbf p_t|\mathcal F_t]=p_t$ gives the improved integrability estimate
$$
\left( 
t\mapsto \E\left[
\mathbf p_t\,\big|\,\mathcal F_t
\right]
\right)
\in
L^2_\F(\Omega,C([0,T],\R)),
$$ 
but this estimate does not give better integrability estimates for the rough adjoint $\mathbf p$ itself.
\end{remark}

\subsection{Proof of the PMP}\label{sec:pmp_unified:conclusion}
To prove the unified \PMP, we use the It\^o PMP below. The proof is standard and is in Appendix \ref{sec:pmp_ito}.
\begin{tcolorbox}[
title={%
\textbf{Theorem \customlabel{thm:pmp:ito}{3.1}}\textbf{(It\^o Stochastic PMP)}\phantom{$A^1_1$}
}, 
coltitle=black, 
boxrule=3pt,
colframe=black!6, %
enhanced, colback=black!2, boxsep=1pt, left=5pt, right=5pt]
Define $T,U,\Omega,\F,\Prob,B,\bar x_0,b,\sigma$ as in Assumption \ref{assum:pmp:ito}. 
 Let   $f:[0,T]\times\R^n\times U\to\R$   
satisfy the same assumptions as $b$.   
Define $f,g,h$ as in Assumption \ref{assum:pmp}, and the It\^o Hamiltonian 
\begin{align*}
H(t,x,u,p,q,\mathfrak p_0)
=
p^\top b(t,x,u)
+
\mathfrak p_0 f(t,x,u)
+
\sum_{j=1}^d q^{j\top} \sigma^j(t,x). 
\end{align*}
Let $(x,u)$ be an optimal solution to \ocp.  
Then, there exist stochastic processes
$(p,q)\in L^2_\F(\Omega,C([0,T],\R^n)) \times  L^2_\F([0,T]\times\Omega,\R^{n\times d})$,  
and nontrivial multipliers $(\mathfrak p_0,\mathfrak p_1,\dots,\mathfrak p_r)\neq0$
with $\mathfrak p_0\in\{-1,0\}$, such that:
\begin{enumerate}[label=(\roman*)]
\item \textbf{Adjoint Equation}: $(p,q)$ solves the BSDE
\begin{equation}
\dd p_t
=
-\nabla_xH(t,x_t,u_t,p_t,q_t,\mathfrak p_0)\dd t
+q_t\dd B_t.
\end{equation}
\item \textbf{Transversality Condition}: almost surely,
\begin{equation}
\smash{
p_T
=
\mathfrak p_0\nabla g(x_T)
+
\sum_{i=1}^{r}\mathfrak p_i\nabla h_i(x_T).
}
\end{equation}
\item \textbf{Maximality Condition}: for $(\dd t\otimes\Prob)$-almost every\,$(t,\omega)$, 
\begin{equation}
u_t
\in
\arg\max_{v\in U}\,
H\left(t,x_t,v,p_t,q_t,\mathfrak p_0\right).
\end{equation}
\end{enumerate} 
\end{tcolorbox} 

The unified stochastic PMP follows from the It\^o PMP and the conditional bridge in Lemma \ref{lem:ito_rough_bridge}.

\begin{proof}[Proof of \PMP (Theorem \ref{thm:pmp})]
Define the cost-augmented state $\widetilde x=(x,x^0)$ with SDE dynamics
\begin{align}
\dd \widetilde x_t
&=
\widetilde b(t,\widetilde x_t,u_t)\dd t
+
\widetilde \sigma(t,\widetilde x_t)\dd B_t,
\hspace{10mm} 
\widetilde x_0=(\bar x_0,0),
\\
\dd \widetilde x_t
&=
\widetilde{\mathbf b}(t,\widetilde x_t,u_t)\dd t
+
\widetilde \sigma(t,\widetilde x_t)\dd\mbB_t,
\hspace{9mm} 
\widetilde x_0=(\bar x_0,0),
\end{align}
where $
\widetilde{\mathbf b}(t,x,u)
:=
\widetilde b(t,\widetilde x,u)
-
\tfrac12\nabla_{\widetilde x}\widetilde\sigma(t,\widetilde x)\widetilde\sigma(t,\widetilde x)$ and
\begin{align*}
\widetilde b(t,\widetilde x,u):=\begin{bmatrix}
b(t,x,u)
\\
f(t,x,u)
\end{bmatrix},
\ \ 
\widetilde {\mathbf b}(t,\widetilde x,u):=\begin{bmatrix}
\mathbf b(t,x,u)
\\
f(t,x,u)
\end{bmatrix},
\ \ 
\widetilde \sigma(t,\widetilde x):=\begin{bmatrix}
\sigma(t,x)
\\
0
\end{bmatrix}.
\end{align*}
Thanks to Proposition \ref{prop:brownian_rde_consistency}, these SDEs have the same well-defined solution  $\widetilde x\in L^2_\F(\Omega,C([0,T],\R^{n+1}))$.  
The linearized drift and diffusion have the structure
\begin{gather*}
\nabla_{\widetilde x}\widetilde b(t,\widetilde x,u)=\AverageSmallMatrix{
\nabla_xb(t,x,u) & 0
\\
\nabla_xf(t,x,u)  & 0
},
\ \ \ 
\nabla_{\widetilde x}\widetilde\sigma^j(t,\widetilde x)=\AverageSmallMatrix{
\nabla_x\sigma^j(t,x) & 0
\\
0 & 0
},
\ \ \ 
\nabla_{\widetilde x}
\widetilde\sigma^j(t,\widetilde x)
\widetilde\sigma^j(t,\widetilde x)
=\AverageSmallMatrix{
\nabla_x\sigma^j(t,x)\sigma^j(t,x)
\\
0
}.
\end{gather*}

By the It\^o PMP, there exist nontrivial multipliers
$$
(\mathfrak p_0,\mathfrak p_1,\dots,\mathfrak p_r)\neq0,
\qquad
\mathfrak p_0\in\{-1,0\},
$$
and $$(p,q)\in L^2_\F(\Omega,C([0,T],\R^n)) \times  L^2_\F([0,T]\times\Omega,\R^{n\times d})$$ that satisfy the It\^o adjoint equation,
transversality condition, 
and maximality condition. Let
$$
\bar p_T
:=
\mathfrak p_0\nabla g(x_T)
+
\sum_{i=1}^r
\mathfrak p_i\nabla h_i(x_T)
\in L^\infty_{\F_T}(\Omega,\R^n).
$$

Let $(\widetilde p, \widetilde q)\in L^2_\F(\Omega,C([0,T],\R^{n+1})) \times  L^2_\F([0,T]\times\Omega,\R^{(n+1)\times d})$ and $\widetilde{\mathbf p}\in L^2(\Omega,C([0,T],\R^{n+1}))$ solve the SDEs
\begin{align} 
\dd \widetilde p_t
&=
-\left(
\nabla_{\widetilde x}\widetilde b(t,\widetilde x_t,u_t)^\top \widetilde p_t 
+
\nabla_{\widetilde x}\widetilde\sigma(t,\widetilde x_t)^\top \widetilde q_t
\right)
\dd t
+
\widetilde q_t\dd B_t,
\hspace{1cm}
\widetilde p_T=(\bar p_T,\mathfrak p_0),
\\ 
\dd \widetilde{\mathbf p}_t
&=
-\nabla_{\widetilde x}\widetilde{\mathbf b}(t,\widetilde x_t,u_t)^\top \widetilde{\mathbf p}_t 
\dd t
-\nabla_{\widetilde x}\widetilde\sigma(t,\widetilde x_t)^\top
\widetilde{\mathbf p}_t
\dd \mbB_t,
\hspace{19mm}\hspace{1pt}
\widetilde{\mathbf p}_T=(\bar p_T,\mathfrak p_0),
\end{align}
and denote
$$
\widetilde p_t=
\begin{bmatrix}
\widehat p_t\\
p^0_t
\end{bmatrix},
\qquad
\widetilde q_t^j=
\begin{bmatrix}
\widehat q_t^j\\
q_t^{0,j}
\end{bmatrix},
\qquad
\widetilde{\mathbf p}_t=
\begin{bmatrix}
\mathbf p_t\\
\mathbf p^0_t
\end{bmatrix}.
$$
These solutions  are well-defined thanks to  Theorem \ref{thm:sdes:ito} and Theorem \ref{thm:sdes:rough},  and $\widetilde{\mathbf p}\in L^2$ by \eqref{eq:bpTinL2+eps=>bpinL2}. 

The conditional bridge in  Lemma \ref{lem:ito_rough_bridge} gives
\begin{equation}\label{eq:pmp:proof:bridge}
\widetilde p_t
=
\E\left[
\widetilde{\mathbf p}_t\,\big|\,
\mathcal F_t
\right]
\implies
\begin{cases}
\widehat p_t
=
\E\left[
\mathbf p_t\,\big|\,
\mathcal F_t
\right],
\\
p_t^0
=
\E\left[
\mathbf p_t^0\,\big|\,
\mathcal F_t
\right].
\end{cases}
\end{equation}

\textbf{(i)-(ii) Rough adjoint equation and transversality condition}: The sparse structure of $(\nabla_{\widetilde x}\widetilde{\mathbf b}, \nabla_{\widetilde x}\widetilde\sigma)$ implies that $\mathbf p^0_t=\mathfrak p_0$ for all $t\in[0,T]$ almost surely, and $\mathbf p$ solves the rough SDE
\begin{align*}
\dd \mathbf p_t
&=
-\left(
\nabla_x\mathbf b(t,x_t,u_t)^\top\mathbf p_t
+
\mathfrak p_0\nabla_x f(t,x_t,u_t)
\right)\dd t
-
\nabla_x\sigma(t,x_t)^\top
\mathbf p_t
\dd\mbB_t
\\
&=
-\nabla_x\mathbf H(t,x_t,u_t,\mathbf p_t,\mathfrak p_0)\dd t
-
\nabla_x\sigma(t,x_t)^\top
\mathbf p_t\dd\mbB_t,
&&\mathbf p_T=\bar p_T,
\end{align*}
which is precisely the rough adjoint equation and transversality condition.

\textbf{It\^o adjoint equation}: 
The conditional bridge \eqref{eq:pmp:proof:bridge}  gives $p_t^0
=
\E\left[
\mathbf p_t^0\,\big|\,
\mathcal F_t
\right]=\mathfrak{p}_0$, so $p_t^0$ is constant.   Thus, the first $n$ components $(\widehat p, \widehat q)$ of the It\^o adjoint solve
\begin{align} 
\dd \widehat p_t
&=
-\left(
\nabla_x b(t,x_t,u_t)^\top \widehat p_t 
+
\mathfrak p_0\nabla_x f(t,x_t,u_t)
+
\nabla_x\sigma(t,x_t)^\top  \widehat q_t
\right)
\dd t
+
 \widehat q_t\dd B_t,
\hspace{0.5cm}
 \widehat p_T=\bar p_T,
\end{align}
which is precisely the It\^o adjoint BSDE. By uniqueness of solutions (Theorem \ref{thm:sdes:ito}), $(\widehat p,\widehat q)=(p,q)$.

\textbf{(iv) It\^o-rough conditional bridge}: From \eqref{eq:pmp:proof:bridge}, we   obtain the conditional bridge of \PMP
$$
p_t
=
\widehat p_t
=
\E\left[
\mathbf p_t\,\big|\,\F_t
\right].
$$ 
  
\textbf{(iii) Rough maximality condition}:
Since $\sigma$ is
independent of the control, the term 
$\sum_{j=1}^d q_t^{j\top}\sigma^j(t,x_t)$ 
in the It\^o Hamiltonian is independent of the control.
Moreover, the correction term
$$
b(t,x,u)
=
\mathbf b(t,x,u)
+
\frac12
\sum_{j=1}^d
\nabla_x\sigma^j(t,x)\sigma^j(t,x),
$$
is independent of $u$. Therefore, for fixed
$(t,\omega)$,  
$$
\mathop{\arg\max}_{v\in U}
H(t,x_t,v,p_t,q_t,\mathfrak p_0)
=
\mathop{\arg\max}_{v\in U}
\mathbf H(t,x_t,v,p_t,\mathfrak p_0).
$$
The It\^o maximality condition then implies
$$
u_t
\in
\mathop{\arg\max}_{v\in U}
\mathbf H(t,x_t,v,p_t,\mathfrak p_0).
$$
Using the conditional bridge $p_t=\E[\mathbf p_t\mid\F_t]$, we obtain
$$
u_t
\in
\mathop{\arg\max}_{v\in U}
\mathbf H
\left(
t,x_t,v,
\E\left[\mathbf p_t\,\big|\,\F_t\right],
\mathfrak p_0
\right)
\quad\text{for $(\dd t\otimes\Prob)$-a.e. $(t,\omega)$},
$$
which proves the rough
Hamiltonian maximality condition, and concludes the proof of Theorem \ref{thm:pmp}. 
\end{proof}

\section{Proof of the Rough Stochastic PMP}%
\label{sec:pmp_rough}

We now prove the rough PMP (Theorem \ref{thm:pmp:rough}) using rough path theory only. %
We  relax the assumptions made for \PMP that uses It\^o calculus, so Theorem \ref{thm:pmp:rough} applies to a broader class of problems. 

\begin{assumption}[Definitions and assumptions for the rough stochastic PMP (Theorem \ref{thm:pmp:rough})]
\label{assum:pmp:rough}
Let 
$T>0$, $\rho\in[1,\frac{3}{2})$, $p\in(2\rho,3)$, 
and  
$U\subseteq\R^m$. 
\begin{itemize}
\item Let $(\Omega,\F,(\F_t)_{0\le t\le T},\Prob)$ be a filtered probability space with $\mbB=(B,\bB)$ an enhanced Gaussian process as defined in Section \ref{sec:background:gaussian_paths}, where  $B$ is a centered, continuous, $\R^d$-valued Gaussian process with independent components satisfying the regular covariance  condition \eqref{eq:gaussian_covariance_regular} for $\rho$. Assume that $\mbB|_{[0,t]}$ is $\F_t$-measurable  for every $t\in[0,T]$. 
 
\item The initial conditions $\bar x_0$ satisfy $\bar x_0\in L^2_{\F_0}(\Omega,\R^n)$. 
\item The drift  $\mathbf b:[0,T]\times\R^n\times U\to\R^n$ and cost $f:[0,T]\times\R^n\times U\to\R$   
satisfy:
\begin{itemize} 
\item 
$\mathbf b(\cdot, x, u):[0,T]\to\R^n$  is  measurable for all $(x,u)\in\R^n\times U$,  
\item  
$\mathbf b(t,\cdot,\cdot):\R^n\times U\to\R^n$ 
is continuous for almost every $t\in[0,T]$,
\item 
$\mathbf b(t,\cdot,u):\R^n\to\R^n$  is  continuously differentiable  for almost every $t\in[0,T]$  and  all $u\in U$,  
\item 
$\|\mathbf b(t,x,u)\|+\big\|\nablaof{x}\mathbf b(t,x,u)\big\| \leq C_{\mathbf b,f}$ 
and 
$\big\|\nablaof{x}\mathbf b(t,x,u)-\nablaof{x}\mathbf b(t,\tilde{x},u)\big\|\leq C_{\mathbf b,f}\|x-\tilde{x}\|$  
for almost every $t\in[0,T]$, all $x,\tilde{x}\in\R^n$, and all $u\in U$,
\item %
$\|\mathbf b(t,x,u)-\mathbf b(t,x,\tilde{u})\|\leq C_{\mathbf b,f}\|u-\tilde{u}\|$ for almost every $t\in[0,T]$, all $x\in\R^n$, and all $u,\tilde{u}\in U$,
\end{itemize}
and similarly for $f$.
\item The diffusion $\sigma:[0,T]\times\R^n\to\R^{n\times d}$ satisfies $\sigma\in C^4_b([0,T]\times\R^n,\R^{n\times d})$.
\item The terminal cost $g:\R^n\to\R$ and constraints $h:\R^n\to\R^r$ are continuously differentiable and %
satisfy 
$\big\|\nabla g(x)\big\|+\big\|\nabla h(x)\big\|
\leq C_{g,h}$
and 
$\big\|\nabla g(x)-\nabla g(\tilde{x})\big\|
+
\big\|\nabla h(x)-\nabla h(\tilde{x})\big\|
\leq C_{g,h}\|x-\tilde{x}\|$ for all $x,\tilde{x}\in\R^n$ 
for a constant $C_{g,h}>0$.
\end{itemize}
\end{assumption}

\subsection*{Proof of the rough PMP}
\label{sec:pmp_rough:proof}

\begin{proof}[Proof of Theorem \ref{thm:pmp:rough}]
The proof follows the classical needle variations and endpoint separation strategy. 
Throughout this section, under Assumption \ref{assum:pmp:rough}, the rough SDEs have well-defined solutions thanks to Theorem \ref{thm:sdes:rough}.

\textbf{Preliminary step: Cost-augmented system, needle variations, endpoint mapping.}

\textit{Cost-augmented system}: 
Define the cost-augmented state $$\widetilde x=\begin{bmatrix}
x \\ x^0
\end{bmatrix}
\in L^2_{\F}\!\left(\Omega,C\!\left([0,T],\R^{n+1}\right)\right)$$ with SDE dynamics
\begin{align}
\label{eq:cost_augmented_state:rough}
\dd \widetilde x_t
&=
\widetilde{\mathbf b}(t,\widetilde x_t,u_t)\dd t
+
\widetilde \sigma(t,\widetilde x_t)\dd\mbB_t,
\hspace{9mm} 
\widetilde x_0=(\bar x_0,0),
\end{align}
where
\begin{gather} 
\widetilde {\mathbf b}(t,\widetilde x,u):=\begin{bmatrix}
\mathbf b(t,x,u)
\\
f(t,x,u)
\end{bmatrix},
\qquad
\widetilde \sigma^j(t,\widetilde x):=\begin{bmatrix}
\sigma^j(t,x)
\\
0
\end{bmatrix},
\\
\label{eq:rough_tilde_b_sig_Phi}
\widetilde{\mathbf{A}}_t
:=
\nabla_{\widetilde x}\widetilde{\mathbf b}(t,\widetilde x_t,u_t)
=
\begin{bmatrix}
\nabla_x\mathbf b(t,x_t,u_t) & 0\\
\nabla_x f(t,x_t,u_t) & 0
\end{bmatrix},
\qquad
\widetilde\Sigma^j_t 
:= 
\nabla_{\widetilde x}\widetilde\sigma^j(t,\widetilde x_t)
=
\begin{bmatrix}
\nabla_x\sigma^j(t,x_t) & 0\\
0 & 0
\end{bmatrix},
\end{gather} 
with $j=1,\dots,d$.

\textit{Needle variations}: We use needle variations to analyze how state trajectories  change under short control variations. 
Let $u\in L^2([0,T]\times\Omega,U)$ and $V\subset U$ be a countable dense subset.  
We say that $t_1\in[0,T]$ is a \textit{Lebesgue time}  if  almost surely,
\begin{subequations}
\label{eq:lebesgue_time}
\begin{align}
\label{eq:lebesgue_time:u}
\lim_{h\to0}\frac{1}{h}\int_{t_1}^{t_1+h}\widetilde {\mathbf b}(t,\widetilde x_t(\omega),u_t(\omega))\dd t
&= 
\widetilde {\mathbf b}(t_1,\widetilde x_{t_1}(\omega),u_{t_1}(\omega)), 
\quad \text{ and}
\\
\label{eq:lebesgue_time:vn}
\lim_{h\to0}\frac{1}{h}\int_{t_1}^{t_1+h}\widetilde {\mathbf b}(t,\widetilde x_t(\omega),v_n)\dd t 
&=
\widetilde {\mathbf b}(t_1,\widetilde x_{t_1}(\omega),v_n)
\ \ \text{for all $v_n\in V$.}
\end{align}
\end{subequations}
Almost every $t_1\in[0,T]$ is a Lebesgue point. 
Indeed, let $V=\{v_1,v_2,\dots\}\subset U$ be a dense subset. 
Almost surely, $t\mapsto \widetilde {\mathbf b}(t,\widetilde x_t(\omega),u_t(\omega))$ and each $t\mapsto \widetilde {\mathbf b}(t,\widetilde x_t(\omega),v_n)$ is in $L^1([0,T],\R^n)$.  
By the Lebesgue differentiation theorem, there exist sets $I_u(\omega),I_1(\omega),I_2(\omega),\dots\subseteq[0,T]$ of full measure such that 
\eqref{eq:lebesgue_time:u} holds for all $t_1\in I_u(\omega)$ and 
\eqref{eq:lebesgue_time:vn} holds for all $t_1\in I_n(\omega)$ for each $n$. 
The set $I(\omega):=I_u(\omega)\cap I_1(\omega)\cap\dots\subseteq[0,T]$ has full measure,  
and any $t_1\in I(\omega)$ satisfies \eqref{eq:lebesgue_time:u} and \eqref{eq:lebesgue_time:vn} for all $n$. 
So, by Fubini's theorem applied to the indicator function of the set $A=\{(t_1,\omega)\in[0,T]\times\Omega:\smash{\lim\limits_{h\to 0}}\frac1h\int_{t_1}^{t_1+h}\widetilde {\mathbf b}(t,\widetilde x_t(\omega),u_t(\omega))\dd t \neq \widetilde {\mathbf b}(t_1,\widetilde x_{t_1}(\omega),u_{t_1}(\omega)) 
\ \text{or} \  
\smash{\lim\limits_{h\to 0}}\frac1h\int_{t_1}^{t_1+h}\widetilde {\mathbf b}(t,\widetilde x_t(\omega),v_n)\dd t \neq \widetilde {\mathbf b}(t_1,\widetilde x_{t_1}(\omega),v_n) \ \text{for some } n=1,2,\dots\}$, almost every $t\in[0,T]$ is a Lebesgue time\footnote{Fubini's theorem gives  
$0=
\int_\Omega\int_0^T\mathbf{1}_A(t,\omega)\dd t\dd\Prob(\omega)=\int_0^T\int_\Omega\mathbf{1}_A(t,\omega)\dd\Prob(\omega)\dd t=\int_0^T\Prob(A_t)\dd t$ with $A_t(\omega)=\{\omega\in\Omega:(t,\omega)\in A\}$. Since $\Prob(A_t)\geq 0$, we get $\Prob(A_t)=0$ for almost every $t_1\in[0,T]$, so almost every $t_1\in[0,T]$ is a Lebesgue time.}. 

Let $0<t_1<\cdots<t_q<T$ 
be Lebesgue times of the optimal control $u\in\U$. For $i=1,\dots,q$, let 
$$
\bar u_i\in L^2_{\F_{t_i}}(\Omega,U),
\qquad
0\leq\eta_i<t_{i+1}-t_i,
$$
with $t_{q+1}:=T$. Define the needle variation
$\pi=\{t_i,\eta_i,\bar u_i\}_{i=1}^q$ of $u$ by
$$
u_t^\pi
=
\begin{cases}
\bar u_i, & t\in[t_i,t_i+\eta_i],\\
u_t, & \text{otherwise}.
\end{cases}
$$

Let  $\widetilde x^\pi\in L^2_\F(\Omega,C([0,T],\R^{n+1}))$ be the solution to the SDE
\begin{align}
\dd \widetilde x^\pi_t
&=
\widetilde{\mathbf b}(t,\widetilde x^\pi_t,u^\pi_t)\dd t
+
\widetilde \sigma(t,\widetilde x^\pi_t)\dd\mbB_t,
\hspace{9mm} 
\widetilde x^\pi_0=(\bar x_0,0),
\end{align}
that corresponds to the control $u^\pi$. 

For each $i=1,\dots,q$, 
let $\widetilde v^{\pi_i}\in L^2_\F(\Omega,C([t_i,T],\R^{n+1}))$ be the solution to the linearized SDE
\begin{equation}
\label{eq:linearized_variation:rough} 
\dd \widetilde v_t^{\pi_i}
=
\widetilde{\mathbf{A}}_t \widetilde v_t^{\pi_i}\dd t
+
\widetilde\Sigma_t\widetilde v_t^{\pi_i}\dd\mbB_t,
 \quad\ 
 \widetilde v_{t_i}^{\pi_i}
=
\Delta\widetilde{\mathbf b}_i
:=
\widetilde{\mathbf b}(t_i,\widetilde x_{t_i},\bar u_i)
-
\widetilde{\mathbf b}(t_i,\widetilde x_{t_i},u_{t_i}).
\end{equation}

\begin{lemma}[Needle variations for rough SDEs]
\label{lem:needle:rough}
Under Assumption \ref{assum:pmp:rough}, define the needle variation $\pi=\{t_i,\eta_i,\bar u_i\}_{i=1}^q$, state trajectories $(\widetilde x,\widetilde x^\pi)$, and linear variations $\{\widetilde v^{\pi_i}\}_{i=1}^q$ as above. 
Let $\Phi:\R^{n+1}\to\R^k$ be continuously differentiable and satisfy $\|\nabla\Phi(x)\|\leq C_\Phi$ and
$\|\nabla\Phi(x)-\nabla\Phi(y)\|
\leq 
C_\Phi\|x-y\|$ for all $x,y\in\R^{n+1}$ for some $C_\Phi<\infty$. 
Then, %
\begin{equation}
\label{eq:rough_needle}
\E\left[
\left\|
\Phi(\widetilde x_T^\pi)
-
\Phi(\widetilde x_T)
-
\sum_{i=1}^q
\eta_i \nabla\Phi(\widetilde x_T)\widetilde v_T^{\pi_i}
\right\|
\right]
\leq
C\,
\E\left[\exp(CN_{\alpha,[0,T]}(\mbB))\right]
o\left(\sum_{i=1}^q\eta_i\right),
\end{equation}
where $0 < C < \infty$ and $0 < \alpha < 1$ depend on $(p, T, b, \sigma, \Phi)$, and $\E\left[\exp(CN_{\alpha,[0,T]}(\mbB))\right]<\infty$.
\end{lemma}
\noindent 
Lemma \ref{lem:needle:rough} extends \cite[Lemma 4.3]{Lew2026} to the case with stochastic needle variations. 
Its proof closely follows  \cite[Lemma 4.3]{Lew2026}, since all error bounds such as
$
\left\|
\widetilde x^\pi(\omega)
{-}
\widetilde x(\omega)
{-}
\eta_1 \widetilde v^{\pi_1}(\omega)
\right\|_\infty
\leq
C\exp(CN_{\alpha,[0,T]}(\mbB(\omega)))o(\eta_1)$ in \cite[Proposition 4.2]{Lew2026} hold pathwise, but they need to be extended to random controls with values in $L^1([0,T],U)$ (instead of deterministic controls $u\in L^\infty([0,T],U)$), and be adapted to avoid dependencies on $(u,\bar{u}_1,\dots,\bar{u}_q)$. These extensions are possible since the drift $b$ is bounded so all error bounds can be derived as a function of the impulse durations $\eta_i$ instead of using bounds on norms of $u$. These  long but straightforward derivations are described in Appendix \ref{sec:proofs:rough_paths}.

\textit{Endpoint mapping}: Next, define the map
\begin{equation}
\label{eq:endpoint:rough_gradient}
\Phi:\mathbb R^{n+1}\to\mathbb R^{r+1},
\ 
\widetilde x\mapsto\Phi(\widetilde x)
=
\begin{bmatrix}
h(x)\\
x^0+g(x)
\end{bmatrix},
\quad\text{with}\ 
\nabla\Phi(\widetilde x)
=
\begin{bmatrix}
\nabla h(x) & 0\\
\nabla g(x) & 1
\end{bmatrix},
\end{equation}
that evaluates the terminal constraints and total cost. 
With $\R^q_+ = \{\eta=(\eta_1,\dots,\eta_q)\in\R^q: \eta_1\geq 0,\dots,\eta_q\geq 0\}$, $\delta = \min\{T-t_q,t_{i+1}-t_i, i=1,\dots,q-1\}$, and $B^q_\delta=\{\eta\in\R^q:\|\eta\|\leq\delta\}$, we define the endpoint map $F:\R^q_+\cap B^q_\delta\to \R^{r+1}$ by 
\begin{equation}
\label{eq:endpoint:rough}
F(\eta)
:=
\E\left[
\Phi(\widetilde x_T^\pi)
-
\Phi(\widetilde x_T)
\right].
\end{equation}
The endpoint map $F$ satisfies $F(0)=0$. 
Also, the linear map
\begin{equation}
\label{eq:endpoint_differential:rough}
\dd F_0:\R^q_+\to\R^{r+1}, 
\
\eta\mapsto\dd F_0(\eta)
=
\sum_{i=1}^q
\eta_i
\E\left[
\nabla\Phi(\widetilde x_T)
\widetilde v_T^{\pi_i}
\right].
\end{equation}
is the Gateaux differential of $F$ at $0$ in the direction $\eta$: 
\begin{equation}
\label{eq:gateaux_limit:rough}
\lim_{\alpha>0,\alpha\to0}
\frac{F(\alpha\eta)}{\alpha}
=
\dd F_0(\eta)
\ \text{ for any $\eta\in\mathbb R_+^q$.}
\end{equation}
Indeed, by Lemma \ref{lem:needle:rough} and Jensen's inequality,  
$
\left\|
F(\alpha\eta)
-
\dd F_0(\alpha\eta)
\right\|
\leq
Co(\sum_{i=1}^q\alpha\eta_i),
$ 
for a constant $C>0$ and $\alpha>0$ small-enough, so \eqref{eq:gateaux_limit:rough} follows after dividing by $\alpha$ and letting
$\alpha\to 0$.

\textbf{Step 1: Variational linearization and separation argument.} 
This step appears identically in the proof of the rough stochastic PMP for deterministic controls \cite{Lew2026} and of the classical It\^o PMP \cite{BonalliLewESAIM2022}, see Section \ref{sec:pmp_ito}. 
We define the closed convex cone
$$
K
=
\overline{
\left\{
\sum_{i=1}^{\widetilde q}
\alpha_i
\E\left[
\nabla\Phi(\widetilde x_T)
\widetilde v_T^{\pi_i}
\right]
:
\alpha_i\geq0,\ 
\pi_i \text{ a needle variation},\
\widetilde q\in\mathbb N
\right\}
},
$$
and note that  $K\neq\R^{r+1}$. Otherwise, 
a fixed point argument would give a feasible needle perturbation with  strictly
smaller expected cost, contradicting the optimality of $(x,u)$. Indeed, by contradiction, assume that $K=\R^{r+1}$. Then there exist $q\in\N$ needle variations $\pi_i=\{t_i,\eta_i,\bar{u}_i\}$ such that $K=\{
\sum_{i=1}^q\alpha_i
\E[
\nabla\widetilde\Phi(\tilde{x}_T)
\tilde{v}^{\pi_i}_T
]:\alpha_i\geq 0\}$. With this needle variation $\pi=\{t_1,\dots,t_q,\eta_1,\dots,\eta_q,\bar{u}_1,\dots,\bar{u}_q\}$ of $u$,  define the maps $F$ and $\dd F_0$. Then, $\dd F_0(\R^q_+)=K=\R^{r+1}$, so $0\in\Int(F(B^q_\delta\cap\R^q_+))$ by \cite[Lemma 12.4]{Agrachev2004} (whose proof relies on Brouwer's fixed point theorem). %
So, there exists another feasible trajectory $(x^\pi,u^\pi)$ with a strictly lower cost $\E[x^{\pi,0}_T+g(x^\pi_T)]<\E[x^0_T+g(x_T)]$, so $(x,u)$ is not optimal: a contradiction. 
Thus, $K\subset\R^{r+1}$.

Therefore, by the hyperplane separation theorem, there exist
$$
\mathfrak p
=
(\mathfrak p_1,\dots,\mathfrak p_r,\mathfrak p_0)
\in\mathbb R^{r+1},
\qquad
\mathfrak p\neq0,
\qquad
\mathfrak p_0\leq0,
$$
such that
$$
\mathfrak p^\top z\leq0
\qquad
\text{for all }z\in K.
$$
After normalization, we may assume $
\mathfrak p_0\in\{-1,0\}$. 
In particular, we get the separation inequality
\begin{equation}
\label{eq:rough_sep_ineq}
\mathfrak p^\top
\E\left[
\nabla\Phi(\widetilde x_T)
\widetilde v_T^{\pi_1}
\right]
\leq0
\ \text{ for any needle variation $\pi_1=(t_1,\eta_1,\bar u_1)$}.
\end{equation}

\textbf{Step 2: Adjoint equation (i) and transversality condition (ii).} 
This step follows Section \ref{sec:pmp_unified:conclusion}. 
Let $\widetilde{\mathbf p}\in L^2(\Omega,C([0,T],\R^{n+1}))$ be the solution to the rough SDE
\begin{equation}
\label{eq:rough_aug_adjoint}
\dd\widetilde{\mathbf p}_t
=
-
\widetilde{\mathbf{A}}_t^\top
\widetilde{\mathbf p}_t
\dd t
-
\widetilde\Sigma_t^\top
\widetilde{\mathbf p}_t
\dd\mbB_t,
\qquad
\widetilde{\mathbf p}_T
=
\mathfrak p^\top \nabla\Phi(\widetilde x_T).
\end{equation}
Denoting $ 
\widetilde{\mathbf p}_t
=(\mathbf p_t,\mathbf p_t^0)$, the transversality condition follows directly:
$$
\widetilde{\mathbf p}_T
=
\begin{bmatrix}
\mathbf p_T\\
\mathbf p_T^0
\end{bmatrix}
=
\mathfrak p^\top\nabla\Phi(\widetilde x_T)
\mathop{=}^{\eqref{eq:endpoint:rough_gradient}}
\begin{bmatrix}
\mathfrak p_0\nabla g(x_T)
+
\sum_{i=1}^r
\mathfrak p_i\nabla h_i(x_T)
\\
\mathfrak p_0
\end{bmatrix}.
$$
The adjoint equation follows from the sparse structure of $(\widetilde{\mathbf{A}},\widetilde\Sigma)= \eqref{eq:rough_tilde_b_sig_Phi}$ in the SDE \eqref{eq:rough_aug_adjoint}:
$$
\dd\mathbf p_t^0=0,
\qquad
\mathbf p_T^0=\mathfrak p_0,
$$
so $\mathbf p_t^0=\mathfrak p_0$ 
almost surely for all $t\in[0,T]$. The first $n$ components then satisfy
\begin{align*}
\dd \mathbf p_t
&=
-
\left(
\nabla_x\mathbf b(t,x,u)^\top\mathbf p
+
\mathfrak p_0\nabla_x f(t,x,u)
\right)\dd t
-
\nabla_x\sigma(t,x_t)^\top\mathbf p_t\dd\mbB_t
\\
&=
-
\nabla_x\mathbf H(t,x,u,\mathbf p,\mathfrak p_0)
\dd t
-
\nabla_x\sigma(t,x_t)^\top\mathbf p_t\dd\mbB_t,
\end{align*}
which is precisely the adjoint equation. 

\textbf{Step 3: Maximality condition.} 
This step follows the proof of the rough stochastic PMP \cite{Lew2026}, replacing  deterministic needle variations for stochastic ones, and using a conditional expectation to obtain non-anticipative controls.

Let 
$\pi_1=\{t_1,\eta_1,\bar u_1\}$ be a needle variation  of the optimal control $u$, and 
 $\widetilde v^{\pi_1}$ solve the rough SDE 
\begin{align}\tag{\ref{eq:linearized_variation:rough}}
\dd \widetilde v_t^{\pi_1}
=
\widetilde{\mathbf{A}}_t\widetilde v_t^{\pi_1}\dd t
+
\widetilde\Sigma_t\widetilde v_t^{\pi_1}\dd\mbB_t,
\qquad
\widetilde v_{t_1}^{\pi_1}
=
\Delta\widetilde{\mathbf b}_1
:=
\widetilde{\mathbf b}(t_1,\widetilde x_{t_1},\bar u_1)
-
\widetilde{\mathbf b}(t_1,\widetilde x_{t_1},u_{t_1}).
\end{align}
By   the chain rule for rough paths (Corollary \ref{cor:duality:rough:cost_augmented}), 
the inner product $\ip{\widetilde{\mathbf p}_t}{\widetilde v_t^{\pi_1}}$ is constant on $t\in[t_1,T]$, so
$$
\ip{\widetilde{\mathbf p}_T}{\widetilde v_T^{\pi_1}}
=
\ip{\widetilde{\mathbf p}_{t_1}}{\Delta\widetilde{\mathbf b}_1}$$
almost surely. 
Using the terminal condition $
\widetilde{\mathbf p}_T^\top
=
\mathfrak p^\top
\nabla\Phi(\widetilde x_T)$ and 
the inequality $\mathfrak p^\top
\E[
\nabla\Phi(\widetilde x_T)
\widetilde v_T^{\pi_1}
]
\smash{\mathop{\leq}\limits^{\eqref{eq:rough_sep_ineq}}}
0$,
$$
0
\geq
\E\big[
\ip{
\widetilde{\mathbf p}_T}{
\widetilde v_T^{\pi_1}
}
\big]
=
\E\big[
\ip{
\widetilde{\mathbf p}_{t_1}
}{
\Delta\widetilde{\mathbf b}_1
}
\big].
$$
Expanding the augmented variables and using
$\mathbf p_{t_1}^0=\mathfrak p_0$, we obtain
\begin{align}
\label{eq:rough_variational_ineq_raw}
0
&\geq
\E\Big[
\mathbf p_{t_1}^\top
\big(
\mathbf b(t_1,x_{t_1},\bar u_1)
-
\mathbf b(t_1,x_{t_1},u_{t_1})
\big)
+
\mathfrak p_0
\big(
f(t_1,x_{t_1},\bar u_1)
-
f(t_1,x_{t_1},u_{t_1})
\big)
\Big].
\end{align}
The rough adjoint $\mathbf p_{t_1}$ is generally not
$\F_{t_1}$-measurable, since it is defined backwards-in-time from the terminal condition. However, the needle impulse is
$\F_{t_1}$-measurable:
$$
\mathbf b(t_1,x_{t_1},\bar u_1)
-
\mathbf b(t_1,x_{t_1},u_{t_1})
\quad\text{and}\quad
f(t_1,x_{t_1},\bar u_1)
-
f(t_1,x_{t_1},u_{t_1})
$$
are $\F_{t_1}$-measurable. 
The conditional expectation identity
$$
\E[\mathbf p_{t_1}^\top Z]
=
\E[
\E[\mathbf p_{t_1}\mid\F_{t_1}]^\top Z
],
\qquad
Z\in L^2_{\F_{t_1}}(\Omega,\R^n),
$$
turns \eqref{eq:rough_variational_ineq_raw} into
\begin{align}
0
&\geq
\E\Big[
\E\left[\mathbf p_{t_1}|\F_{t_1}\right]^\top
\big(
\mathbf b(t_1,x_{t_1},\bar u_1)
-
\mathbf b(t_1,x_{t_1},u_{t_1})
\big)
+
\mathfrak p_0
\big(
f(t_1,x_{t_1},\bar u_1)
-
f(t_1,x_{t_1},u_{t_1})
\big)
\Big].
\end{align}
Equivalently,
\begin{equation}
\label{eq:hamiltonian_variational_ineq:rough}
\E\big[
\mathbf H(t_1,x_{t_1},\bar u_1,\E[\mathbf p_{t_1}|\F_{t_1}],\mathfrak p_0)
-
\mathbf H(t_1,x_{t_1},u_{t_1},\E[\mathbf p_{t_1}|\F_{t_1}],\mathfrak p_0)
\big]
\leq 0,
\end{equation}
which holds for any Lebesgue time $t_1$  and $\bar u_1\in L^2_{\F_{t_1}}(\Omega,U)$.

We now  prove the maximality condition \eqref{eq:pmp:rough:maximality}. By contradiction,
suppose that   \eqref{eq:pmp:rough:maximality}  does not hold on a set of positive $(\dd t\otimes\Prob)$-measure. 
Then, by the argument in Remark \ref{remark:pmp:rough:localization}, there exist a Lebesgue point
$t_1$, a set  $A_{t_1}\in\F_{t_1}$ with $\Prob(A_{t_1})>0$, and a random variable $\widehat u_1\in L^2_{\F_{t_1}}(\Omega,U)$, such that
\begin{equation}\label{eq:hamiltonian_variational_ineq:rough:con}
\mathbf H(t_1,x_{t_1},\widehat u_1,\E[\mathbf p_{t_1}|\F_{t_1}],\mathfrak p_0)
>
\mathbf H(t_1,x_{t_1},u_{t_1},\E[\mathbf p_{t_1}|\F_{t_1}],\mathfrak p_0)
\quad\text{on }A_{t_1}.
\end{equation}
Define
$\bar u_1 := \widehat u_1\mathbf 1_{A_{t_1}} + u_{t_1}\mathbf 1_{A_{t_1}^c}$. 
Then, $\bar u_1\in L^2_{\F_{t_1}}(\Omega,U)$, and
$$
\E\left[
\mathbf H
\left(
t_1,x_{t_1},\bar u_1,
\E[\mathbf p_{t_1}\mid\F_{t_1}],
\mathfrak p_0
\right)
-
\mathbf H
\left(
t_1,x_{t_1},u_{t_1},
\E[\mathbf p_{t_1}\mid\F_{t_1}],
\mathfrak p_0
\right)
\right]
>0,
$$
which contradicts \eqref{eq:hamiltonian_variational_ineq:rough}. Therefore, the maximality condition holds, which concludes the proof of the rough PMP.
\end{proof}

\begin{remark}\label{remark:pmp:rough:localization}
Define the map $\varphi:[0,T]\times\Omega\times U\to\R$ by 
$$
\varphi(t,\omega,v)
:=
\mathbf H
\left(
t,x_t(\omega),v,
\E[\mathbf p_t\mid\F_t](\omega),
\mathfrak p_0
\right),
$$
where $(t,\omega)\mapsto\E[\mathbf p_t\mid\F_t](\omega)$ denotes an optional (hence progressively measurable) projection of $\mathbf p$. Its existence follows from \cite[Chapter 2, Theorem 4.2]{EthierKurtz1986} applied componentwise to the positive
and negative parts of $\mathbf p$, using $\E[\sup_{t\in[0,T]}\|\mathbf p_t\|]<\infty$. 

Let $\{v^k\}_{k\in\N}$ be a dense countable subset of $U$, and define the sets
\begin{align*}
A&:=
\left\{
(t,\omega)\in[0,T]\times\Omega:
\exists v(t,\omega)\in U: \varphi(t,\omega,v)
>
\varphi(t,\omega,u_t(\omega))
\right\},
\\
A^k&:=
\left\{
(t,\omega)\in[0,T]\times\Omega:
\varphi(t,\omega,v^k)
>
\varphi(t,\omega,u_t(\omega))
\right\},
\end{align*}
where $A^k$ is progressively measurable. 

By contradiction, 
suppose that the maximality condition \eqref{eq:pmp:rough:maximality} fails. 
Then, $(\dd t\otimes\Prob)(A)>0$. Since $A=\bigcup_{k\in\N}A^k$ by the continuity of $v\mapsto\varphi(t,\omega,v)$ and density of $\{v^k\}_{k\in\N}$, we get $
0\leq (\dd t\otimes\Prob)(\bigcup_{k\in\N}A^k)\leq\sum_{k\in\N}(\dd t\otimes\Prob)(A^k)$, so there exists a set $A^k$ for some $k\in\N$ that has positive measure $(\dd t\otimes\Prob)(A^k)>0$. 

Next, define the sets 
$$
I^k:=\{t\in[0,T]:\Prob(A^k_t)>0\},
\qquad
A^k_t:=\{\omega\in\Omega:(t,\omega)\in A^k\}.
$$
where $A^k_t\in\F_t$. 
Then, $I^k$ has positive Lebesgue measure. Indeed, by Fubini's theorem, $(\dd t\otimes\Prob)(A^k)=\int_0^T\Prob(A^k_t)\dd t>0$. Also, the set of Lebesgue points $t_1$ where \eqref{eq:hamiltonian_variational_ineq:rough} holds has full Lebesgue measure. Thus, there exists a Lebesgue point $t_1$  in $I^k$ such that $\Prob(A^k_{t_1})>0$ and
$$
\varphi(t_1,\omega,v^k)
>
\varphi(t_1,\omega,u_{t_1}(\omega)),
\qquad
\omega\in A^k_{t_1}.
$$
Thus, the constant control 
$\widehat u_1 :=  v^k$ satisfies the desired inequality \eqref{eq:hamiltonian_variational_ineq:rough:con} on $A_{t_1}:=A^k_{t_1}$.
\end{remark}

\begin{remark}[Deterministic and anticipative controls]
If the admissible controls are deterministic open-loop controls like in \cite{Lew2026}, then the
needle perturbation $\bar u_1$ is deterministic rather than
$\F_{t_1}$-measurable. The same proof gives the maximality condition
$$
u_t
\in
\operatorname*{arg\,max}_{v\in U}
\E\left[
\mathbf H(t,x_t,v,\mathbf p_t,\mathfrak p_0)
\right]
\qquad
\text{for a.e. }t\in[0,T].
$$
Adapting the proof to (not necessarily $\F$-adapted) anticipative controls is also straightforward.
\end{remark}

\section{Applications}\label{sec:applications}

\subsection{Anticipative vs adapted controls}\label{sec:applications:example}
We give an example to show the difference between adapted and anticipative control solutions, as an application of the rough PMP (Theorem \ref{thm:pmp:rough}). Consider the problem with $x=(x_1,x_2)$ 
\begin{align} 
\min\limits_{u\in \mathcal U} \ \ 
	&\E\bigg[
\int_0^T \tfrac12\|u_t\|^2\dd t+\tfrac12x_{2,T}^2
\bigg] 
\ \ \textrm{s.t.}   \ \  
\begin{bmatrix}
\dd x_{1,t}
\\
\dd x_{2,t}
\end{bmatrix} = \begin{bmatrix}
0 \\ u_t
\end{bmatrix}
\dd t +
\begin{bmatrix}
1
\\
\bar\sigma\cos(x_{1,t})
\end{bmatrix}\dd\mbB_t, 
\qquad 
x_0=0,
\end{align}
with $0<T<1$, $\bar\sigma>0$, $U=[-\bar u,\bar u]$ for $\bar u>\bar\sigma/(1-T)$, and $\mbB$ the Stratonovich lift of Brownian motion, which is an instance of \OCP{} without terminal  constraints. State trajectories satisfy
$$
\smash{
x_{1,t}=B_t,
\qquad
x_{2,t}=\int_0^tu_s\dd s+\bar\sigma\sin(B_t).
}
$$
Thus, $|x_{2,t}|\leq T\bar{u}+\bar\sigma<\bar{u}$ almost surely for all $t\in[0,T]$ and any control $u\in\U$, so the terminal cost $g(x)=\tfrac12x_{2}^2$ has bounded and Lipschitz gradient. 
The Hamiltonian  is
$
\mathbf H\big(t,x,u,(\mathbf p_1,\mathbf p_2),\mathfrak p_0\big)
=
\mathbf p_2 u
+
 \tfrac{\mathfrak p_0}2u^2. 
$
Let $(x,u)$ be a solution to the optimal control problem. By the rough PMP (Theorem \ref{thm:pmp:rough}), let $\mathfrak{p}_0=-1$ (due to the non-triviality condition), and define  the adjoint as the solution to the rough SDE \eqref{eq:pmp:rough:adjoint}. The second component satisfies
\begin{align}
\dd \mathbf p_{2,t}
&=
0,
\qquad
\mathbf p_{2,T}=-\nabla_{x_2} g(x_T)=-x_{2,T}
\quad
\implies
\mathbf p_{2,t}=-x_{2,T}
\end{align}
for all times $t\in[0,T]$ almost surely. 

\textbf{The anticipative case}: Let $\U\,{=}\,L^2([0,T]\,{\times}\,\Omega,U)$. By the maximality condition  \eqref{eq:pmp:rough:maximality:anticipative} and $|x_{2,t}|\,{<}\,\bar{u}$,  
\begin{gather*}
\smash{
u_t
\in
\mathop{\arg\max}_{v\in U}\,
\mathbf H\left(t,x_t,v,\mathbf p_t,-1\right)
=
\mathop{\arg\max}_{v\in U}\,
\left(
\mathbf p_{2,t} v
-\tfrac12v^2
\right) 
\implies
u_t=\mathbf p_{2,t}=-x_{2,T},
}
\end{gather*}
so the control is anticipative and constant. Since $x_{2,T}=\int_0^Tu_s\dd s+\bar\sigma\sin(B_T)=-Tx_{2,T}+\bar\sigma\sin(B_T)$, the anticipative optimal control is
$$
u_t=-\frac{\bar\sigma}{1+T}\sin(B_T).
$$
 
\textbf{The adapted case}:  Let $\U=L^2_\F([0,T]\times\Omega,U)$. By the maximality condition \eqref{eq:pmp:rough:maximality},
\begin{gather*}  
u_t=\E[\mathbf p_{2,t}\, |\, \F_t]
=-\E[x_{2,T}\, |\, \F_t].
\end{gather*}  
Define $X_t:=\E[x_{2,T}\, |\, \F_t]$, which satisfies $\E[X_t|\F_s]=X_s$ for $t\geq s$. Then, using $\E[\sin(B_T)|\F_t]=e^{-(T-t)/2}\sin(B_t)$,
\begin{align*}
X_t
&= 
\E\bigg[
x_{2,t} +\int_t^Tu_s\dd s+\bar\sigma(\sin(B_T)-\sin(B_t))\, \bigg|\, \F_t\bigg]
=
x_{2,t} -
\int_t^T\E[X_s|\F_t]\dd s
+
\bar\sigma
(e^{-(T-t)/2}-1)
\sin(B_t)
\\
&=x_{2,t} 
-
(T-t)X_t
+
\bar\sigma (e^{-(T-t)/2}-1)\sin(B_t).
\end{align*}
Thus, $X_t=\frac{1}{1+T-t}(x_{2,t}-
\bar\sigma (1-e^{-(T-t)/2})\sin(B_t))$, and the control is
$$
u_t=-\frac{x_{2,t}-\bar\sigma(1-e^{-(T-t)/2})\sin(x_{1,t})}{1+T-t}.
$$
This example shows the difference between anticipative and adapted solutions. While the anticipative control is constant and depends on $B_T$ which is unknown in many applications, the adapted control varies in time as a function of $x^1_t$, which is an observable state.
The adapted solution could also be computed using the It\^o PMP. 
The conditional bridge in \PMP connects with such an approach.

\subsection{Adjoint matching for fine-tuning generative models}\label{sec:applications:adjoint_matching}
We rederive the adjoint matching method for fine-tuning diffusion models \cite{Domingo2025} using the unified \PMP, and extend it to the case with state-dependent diffusion. 
In this application, the state SDE represents a generative model. Its goal is to generate final states $x_T$ from a desired distribution (e.g., representing images) by selecting appropriate controls $u_t$ (weights). 
Using an SDE as a generative model provides high sample quality and training stability, thanks to the iterative nature of the generation process. A common post-training objective is to fine-tune this model to respect different desiderata by specifying an objective $g(x)$. This fine-tuning problem can be posed as the stochastic optimal control problem
\begin{align} 
\min\limits_{u\in \mathcal U} \ \ 
	&\E\bigg[
\int_0^T (f(x_t,t)+\tfrac12\|u_t\|^2)\dd t+g(x^u_T)\bigg]
\\
\ \textrm{s.t.}   \ \ 
\label{eq:matching_adjoint:x}
&
\dd x^u_t = (b(t,x^u_t)+\sigma(t,x^u_t)u_t)\dd t +\sigma(t,x^u_t)\dd\mbB_t,
\quad t\in[0,T],
\\
&x^u_0=\bar{x}_0,
\end{align}
with $U=\R^m$. 
In the derivations that follow, we apply \PMP{} to this problem  (though it may not satisfy Assumption \ref{assum:pmp} as the cost and dynamics may be unbounded, one may consider an arbitrarily large set $U$, or replace it with a compact set to obtain a different algorithm). %
The Hamiltonian is
\begin{align*}
\mathbf H(t,x,u,\mathbf p,\mathfrak p_0)
&=
\mathbf p^\top (b(t,x)+\sigma(t,x)u)
+
\mathfrak p_0 (f(x,t)+\tfrac12\|u\|^2). 
\end{align*}
Let $(x,u)$ be a solution to the optimal control problem, %
$\mathfrak{p}_0=-1$ (due to the non-triviality condition), and define  the rough adjoint as the solution to the rough SDE
\begin{subequations}
\label{eq:matching_adjoint:p}
\begin{align}
\dd \mathbf p^u_t
&=
-(
\nabla_xb(t,x^u_t)+\nabla_x(\sigma(t,x^u_t)u_t))^\top \mathbf p^u_t-\nabla_x f(t,x^u_t)) 
\dd t
-
\nabla_x\sigma(t,x^u_t)^\top\mathbf p^u_t\dd\mbB_t,
\\
\mathbf p^u_T&=-\nabla g(x^u_T).
\end{align}
\end{subequations}
Let $\sigma_t^u:=\sigma(t,x^u_t)$.
Then,  by the maximality condition, the control satisfies
\begin{gather*}
u_t
\in
\mathop{\arg\max}_{v\in U}\,
\mathbf H\left(t,x^u_t,v,\E[\mathbf p^u_t\, |\, \F_t],-1\right)
=
\mathop{\arg\max}_{v\in U}\,
\left(
\E[\mathbf p^u_t\, |\, \F_t]^\top\sigma^u_tv
-\tfrac12\|v\|^2
\right)
\implies
u_t=\sigma_t^{u\top}\E[\mathbf p^u_t\, |\, \F_t].
\end{gather*} 
These derivations motivate defining the map $\mathcal{L}:\U\times\U\to\R$ by 
\begin{align*} 
\mathcal{L}(u,v)
:= \E\bigg[
\int_0^T\left\|v_t-\sigma_t^{u\top}\E[\mathbf p^u_t\, |\, \F_t]\right\|^2
\bigg].
\end{align*}
Indeed, any optimal control $u\in \U$ must satisfy $\mathcal{L}(u,u)=0$.
Interestingly, if we only minimize over the second variable, 
\begin{equation}\label{eq:matching_adjoint_minv}
\mathop{\arg\min}_{v\in\U}
\mathcal{L}(u,v)
=
\mathop{\arg\min}_{v\in\U}
\E\bigg[\int_0^T
\left\|
v_t-\sigma_t^{u\top}\mathbf p_t^u
\right\|^2
\bigg],
\end{equation}
which follows from properties of conditional expectations, see Remark \ref{sec:matching_adjoint_minv}. %
This observation motivates the following method of successive approximations (MSA)  \cite{Chernousko1982}: 
From an initial guess $u\in\U$:
\begin{enumerate}
\item Evaluate $p^u$:
\begin{enumerate}
\item Forward simulate $x^u$ by solving the SDE \eqref{eq:matching_adjoint:x} from $\bar x_0$,
\item Backward simulate $\mathbf p^u$ by solving the SDE \eqref{eq:matching_adjoint:p} from $-\nabla g(x^u_T)$.
\end{enumerate}
\item Update $u\gets\mathop{\arg\min}_{v}\mathcal{L}(u,v)$ using \eqref{eq:matching_adjoint_minv}:
\begin{align*}
u\gets  
&\mathop{\arg\min}_v \ 
\E\bigg[\int_0^T
\left\|
v_t-\sigma_t^{u\top} \mathbf p_t^u
\right\|^2
\bigg].
\end{align*}
\end{enumerate}
This algorithm is the adjoint matching method of \cite{Domingo2025}. For fixed $u$, the trajectory $x^u$ and adjoint $\mathbf p^u$ are frozen while the residual $\mathcal{L}(u,v)$ is minimized only over $v$. %
As a result, in \eqref{eq:matching_adjoint_minv}, the conditional expectation disappears, which makes the adjoint matching method efficient and practical.

The bridge also identifies the corresponding It\^o adjoint:
\begin{align*} 
\mathcal{L}(u,v)
:= \E\bigg[
\int_0^T\left\|v_t-\sigma_t^{u\top} \E[\mathbf p^u_t\, |\, \F_t]\right\|^2
\bigg]
=\E\bigg[
\int_0^T\left\|v_t-\sigma_t^{u\top} p^u_t\right\|^2
\bigg],
\end{align*}
where  $(p^u,q^u)$ solves the corresponding It\^o BSDE in \PMP, which is generally challenging to solve. \PMP provides an informative connection between the It\^o PMP and the pathwise training approach that is used in adjoint matching.

\begin{remark}[The state-independent diffusion case]
The adjoint matching method was originally derived in \cite{Domingo2025} for the case
$$
\sigma(t,x)=\hat\sigma(t).
$$
Then, the Stratonovich-to-It\^o correction is zero, and the state rough SDE is equivalent to an It\^o SDE. Also, the rough adjoint is the solution to the random ODE (a rough SDE without diffusion term)
\begin{subequations}
\label{eq:matching_adjoint:p:ode}
\begin{align}
\dd \mathbf p^u_t
&=
-(\nabla_xb(t,x^u_t)^\top \mathbf p^u_t-\nabla_x f(t,x^u_t)) 
\dd t,
\\
\mathbf p^u_T&=-\nabla g(x^u_T),
\end{align}
\end{subequations}
and the other derivations are otherwise unchanged.  
The It\^o adjoint pair $(p^u,q^u)$ then solves the BSDE
\begin{subequations}
\label{eq:matching_adjoint:p:sigma(t):ito}
\begin{align} 
\dd p^u_t
&=
-(\nabla_xb(t,x^u_t)^\top p^u_t-\nabla_x f(t,x^u_t))\dd t
+q^u_t\dd B_t,
\\
p^u_T&=-\nabla g(x^u_T),
\end{align} 
\end{subequations}
which remains computationally challenging compared to the pathwise approach, so it is avoided.
\end{remark}

\begin{remark}[Previous It\^o analysis in {\cite{Domingo2026}}]\label{sec:matching_adjoint:comments}
The analysis in \cite{Domingo2026}  via the It\^o PMP and the MSA inspired the derivation above. 
For a diffusion $\sigma(t)$, the adjoint in \cite[Eq.~(40)--(41)]{Domingo2026} is the random terminal-value ODE \eqref{eq:matching_adjoint:p:ode}, which is well-defined. 
However, the state-dependent extension with $\sigma(t,x)$ is more challenging. \cite[Eqs.~(30)-(31)]{Domingo2026} define the adjoint $a$ through a backward SDE with final condition $a_T=\nabla g(X_T)$ and state that it can be solved pathwise via time reversal, while \cite[Lemma~8]{Domingo2026} calls the adjoint $a$ adapted. These requirements are incompatible: If $\dd X_t=\dd B_t$ and $g(x)=\frac12x^2$, \cite[Eqs.~(30)-(31)]{Domingo2026} or \cite[Eq.~(93)]{Domingo2026} give $a_t=X_T$, which is not $\F_t$-adapted, so the It\^o product calculation in \cite[Eqs.~(117)--(126)]{Domingo2026} is incompatible with the anticipative backward process. Using the It\^o PMP, the correct BSDE would be $\dd a_t=q_t\dd B_t$, so that $q_t=1$ and $a_t=B_t$, %
disagreeing with \cite[Eq.~(93)]{Domingo2026}.
\end{remark}

\begin{remark}[Details on loss function derivation]\label{sec:matching_adjoint_minv}
We derive \eqref{eq:matching_adjoint_minv}. First, note that
\begin{align*}
\left\|v_t-\sigma_t^\top \mathbf p^u_t\right\|^2
&=
\left\|
v_t-\sigma_t^\top \E[\mathbf p^u_t\, |\, \F_t]
-
\sigma_t^\top (\mathbf p^u_t-\E[\mathbf p^u_t\, |\, \F_t])
\right\|^2
\\
&\hspace{-10mm}=
\left\|
v_t-\sigma_t^\top \E[\mathbf p^u_t\, |\, \F_t]
\right\|^2
+
\left\|\sigma_t^\top 
(\mathbf p^u_t-\E[\mathbf p^u_t\, |\, \F_t])
\right\|^2
-2
\left\langle
v_t-\sigma_t^\top \E[\mathbf p^u_t\, |\, \F_t]
,
\sigma_t^\top (\mathbf p^u_t-\E[\mathbf p^u_t\, |\, \F_t])
\right\rangle.
\end{align*}
The last term is zero on average:
\begin{align*}
\E\big[
\left\langle
v_t-\sigma_t^\top \E[\mathbf p^u_t\, |\, \F_t]
,
\sigma_t^\top (\mathbf p^u_t-\E[\mathbf p^u_t\, |\, \F_t])
\right\rangle
\big]
&=
\E\big[
\E\left[
\left\langle
v_t-\sigma_t^\top \E[\mathbf p^u_t\, |\, \F_t]
,
\sigma_t^\top (\mathbf p^u_t-\E[\mathbf p^u_t\, |\, \F_t])
\right\rangle
\, |\, \F_t\right]
\big]
\\
&=
\E\big[
\langle
v_t-\sigma_t^\top \E[\mathbf p^u_t\, |\, \F_t]
,
\underbrace{
\E\left[
\sigma_t^\top (\mathbf p^u_t-\E[\mathbf p^u_t\, |\, \F_t])
\rangle
\, |\, \F_t\right]
}_{=0}
\big].
\end{align*}
So,
\begin{align*}
\E\left[
\left\|v_t-\sigma_t^\top \E[\mathbf p^u_t\, |\, \F_t]\right\|^2
\right]
&=
\E\left[
\left\|
v_t-\sigma_t^\top \mathbf p^u_t
\right\|^2
\right]
-
\E\left[
\left\|\sigma_t^\top 
(\mathbf p^u_t-\E[\mathbf p^u_t\, |\, \F_t])
\right\|^2
\right].
\end{align*}
The last term is independent over the variable $v$, which proves \eqref{eq:matching_adjoint_minv}.
\end{remark}

\subsection{Towards a shooting method}\label{sec:applications:shooting}
We use \PMP to inform the design of indirect shooting methods, extending the shooting method for open-loop problems in \cite{Lew2026}. Like in \cite{Lew2026}, we do not prove guarantees for this method, which should be treated as a heuristic informed by \PMP{}. 
To simplify the presentation, we consider the control-affine problem in \cite{Lew2026} with   unbounded cost, drift, and diffusion maps, but the derivations can be adapted to 
bounded maps. 
Consider the \OCP{}
\begin{align*} 
\min\limits_{u\in \mathcal U} 
\ \E\bigg[
\int_0^T \underbrace{(f_0(t,x_t)+\tfrac r2\|u_t\|^2)}_{=f(t,x_t,u_t)}\dd t+g(x_T)
\bigg] 
\ \, \textrm{s.t.}   
\ \, 
\dd x_t = \underbrace{(\mathbf b_0(t,x_t)+\bar Bu_t)}_{=\mathbf b(t,x_t,u_t)}\dd t +\sigma(t,x_t)\dd\mbB_t,
\ \ 
&x_0=\bar{x}_0,
\end{align*}
where $r>0$ and $\bar x_0\in\R^n$. 
There are no final constraints, so $\mathfrak{p}_0=-1$ and the rough Hamiltonian is
$
\mathbf H(t,x,u,\mathbf p,-1)
=
\mathbf p^\top\mathbf b(t,x,u)-f(t,x,u).
$
Applying \PMP, any optimal solution $(x,u)$ satisfies
\begin{subequations}
\label{eq:shooting:pmp}
\begin{align}
\dd x_t
&=
\mathbf b(t,x_t,u_t)\dd t
+\sigma(t,x_t)\dd\mbB_t,
&&x_0=\bar x_0,
\\
\dd\mathbf p_t
&=
-\nabla_x\mathbf H(t,x_t,u_t,\mathbf p_t,-1)\dd t
-\nabla_x\sigma(t,x_t)^\top\mathbf p_t\dd\mbB_t,
&&\mathbf p_T=-\nabla g(x_T),
\\
p_t&=\E[\mathbf p_t\mid\F_t],
\qquad
u_t
\in
\mathop{\arg\max}_{v\in U}
\mathbf H(t,x_t,v,p_t,-1),
\end{align}
\end{subequations}
where $\mathbf p$ is the anticipative rough adjoint and $p$ is the adapted It\^o adjoint. 
Then, applying the maximality condition gives $
u_t=\mathop{\arg\max}_{v\in U} 
\left(
p_t^\top \bar Bv-\tfrac r2\|v\|^2
\right)$, so
\begin{equation}\label{eq:shooting:u=maximality}
 u_t=\Pi(t,p_t):=\tfrac1r\bar B^\top p_t.
\end{equation}

Assume the deterministic Markovian
representation
\begin{equation}
p_t=\E[\mathbf p_t\mid \F_t]=\Xi(t,x_t).
\label{eq:shooting:markov}
\end{equation}
Substituting 
into \eqref{eq:shooting:pmp}, we obtain the coupled rough SDE 
\begin{equation}
\begin{cases}
\begin{aligned}
\dd x_t
&=
\mathbf b\big(t,x_t,\Pi(t,\Xi(t,x_t))\big)\dd t
+\sigma(t,x_t)\dd\mbB_t,
&&x_0=\bar x_0,
\\
\dd\mathbf p_t
&=
-\nabla_x\mathbf H
 \bigl(t,x_t,\Pi(t,\Xi(t,x_t)),\mathbf p_t,-1\bigr)\dd t
-\nabla_x\sigma(t,x_t)^\top\mathbf p_t\dd\mbB_t,
&&\mathbf{p}_0=z,
\end{aligned}
\end{cases}
\end{equation}
whose solution is denoted by $(x^z,\mathbf p^z)$. 
By the transversality condition, if $(x^z,\mathbf p^z)$ is an optimal solution to \OCP{}, then $\mathbf p^z_T=-\nabla g(x^z_T)$. 
Thus, candidate solutions to \OCP{} are zeros of the map
$$
F: L^2(\Omega,\R^n)\to L^2(\Omega,\R^n), 
\ 
\mathbf p_0\mapsto
F(\mathbf p_0)=\mathbf p_T^{\mathbf p_0}+\nabla g(x^{\mathbf p_0}_T).
$$
Thus, we obtain the root-finding problem $F(\mathbf p_0)=0$.  
There are two main difficulties:
\begin{enumerate}
\item[(1)] The conditional expectation map $\Xi$ is generally
unknown.
\item[(2)] The optimization is over the infinite-dimensional space of $\R^n$-valued random variables $\mathbf p_0$.
\end{enumerate}
Both challenges require future work to obtain a method with guarantees. We give preliminary ideas and empirical results below.

\paragraph{(1) Approximating the conditional expectation and connection to the open-loop problem.} We give an example for the conditional approximation and connect it with the open-loop approach in \cite{Lew2026}. 
Since $\mathbf p_t\in L^2$ and under the Markovian assumption, the conditional expectation satisfies 
\begin{equation}
\Xi(t,x_t)
=
\mathop{\arg\min}_{Z\in L_{\boldsymbol\sigma(x_t)}^2(\Omega,\R^n)}
\E\left[\|\mathbf p_t-Z\|^2\right]
=
\mathop{\arg\min}_{\Xi_t:\R^n\to\R^n}
\E\left[\|\mathbf p_t-\Xi_t(x_t)\|^2\right],
\label{eq:shooting:conditional_projection}
\end{equation}
where $\boldsymbol\sigma(x_t)$ is the $\sigma$-algebra generated by $x_t$, and the class of maps $\Xi_t:\R^n\to\R^n$ depends on the problem. 
For instance, by optimizing over the class of linear maps $\widehat\Xi_t(x)=D_tx$ with diagonal matrices  $D_t=\textrm{diag}(D_{1,t},\dots,D_{n,t})$, we obtain
\begin{equation}
\widehat\Xi(t,x_t)
=
\mathop{\arg\min}_{D_t\in \R^{n\times n}}
\E\left[\|\mathbf p_t-D_tx_t\|^2\right]
\ \implies\ 
D_{j,t}=\E[\mathbf p_{j,t}x_{j,t}]/\E[x_{j,t}^2],
\quad j=1,\dots,n.
\end{equation}
Assuming $n=m$ and the diagonal structure $\bar B=\textrm{diag}(\bar B_1,\dots,\bar B_m)$, the maximality condition \eqref{eq:shooting:u=maximality} gives
\begin{equation}\label{eq:shooting:approx:closedloop}
u_{j,t}=\tfrac1r\bar B_j \widehat\Xi_j(t,x_t)=\frac{\bar B_j}{r}\frac{\E[\mathbf p_{j,t}x_{j,t}]}{\E[x_{j,t}^2]}x_{j,t}.
\end{equation}
We obtained a linear feedback control $
u_{j,t}=k_{j,t}x_{j,t}$ with gains $k_{j,t}=\frac{\bar B_j}{r}\frac{\E[\mathbf p_{j,t}x_{j,t}]}{\E[x_{j,t}^2]}$.

\textit{Connection to open-loop problems}. As in \cite[Section 5]{Lew2026} that only handles  problems with deterministic open-loop controls, consider optimizing over the feedback gains of the control $u=Kx$ by solving%
\begin{align*} 
\min\limits_{k\in \mathcal U^{\text{OL}}} 
\ \E\bigg[
\int_0^T \underbrace{(f_0(t,x_t)+\tfrac r2\|K_tx_t\|^2)}_{=f(t,x_t,u_t)}\dd t+g(x_T)
\bigg] 
\ \, \textrm{s.t.}   
\ \, 
\dd x_t = \underbrace{(\mathbf b_0(t,x_t)+\bar BK_tx_t)}_{=\mathbf b(t,x_t,u_t)}\dd t +\sigma(t,x_t)\dd\mbB_t,
\end{align*}
where  $k\in\U^{\rm OL}:=L^\infty([0,T],\R^m)$ parameterizes  diagonal deterministic gains 
$
K=\textrm{diag}(k_1,\dots,k_m)\in L^\infty([0,T],\R^{m\times m}).
$ 
Then, the open-loop maximality condition \eqref{eq:maximality:openloop} is
\begin{equation}
k_t
=
\mathop{\arg\max}_{k\in U}
\ 
\E\left[
\mathbf p_t^\top \bar BKx_t-\tfrac r2\|Kx_t\|^2
\right]
\ \implies\ 
k_{j,t}=\frac{\bar B_j}{r}\frac{\E[\mathbf p_{j,t}x_{j,t}]}{\E[x_{j,t}^2]},
\quad j=1,\dots,m,
\end{equation}
assuming that $n=m$ and $\bar B$ is diagonal. 

Thus, we recover the same structure 
$
u_{j,t}=k_{j,t}x_{j,t}$ for $j=1,\dots,m$, with the same expression for the gains $k$.
However, the adjoint equations of the two shooting methods are different, which we will show can make an important difference in the numerical robustness of the method. %

\paragraph{(2) Sample average (Monte Carlo) approximation.} As in \cite{Lew2026}, we sample independent rough paths
$\mbB^1,\dots,\mbB^M$, define the shooting variables
$z=(z^1,\ldots,z^M)\in\R^{Mn}$ that parameterize  samples of $\mathbf p_0$, and approximate each expectation with its empirical average. 
Then, we integrate the coupled RDE
\begin{equation}\label{eq:shooting:coupled_rde}
\begin{cases}
\begin{aligned}
\dd x^i_t
&=
\mathbf b\big(t,x^i_t,\Pi(t,\Xi^M(t,x^i_t))\big)\dd t
+\sigma(t,x^i_t)\dd\mbB^i_t,
\\
\dd\mathbf p^i_t
&=
-\nabla_x\mathbf H
 \bigl(t,x^i_t,\Pi(t,\Xi^M(t,x^i_t)),\mathbf p^i_t,-1\bigr)\dd t
-\nabla_x\sigma(t,x^i_t)^\top\mathbf p^i_t\dd\mbB^i_t,
\\
x_0^i&=\bar x_0,
\qquad
\mathbf p_0^i=z^i,
\qquad i=1,\ldots,M,
\end{aligned}
\end{cases}
\end{equation}
with  the sample-based approximation
\begin{equation}
\Xi^M(t,\cdot)
=
\mathop{\arg\min}_{\widehat\Xi_t}
\frac1M\sum_{i=1}^M
\|\mathbf p_t^i-\widehat\Xi_t(x_t^i)\|^2
\label{eq:shooting:conditional_projection:sample}
\end{equation}
over suitable functions $\widehat\Xi_t:\R^n\to\R^n$, 
and consider the shooting map
\begin{equation}
F^M(z)
=
\left(
\mathbf p_T^1+\nabla g(x_T^1),
\dots,
\mathbf p_T^M+\nabla g(x_T^M)
\right)\in\R^{Mn}.
\label{eq:shooting:sample_residual}
\end{equation}
To find zeros of $F^M(\cdot)$, the shooting method iteratively applies the Newton steps
$$
z\gets z-\nabla F^M(z)^{-1}F^M(z)
$$
from an initial guess $z$. Thus, the shooting method sequentially integrates the coupled RDE \eqref{eq:shooting:coupled_rde}, evaluates its sensitivities, and solves the linear system $\nabla F^M(z)\Delta z =F^M(z)$ to update $z\gets z-\Delta z$.

\paragraph{Numerical example.} We evaluate the shooting method using approximations (1) and (2) above. We consider the OCP in \cite[Section 5]{Lew2026}, which is an instance of the problem above with $n=m=3$, $f_0(t,x_t)=\tfrac{10}{2}\|x\|^2$, $r=3$, $\mathbf b_0(x)=-J^{-1}S(x)Jx$ with $ 
S(x)=
\SmallMatrix{
0&-x_3&x_2
\\
x_3&0&-x_1
\\
-x_2&x_1&0
}$ and  
$J=\text{diag}(J_1,J_2,J_3)=\text{diag}(3,2,4)=\bar{B}^{-1}$, 
$\sigma(t,x)=0.4\,\textrm{diag}(x)$, 
$x_0=\frac{\pi}{180}(-1, -4.5, 4.5)$,   and  
$B$ a  standard $n$-dimensional Brownian motion. 
We compare two sample-based shooting methods.

\textbf{Feedback shooting method}: 
The first is the shooting method  proposed above, which approximates the conditional expectation as $\widehat\Xi_j(t,x)=D_{j,t}x_{j,t}$. 
Its adjoint equation is
$$
\dd\mathbf p_t=-\left(\nabla\mathbf b_0(t,x_t)^\top \mathbf p_t - \nabla_x f_0(t,x_t)
\right)
\dd t
-\nabla_x\sigma(t,x_t)^\top\mathbf p_t\dd \mbB_t.
$$

\textbf{Open-loop shooting method}: The second is the open-loop shooting method proposed in \cite[Section 5]{Lew2026}, which optimizes over the open-loop gains $k$ and returns the closed-loop controls $u_{j,t}(\omega)=k_{j,t}x_{j,t}(\omega)$, as  described above. Its adjoint equation is
$$
\dd\mathbf p_t=-\left(\nabla\mathbf b_0(t,x_t)^\top \mathbf p_t - \nabla_x f_0(t,x_t)
+
K_t^\top(\bar B^\top \mathbf p_t-rK_tx_t)
\right)
\dd t
-\nabla_x\sigma(t,x_t)^\top\mathbf p_t\dd \mbB_t.
$$
The only difference between the two methods is the term $K_t^\top(\bar B^\top \mathbf p_t-rK_tx_t)$ in the adjoint equation. 
\cite[Section 5]{Lew2026} reports numerical sensitivity to the choice of initial guess $\mathbf p_0$. %
Our hypothesis is that the new feedback shooting method is more robust  since its adjoint equation  does not include this extra term. 

\textbf{Details}: We follow \cite[Section 5]{Lew2026}. We use a zero initial guess $\mathbf p_0^i=0$ and run each Newton method until either $\|F(\{\mathbf p_0^i\}_{i=1}^M)\|_\infty<10^{-12}$ or the method diverges. We use $T=2$ and discretize with a Milstein scheme and $N=40$ nodes.  
We directly solve OCP for $r$ and do not use a homotopy method. 
The code extends the code used in \cite{Lew2026}, and is available at \url{github.com/ToyotaResearchInstitute/rspmp}. 

\textbf{Results}: We test the two shooting methods  for $M=10$ and varying $r\in\{0.5, \dots,5\}$, and for $r=3$ and varying sample sizes $M$. %
We repeat each experiment $100$ times for new samples of $\mbB$. 
  
 First, we find that whenever both methods converge, they return the same solution up to numerical accuracy. 
 An example of solution for $r=3$ is shown in Figure \ref{fig:applications:shooting:trajs}, which corresponds to \cite[Figure 1.1]{Lew2026}. 
 
 Second, we observe that the feedback shooting method using the conditional bridge converges more reliably than the open-loop shooting method \cite{Lew2026} when varying the value of $r$. Results in Figure \ref{fig:applications:shooting:eval} (left, $M=10$) show that the open-loop method is sensitive to the value of $r$, as observed in \cite{Lew2026}, with poor convergence rates for small values of $r$, and higher convergence rates as $r$ increases. In contrast, the proposed feedback shooting method is more robust for  varying control penalizations $r$, which might be due to the simpler adjoint equation.  Results in Figure \ref{fig:applications:shooting:eval} (right, $r=3$) show that the cost of the returned approximate solution (evaluated via $10^5$ simulations) decreases and plateaus as the sample size $M$ increases, 
 suggesting the stability of the sample average approximation as $M$ increases.

\begin{figure}[t]
\centering
\includegraphics[width=0.99\textwidth,trim={0 8mm 0 10mm},clip]
{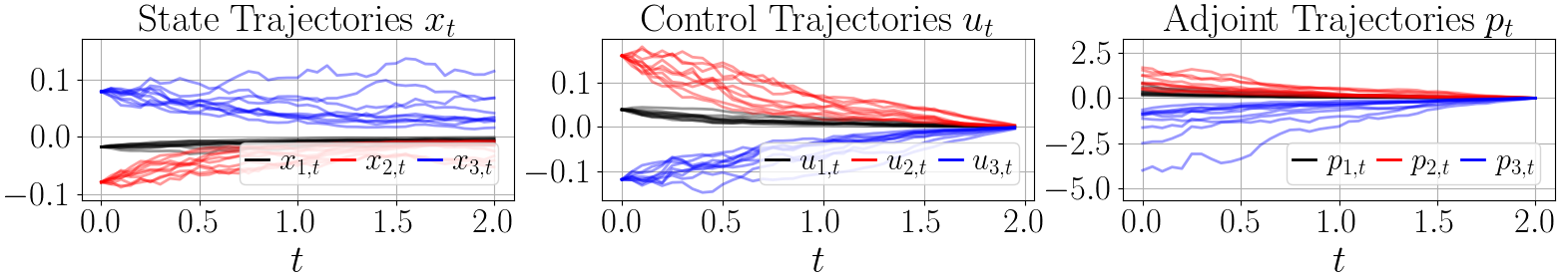}
\vspace{-3mm}
\caption{Ten sample paths of the approximate solutions to the optimal control problem in Section \ref{sec:applications:shooting}.}
\label{fig:applications:shooting:trajs}
\vspace{3mm}
\centering
\includegraphics[width=0.475\textwidth]
{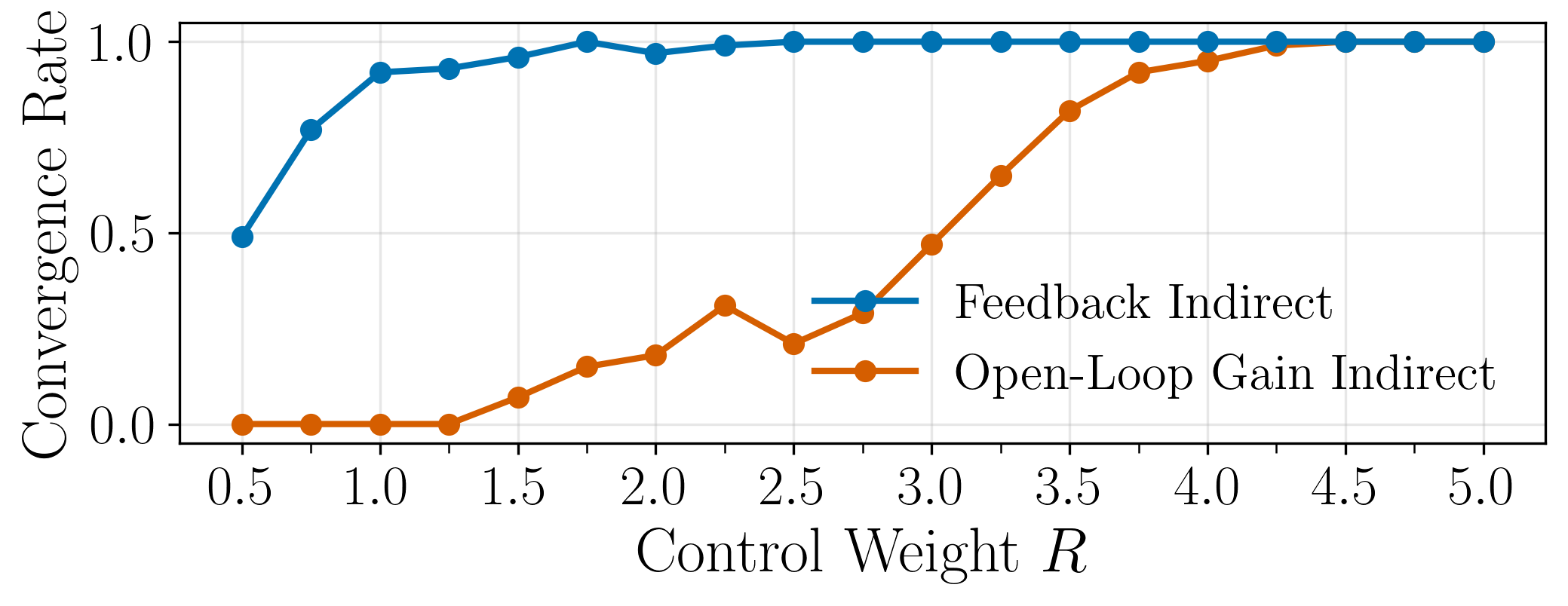}
\hspace{0.02\textwidth}
\includegraphics[width=0.475\textwidth]
{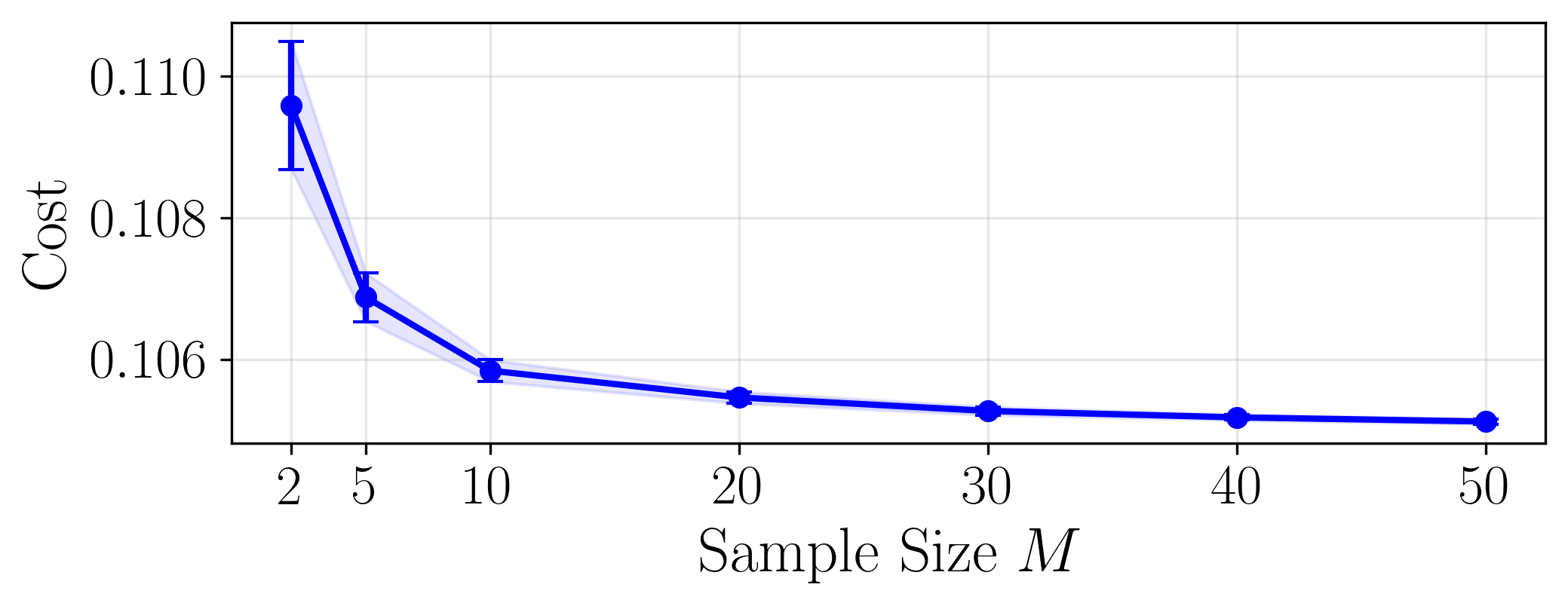}
\vspace{-4mm}
\caption{Left: Convergence rate of the two shooting methods ($1$ = all converged, $0$ = all diverged). Right: Cost of the returned  solution when varying the sample size $M$, with $\pm \,95\%$ confidence intervals.}
\vspace{-2mm}
\label{fig:applications:shooting:eval}
\end{figure}

\appendix
\addtocontents{toc}{\protect\setcounter{tocdepth}{1}}
\section{Proofs of It\^o-Stratonovich Conversion Results} \label{sec:appendix:additional_proofs}
\subsection{It\^o-Stratonovich conversion}\label{sec:ito_stratonovich_conversion:details}
\begin{proof}[Proof of Proposition \ref{prop:ito_stratonovich_conversion}]
Let $x$ solve the It\^o SDE. 
By It\^o's formula applied  to $(t,x)\mapsto \sigma^{ij}(t,x)$ and $x$,
\begin{align*}
\sigma^{ij}(t,x_t)
&=
\sigma^{ij}(0,x_0)
+
\int_0^t\frac{\partial\sigma^{ij}}{\partial t}(s,x_s)\dd s 
+
\sum_{k=1}^n\int_0^t\frac{\partial\sigma^{ij}}{\partial x^k}(s,x_s)\dd x^k_s + 
\int_0^t
\frac12\sum_{k,\ell=1}^n\frac{\partial^2\sigma^{ij}}{\partial x^k\partial x^\ell}(s,x_s)\dd[x^k,x^\ell]_s 
\\
&=
\sigma^{ij}(0,x_0)
+
\int_0^t\frac{\partial\sigma^{ij}}{\partial t}\dd s 
+
\sum_{k=1}^n\int_0^t\frac{\partial\sigma^{ij}}{\partial x^k}
\left(b^k_s\dd s + \sigma^{k\boldsymbol\cdot}_s\dd B_s\right) + 
\int_0^t
\frac12\sum_{k,\ell=1}^n\frac{\partial^2\sigma^{ij}}{\partial x^k\partial x^\ell}\dd[x^k,x^\ell]_s .
\end{align*}
Combining this equation with \cite[Equation (5.2)]{LeGall2016}, we obtain
\begin{align*}
\left[\sigma^{ij}(\cdot,x_\cdot),B^j\right]_t
&=
\left[
\left(
\sigma^{ij}(0,x_0)
+
\sum_{k=1}^n\int_0^\cdot\frac{\partial\sigma^{ij}}{\partial x^k}(s,x_s)
\sigma^{k\boldsymbol\cdot}_s\dd B_s + 
\dots\right),\ B^j\right]_t
=
\int_0^t
\sum_{k=1}^n
\frac{\partial\sigma^{ij}}{\partial x^k}(s,x_s)
\sigma^{kj}_s\dd s,
\end{align*}
where every other term is zero since higher-order terms drop and $[B^\ell,B^j]_t=\delta_{\ell j}t$. Thus, for $i=1,\dots,n$,
\begin{align*}
\int_0^t\sigma^i(s,x_s)\circ\dd B_s
&=
\int_0^t\sigma^i(s,x_s)\dd B_s
+\frac12\sum_{j=1}^d
[\sigma^{ij}(\cdot,x),B^j]_t
\\
&=
\int_0^t\sigma(s,x_s)\dd B_s
+\frac12\int_0^t
\nabla_x\sigma(s,x_s)\sigma(s,x_s)\dd s
\\
&=
\int_0^t\sigma(s,x_s)\dd B_s
+\int_0^t \left(
b(s,x_s,u_s)-\mathbf b(s,x_s,u_s)
\right)\dd s.
\end{align*}
Thus, $x$ also solves the Stratonovich SDE  \eqref{eq:background:stratonovich_sde}. 
Conversely, suppose that $x$ solves \eqref{eq:background:stratonovich_sde}. Then,
\begin{align*}
x_t
&=
\bar x_0
+\int_0^t\mathbf b(s,x_s,u_s)\dd s
+\sum_{j=1}^d\int_0^t\sigma^j(s,x_s)\dd B^j_s
+\frac12\sum_{j=1}^d
[\sigma^j(\cdot,x),B^j]_t.
\end{align*}
Following similar derivations yields the It\^o SDE.
\end{proof}

\subsection{Forward tangent processes}
\label{sec:tangent:strat_to_ito_conversion:details}
We verify the identity
\begin{equation}
\tag{\ref{eq:tangent:strat_to_ito_conversion}}
\big(\nabla_x \Sigma^j\,\sigma^j\big)v+\Sigma^j\Sigma^j v
=
\nabla_x(\Sigma^j\sigma^j)v,
\end{equation}
where dependencies on $(t,x)$ are omitted for conciseness.
First, the product rule gives
\begin{align*} 
\nabla_x(\Sigma^j\sigma^j)v
&=
\big(\nabla_x \Sigma^jv\big)\sigma^j
+
\Sigma^j(\nabla_x\sigma^j)v
=
\big(\nabla_x \Sigma^jv\big)\sigma^j
+
\Sigma^j\Sigma^j v.
\end{align*}
It remains to justify
$$
\big(\nabla_x \Sigma^j\sigma^j\big)v
=
\big(\nabla_x \Sigma^j\big)\sigma^j,
$$
which we check componentwise. We write
$
\Sigma^j_{k\ell}(t,x):=\partial_{x_\ell}\sigma^j_k(t,x)$. 
For each $k$,
\begin{align*}
\left[
\bigl(\nabla_x \Sigma^j\sigma^j\bigr)v
\right]_k
&=
\sum_{m,\ell}
\partial_{x_m}\Sigma^j_{k\ell}\sigma^j_m\,v_\ell
\qquad\text{and}\qquad
\left[
\big(\nabla_x \Sigma^j\,v\big)\sigma^j
\right]_k
=
\sum_{\ell,m}
\partial_{x_\ell}\Sigma^j_{km}v_\ell\sigma^j_m.
\end{align*}
Since $\Sigma^j_{k\ell}=\partial_{x_\ell}\sigma^j_k$, symmetry of second derivatives gives
$$
\partial_{x_m}\Sigma^j_{k\ell}
=
\partial_{x_m}\partial_{x_\ell}\sigma^j_k
=
\partial_{x_\ell}\partial_{x_m}\sigma^j_k
=
\partial_{x_\ell}\Sigma^j_{km}.
$$
Therefore, for every $k$,
$$
\left[
\bigl(\nabla_x \Sigma^j\,\sigma^j\bigr)v
\right]_k
=
\sum_{m,\ell}
\partial_{x_\ell}\Sigma^j_{km}\,\sigma^j_m\,v_\ell
=
\left[
\bigl(\nabla_x \Sigma^j\,v\bigr)\sigma^j
\right]_k.
$$
Hence, 
$
\bigl(\nabla_x \Sigma^j\,\sigma^j\bigr)v
=
\bigl(\nabla_x \Sigma^j\,v\bigr)\sigma^j
$. 
Substituting this identity into the product-rule expansion above yields
$$
\nabla_x(\Sigma^j\sigma^j)v
=
\bigl(\nabla_x \Sigma^j\,\sigma^j\bigr)v
+
\Sigma^j\Sigma^j v,
$$
which is exactly \eqref{eq:tangent:strat_to_ito_conversion}.

\section{Proof of the Itô Stochastic PMP}\label{sec:pmp_ito} 

We now prove the It\^o PMP (Theorem \ref{thm:pmp:ito}),   to highlight similarities and differences with the proof of the rough PMP (Theorem \ref{thm:pmp:rough}) in Section \ref{sec:pmp_rough}, and  for completeness. We consider a scalar Brownian motion ($d=1$) for conciseness, but the proof extends by replacing instances of
$
\sigma_t\dd B_t$ with $\sum_{j=1}^d\sigma_t^j\dd B_t^j$. 
The proof follows the classical needle variations and endpoint separation strategy, as in the proof in \cite{BonalliLewESAIM2022}.

\begin{proof}[Proof of the It\^o PMP (Theorem \ref{thm:pmp:ito})]

Throughout this section,  all the SDEs have well-defined solutions thanks to Theorem \ref{thm:sdes:ito}.

\textbf{Preliminary step: Cost-augmented system, needle variations, endpoint mapping.}

\textit{Cost-augmented system}: 
Define the cost-augmented state $$\widetilde x=\begin{bmatrix}
x \\ x^0
\end{bmatrix}
\in L^2_{\F}\!\left(\Omega,C\!\left([0,T],\R^{n+1}\right)\right)$$ with SDE dynamics
\begin{align}
\label{eq:cost_augmented_state:ito}
\dd \widetilde x_t
&=
\widetilde b(t,\widetilde x_t,u_t)\dd t
+
\widetilde \sigma(t,\widetilde x_t)\dd B_t,
\hspace{9mm} 
\widetilde x_0=(\bar x_0,0),
\end{align}
where
\begin{gather} 
\widetilde b(t,\widetilde x,u):=\begin{bmatrix}
b(t,x,u)
\\
f(t,x,u)
\end{bmatrix},
\qquad
\widetilde \sigma(t,\widetilde x):=\begin{bmatrix}
\sigma(t,x)
\\
0
\end{bmatrix},
\\
\widetilde A_t:=\nabla_{\widetilde x}\widetilde b(t,\widetilde x_t,u_t)
=
\begin{bmatrix}
\nabla_x b(t,x_t,u_t) & 0\\
\nabla_x f(t,x_t,u_t) & 0
\end{bmatrix},
\qquad
\widetilde\Sigma_t:=\nabla_{\widetilde x}\widetilde\sigma(t,\widetilde x_t)
=
\begin{bmatrix}
\nabla_x\sigma(t,x_t) & 0\\
0 & 0
\end{bmatrix}.
\end{gather} 

\textit{Needle variations}:  
Let $u\in L^2([0,T]\times\Omega,U)$ and $V\subset U$ be a countable dense subset.  
We say that $t_1\in[0,T]$ is a \textit{Lebesgue time}  if  almost surely,
\begin{subequations}
\begin{align}
\tag{\ref{eq:lebesgue_time:u}}
\lim_{h\to0}\frac{1}{h}\int_{t_1}^{t_1+h}
\widetilde b(t,\widetilde x_t(\omega),u_t(\omega))\dd t
&= 
\widetilde b(t_1,\widetilde x_{t_1}(\omega),u_{t_1}(\omega)), 
\quad \text{ and}
\\
\tag{\ref{eq:lebesgue_time:vn}}
\lim_{h\to0}\frac{1}{h}\int_{t_1}^{t_1+h}
\widetilde b(t,\widetilde x_t(\omega),v_n)\dd t 
&=
\widetilde b(t_1,\widetilde x_{t_1}(\omega),v_n)
\ \ \text{for all $v_n\in V$.}
\end{align}
\end{subequations}
almost surely.\footnote{In contrast, \cite{BonalliLewESAIM2022} defines a Lebesgue time as satisfying $\lim_{h\to 0}\E[
\frac1h\int_{t_1}^{t_1+h}\left\|
u_t-u_{t_1}
\right\|^2
\dd t 
]=0$  and using Bochner integrals. Our pathwise definition implies the same limit property, but helps when working  with rough SDEs and deriving error bounds like for Lemma \ref{lem:needle:rough} via a rough path approach.} 
Almost every $t\in[0,T]$ is a Lebesgue time, see Section \ref{sec:pmp_rough}.  
Any Lebesgue time $t_1$ satisfies 
\begin{align*}
\E\left[
\frac1h\int_{t_1}^{t_1+h}\left\|
\widetilde b(t,\widetilde x_t,u_t)-\widetilde b(t_1,\widetilde x_{t_1},u_{t_1})
\right\|^2
\dd t 
\right]
&\leq
C_b
\E\left[
\frac1h\int_{t_1}^{t_1+h}\left\|
\widetilde b(t,\widetilde x_t,u_t)-\widetilde b(t_1,\widetilde x_{t_1},u_{t_1})
\right\|
\dd t
\right]
 \mathop{\longrightarrow}\limits_{h\to 0} 0,
\end{align*}
since $b$ is bounded and 
by the dominated convergence theorem.

Let $0<t_1<\cdots<t_q<T$ 
be Lebesgue times of the optimal control $u\in\U$. For $i=1,\dots,q$, let 
$$
\bar u_i\in L^2_{\F_{t_i}}(\Omega,U),
\qquad
0\leq\eta_i<t_{i+1}-t_i,
$$
with $t_{q+1}:=T$. Define the needle variation
$\pi=\{t_i,\eta_i,\bar u_i\}_{i=1}^q$ of $u$ by
$$
u_t^\pi
=
\begin{cases}
\bar u_i, & t\in[t_i,t_i+\eta_i],\\
u_t, & \text{otherwise}.
\end{cases}
$$

Let  $\widetilde x^\pi\in L^2_\F(\Omega,C([0,T],\R^{n+1}))$ be the solution to the SDE
\begin{align}
\dd \widetilde x^\pi_t
&=
\widetilde b(t,\widetilde x^\pi_t,u^\pi_t)\dd t
+
\widetilde \sigma(t,\widetilde x^\pi_t)\dd B_t,
\hspace{9mm} 
\widetilde x^\pi_0=(\bar x_0,0),
\end{align}
that corresponds to the control $u^\pi$. 

For each $i=1,\dots,q$, 
let $\widetilde v^{\pi_i}\in L^2_\F(\Omega,C([t_i,T],\R^{n+1}))$ be the solution to the linearized SDE
\begin{equation}
\label{eq:linearized_variation:ito} 
\dd \widetilde v_t^{\pi_i}
=
\widetilde A_t \widetilde v_t^{\pi_i}\dd t
+
\widetilde\Sigma_t\widetilde v_t^{\pi_i}\dd B_t,
 \quad\ 
 \widetilde v_{t_i}^{\pi_i}
=
\Delta\widetilde b_i
:=
\widetilde b(t_i,\widetilde x_{t_i},\bar u_i)
-
\widetilde b(t_i,\widetilde x_{t_i},u_{t_i}).
\end{equation}
By Theorem \ref{thm:sdes:ito}, the solution satisfies
\begin{equation}
\widetilde v_t^{\pi_i} = 
\widetilde\phi_t\widetilde\psi_{t_1}\widetilde v^{\pi_i}_{t_i}
\quad\text{ for any $t\in[t_1,T]$,}
\end{equation}
where $\widetilde\phi,\widetilde\psi\in L^2_\F(\Omega,C([0,T],\R^{n\times n}))$  with $\phi_t=\psi^{-1}_t$  solve the matrix-valued SDEs
\begin{equation}
\dd \widetilde\phi_t=
\widetilde A_t\widetilde\phi_t\dd t+\widetilde\Sigma_t\widetilde\phi_t\dd B_t, 
\ \  
\phi_0=I,
\qquad\quad
\dd \widetilde\psi_t=-\widetilde\psi_t(\widetilde A_t-\widetilde\Sigma_t^2)\dd t-\widetilde\psi_t\widetilde\Sigma_t\dd B_t,
\ \ 
\psi_0=I.
\end{equation}
By the sparse structure of the matrices $\widetilde A$ and $\widetilde\Sigma$,
\begin{equation}\label{eq:ito_sde:matrix:sparsity}
\widetilde\phi_t
=
\begin{bmatrix}
\phi_t & 0\\
\zeta_t & 1
\end{bmatrix},
\qquad
\widetilde\psi_t
=
\begin{bmatrix}
\psi_t & 0\\
\rho_t & 1
\end{bmatrix},
\qquad
\rho_t=-\zeta_t\psi_t,
\end{equation}
where
$$
\begin{cases}
\dd\phi_t=A_t\phi_t\dd t+\Sigma_t\phi_t\dd B_t,
\\
\phi_0=I_n,
\end{cases}
\qquad
\begin{cases}
\dd\zeta_t=a_t\phi_t\dd t,
\\
\zeta_0=0,
\end{cases}
\qquad
\begin{cases}
\dd\psi_t
=
-\psi_t(A_t-\Sigma_t^2)\dd t
-
\psi_tC_t\dd B_t,
\\
\psi_0=I_n.
\end{cases}
$$

Next, define the map
\begin{equation}
\label{eq:endpoint:ito_gradient}
\Phi:\mathbb R^{n+1}\to\mathbb R^{r+1},
\ 
\widetilde x\mapsto\Phi(\widetilde x)
=
\begin{bmatrix}
h(x)\\
x^0+g(x)
\end{bmatrix},
\quad\text{with}\ 
\nabla\Phi(\widetilde x)
=
\begin{bmatrix}
\nabla h(x) & 0\\
\nabla g(x) & 1
\end{bmatrix},
\end{equation}
that evaluates the terminal constraints and total cost. 
\begin{lemma}[Needle variations for It\^o SDEs {\cite[Lemma 4.2]{BonalliLewESAIM2022}}]
\label{lem:needle:ito}
Under the assumptions of the It\^o PMP (Theorem \ref{thm:pmp:ito}), define the needle variation $\pi=\{t_i,\eta_i,\bar u_i\}_{i=1}^q$, state trajectories $(\widetilde x,\widetilde x^\pi)$, linear variations $\{\widetilde v^{\pi_i}\}_{i=1}^q$, and $\Phi$ as above.  
Then, %
\begin{equation}
\label{eq:needle:ito}
\E\left[
\left\|
\Phi(\widetilde x_T^\pi)
-
\Phi(\widetilde x_T)
-
\sum_{i=1}^q
\eta_i \nabla\Phi(\widetilde x_T)\widetilde v_T^{\pi_i}
\right\|
\right]
=o\left(\sum_{i=1}^q\eta_i\right).
\end{equation}
\end{lemma}

\textit{Endpoint mapping}:  
For $\eta\in\R^q_+\cap B^q_\delta$, define the endpoint map
\begin{equation}
\label{eq:endpoint:ito}
F(\eta)
:=
\E\left[
\Phi(\widetilde x_T^\pi)
-
\Phi(\widetilde x_T)
\right],
\end{equation}
which satisfies $F(0)=0$. 
Thanks to Lemma \ref{lem:needle:ito}, $F$ is
Gateaux differentiable at $0$ in directions $\eta\in\mathbb R_+^q$, with Gateaux differential
\begin{equation}
\label{eq:endpoint_differential:ito}
\dd F_0(\eta)
=
\sum_{i=1}^q
\eta_i
\E\left[
\nabla\Phi(\widetilde x_T)
\widetilde v_T^{\pi_i}
\right].
\end{equation}

\textbf{Step 1: Variational linearization and separation argument.}  This step is identical to Step 1 of the proof of the rough PMP in Section \ref{sec:pmp_rough},  using a fixed point argument and the endpoint map $F$.
By this argument, there exists a non-zero vector $\mathfrak{p}=(\mathfrak{p}_1,\dots,\mathfrak{p}_r,\mathfrak{p}_0)\in\R^{r+1}$   such that $\mathfrak{p}_0\in\{0,-1\}$, and
\begin{equation}\label{eq:spmp:ito:proof:mu^TEgradPhiv<=0}
\mathfrak{p}^\top\,
\E\left[
\nabla\widetilde\Phi(\widetilde{x}_T)
\widetilde{v}^{\pi_1}_T
\right]
\leq 
0
\ \ 
\text{for any needle variations $\pi_1=(t_1,\eta_1,\bar{u}_1)$}.
\end{equation}

\textbf{Step 2: Adjoint equation and transversality condition.}
Define
\begin{equation}
\label{eq:adjoint_augmented:ito}
\widetilde{p}_t^\top
:=
\begin{bmatrix}
p_t
\\
p^0_t
\end{bmatrix}^\top
:= 
\E\left[
\mathfrak{p}^\top
\nabla\widetilde\Phi(\widetilde{x}_T)
\widetilde\phi_T\widetilde\psi_t
\,\Big|\,\F_t
\right]. 
\end{equation}
The transversality condition follows directly using $\widetilde\phi_t=\widetilde\psi_t^{-1}$:
$$
\widetilde{p}_T^\top
=
\E\left[
\mathfrak{p}^\top
\nabla\widetilde\Phi(\widetilde{x}_T)
\widetilde\phi_T\widetilde\psi_T
\,\Big|\,\F_T
\right]
=
\mathfrak{p}^\top
\nabla\widetilde\Phi(\widetilde{x}_T)
=
\begin{bmatrix}
\mathfrak p_0\nabla g(x_T)
+
\sum_{i=1}^{r}\mathfrak p_i\nabla h_i(x_T)
&
\mathfrak p_0
\end{bmatrix}.
$$
Also, the last component $p_t^0$ is constant. Indeed, by the sparsity structure of   $(\widetilde\phi,\widetilde\psi)$ and $\nabla\widetilde\Phi$, we have
$$
p_t^0
=
\widetilde p_t^\top e^{n+1}
\mathop{=}^{\eqref{eq:adjoint_augmented:ito}}
\E\left[
\mathfrak{p}^\top
\nabla\widetilde\Phi(\widetilde{x}_T)
\widetilde\phi_T\widetilde\psi_T
e^{n+1}
\,\Big|\,\F_t
\right]
\mathop{=}^{\eqref{eq:ito_sde:matrix:sparsity}}
\E\left[
\mathfrak{p}^\top
\nabla\widetilde\Phi(\widetilde{x}_T)
e^{n+1}
\,\Big|\,\F_t
\right]
\mathop{=}^{\eqref{eq:endpoint:ito_gradient}}
\E\left[
\mathfrak{p}^\top
e^{r+1}
\,\Big|\,\F_t
\right]
=
\mathfrak p_0,
$$
for all $t\in[0,T]$ almost surely, where $e^k:=(0,\dots,0,1)\in\R^k$.
Next, define
$$
M_t:=
\E\left[
\mathfrak{p}^\top
\nabla\widetilde\Phi(\widetilde{x}_T)
\widetilde\phi_T
\,\Big|\,\F_t
\right].
$$ 
The process 
$M_t$ satisfies $\E[\int_0^T\|M_t\|^2\dd t]<\infty$  and is a martingale. Thus, by the martingale representation theorem, there is a process $\mu\in L^2_\F(\Omega\times[0,T],\R^{(n+1)\times d})$ such that $M_t$ can be represented as
$$
M_t
\mathop{=}^{\eqref{eq:martingale_representation}}
M_0
+\int_0^t\mu_s\dd B_s
=
N+\chi_t,
$$
where $N:=M_0=\E[
\mathfrak{p}^\top
\nabla\widetilde\Phi(\widetilde{x}_T)
\widetilde\phi_T
\,|\,\F_0
]$ and $\chi_t:=\int_0^t\mu_s\dd B_s$. Then, 
\begin{equation}
\label{eq:adjoint:ito:decomposition}
\widetilde p_t^\top=M_t\widetilde\psi_t=(N+\chi_t)\widetilde\psi_t.
\end{equation}
Also, by   It\^o's formula,  
$$
\chi_t\widetilde\psi_t
\mathop{=}^{\eqref{eq:ito_formula}}
-
\int_0^t
\left(
\chi_s\widetilde\psi_s(\widetilde A_s-\widetilde\Sigma_s^2)
+
\mu_s\widetilde\psi_s\widetilde\Sigma_s
\right)
\dd s
+
\int_0^t
\left(
\mu_s\widetilde\psi_s 
-
\chi_s\widetilde\psi_s\widetilde\Sigma_s
\right)
\dd B_s.
$$
So,
\begin{align*}
\dd\widetilde{p}_t^\top
&\mathop{=}^{\eqref{eq:adjoint:ito:decomposition}}
N
\dd\widetilde\psi_t
+
\dd(\chi\widetilde\psi)_t
\\
&=
-\left(
N\widetilde\psi_t(
\widetilde A_t-\widetilde\Sigma_t^2
)
+\chi_t\widetilde\psi_t(\widetilde A_t-\widetilde\Sigma_t^2)
+\mu_t\widetilde\psi_t\widetilde\Sigma_t
\right)
\dd t
+
\left(
-N\widetilde\psi_t\widetilde\Sigma_t 
+
\mu_t\widetilde\psi_t
-
\chi_t\widetilde\psi_t\widetilde\Sigma_t
\right)
\dd B_t
\\
&=
-\left(
(N+\chi_t)\widetilde\psi_t
\widetilde A_t
-
(N+\chi_t)\widetilde\psi_t
\widetilde\Sigma_t^2
+
\mu_t\widetilde\psi_t\widetilde\Sigma_t
\right)
\dd t
+
\left(
\mu_t\widetilde\psi_t
-(N
+
\chi_t)\widetilde\psi_t
\widetilde\Sigma_t
\right)
\dd B_t
\\
&\mathop{=}^{\eqref{eq:adjoint:ito:decomposition}}
-\left(
\widetilde{p}^\top_t\widetilde A_t
+
(
-
\widetilde p_t^\top\widetilde\Sigma_t
+
\mu_t\widetilde\psi_t
)
\widetilde\Sigma_t 
\right)
\dd t
+
\Big(
\mu_t\widetilde\psi_t
-
\widetilde{p}^\top_t\widetilde\Sigma_t
\Big)
\dd B_t.
\end{align*}
Define
$
\widetilde q_t^\top
=
\begin{bmatrix}
q_t\\
q_t^0
\end{bmatrix}:=
\left(
-
\widetilde p_t^\top\widetilde\Sigma_t
+
\mu_t\widetilde\psi_t
\right)
$. 
Then, using the sparsity of $(\widetilde A,\widetilde\Sigma)$, 
\begin{align*}
\dd\widetilde{p}_t 
&=
-\Big(
\widetilde A_t^\top \widetilde{p}_t
+
\widetilde\Sigma_t^\top
\widetilde{q}_t
\Big)
\dd t
+
\widetilde{q}_t
\dd B_t
\\
&=
-\begin{bmatrix}
\nabla_x b(t,x_t,u_t)^\top p_t
+
\nabla_x f(t,x_t,u_t) p_t^0
+
\nabla_x\sigma(t,x_t)^\top q_t
\\
0
\end{bmatrix}
\dd t
+
\begin{bmatrix}
q_t\\
q_t^0
\end{bmatrix}
\dd B_t
\end{align*} 
with $p_t^0=\mathfrak p_0$, which gives the adjoint equation.

\textbf{Step 3: Maximality condition.} 
Let $\pi_1=\{t_1,\eta_1,\bar{u}_1\}$ be any needle variation of the optimal control $u$, and
$$
\widetilde v_T^{\pi_1}
=
\widetilde\phi_T\widetilde\psi_{t_1}
\Delta\widetilde b_1,
\quad
\text{with }
\Delta\widetilde b_1
:=
\widetilde b(t_1,\widetilde x_{t_1},\bar u_1)
-
\widetilde b(t_1,\widetilde x_{t_1},u_{t_1}),
$$
be the solution to the linearized SDE \eqref{eq:linearized_variation:ito}.
By the separation inequality and the definition of $\widetilde p_{t_1}$,
$$
0
\mathop{\geq}^{\eqref{eq:spmp:ito:proof:mu^TEgradPhiv<=0}}
\mathfrak p^\top
\E\left[
\nabla\widetilde\Phi(\widetilde x_T)\widetilde v_T^{\pi_1}
\right]
=
\E\left[
\mathfrak p^\top
\nabla\widetilde\Phi(\widetilde x_T)
\widetilde\phi_T\widetilde\psi_{t_1}
\Delta\widetilde b_1
\right]
=
\E\left[
\E\left[
\mathfrak p^\top
\nabla\widetilde\Phi(\widetilde x_T)
\widetilde\phi_T\widetilde\psi_{t_1}
\,\Big|\,\F_{t_1}
\right]
\Delta\widetilde b_1
\right]
\mathop{=}^{\eqref{eq:adjoint_augmented:ito}}
\E\left[
\widetilde p_{t_1}^\top\Delta\widetilde b_1
\right]\!.
$$
Therefore,
$$
\E\left[
p_{t_1}^\top
\big(
b(t_1,x_{t_1},\bar u_1)-b(t_1,x_{t_1},u_{t_1})
\big)
+
\mathfrak p_0
\big(
f(t_1,x_{t_1},\bar u_1)-f(t_1,x_{t_1},u_{t_1})
\big)
\right]\leq 0.
$$
Equivalently,
\begin{equation}
\label{eq:hamiltonian_variational_ineq:ito}
\E\left[
H(t_1,x_{t_1},\bar u_1,p_{t_1},q_{t_1},\mathfrak p_0)
-
H(t_1,x_{t_1},u_{t_1},p_{t_1},q_{t_1},\mathfrak p_0)
\right]\leq 0,
\end{equation}
which holds for any Lebesgue time $t_1$  and $\bar u_1\in L^2_{\F_{t_1}}(\Omega,U)$.

The conclusion follows the  proof of the rough PMP. By contradiction, suppose that the maximality condition does not hold on a set of positive
$(\dd t\otimes\Prob)$-measure. 
Then, by the argument in Remark \ref{remark:pmp:rough:localization}, there exist a Lebesgue point
$t_1$, a set  $A_{t_1}\in\F_{t_1}$ with $\Prob(A_{t_1})>0$, and  $\widehat u_1\in L^2_{\F_{t_1}}(\Omega,U)$, such that
\begin{equation}\label{eq:hamiltonian_variational_ineq:ito:con}
H(t_1,x_{t_1},\widehat u_1,p_{t_1},q_{t_1},\mathfrak p_0)
>
H(t_1,x_{t_1},u_{t_1},p_{t_1},q_{t_1},\mathfrak p_0)
\quad\text{on }A_{t_1}.
\end{equation}
Define
$\bar u_1 := \widehat u_1\mathbf 1_{A_{t_1}} + u_{t_1}\mathbf 1_{A_{t_1}^c}$. 
Then, $\bar u_1\in L^2_{\F_{t_1}}(\Omega,U)$, and
$$
\E\left[
H(t_1,x_{t_1},\bar u_1,p_{t_1},q_{t_1},\mathfrak p_0)
-
H(t_1,x_{t_1},u_{t_1},p_{t_1},q_{t_1},\mathfrak p_0)
\right]
>0,
$$
which contradicts  \eqref{eq:hamiltonian_variational_ineq:ito}. Therefore, the maximality condition holds, which concludes the proof of the It\^o PMP.
\end{proof}

\section{Proofs for Rough SDEs}\label{sec:proofs:rough_paths} 
The background material on rough SDEs (Section \ref{sec:background:rough_sdes}) and corresponding needle variations error bounds  (Lemma \ref{lem:needle:rough}) are non standard, due to the random control $u\in L^2_{\F}([0,T]\times\Omega,U)$ that requires careful derivations to prove existence, uniqueness, measurability, and integrability of solutions. These results follow from existing results in stochastic control and rough path theory. We describe how to derive them by adapting the results in \cite{Lew2026} that  only consider open-loop controls $u\in L^\infty([0,T],U)$. 

To derive these results, we view the control $u\in L^2_{\F}([0,T]\times\Omega,U)$ as the measurable composition
$$
u: \ \Omega\to L^1([0,T],U), \ \ \omega\mapsto u(\omega).
$$ 
Then, as is standard in rough path theory, we prove the continuity of the  It\^o-Lyons map  that returns the solution to the rough differential equation (RDE)
\begin{align*}
\left(\R^n,L^1([0,T],U),\sC_g^p([0,T],\R^d)\right)
\ &\longrightarrow \ 
C([0,T],\R^n)
\\
(\bar y,u,\mbX)
\ &\longmapsto\ 
Y
=
\operatorname{SolveRDE}(\bar y,u,\mbX),
\end{align*}
so that solutions to rough SDEs can be defined as the measurable composition
$$
\omega
\ \longmapsto\ 
\big(\bar x_0(\omega),u(\cdot,\omega),\mbB(\omega)\big)
\ \mathop{\longmapsto} \ 
x(\omega)
=
\operatorname{SolveRDE}
\big(\bar x_0(\omega),u(\cdot,\omega),\mbB(\omega)\big).
$$
Finally, by leveraging $p$-variation bounds and favorable integrability properties of Gaussian rough paths, we derive integrable error bounds   and the needle variation estimates in Lemma \ref{lem:needle:rough} for the rough PMP.

Proving these results %
requires long but straight-forward extensions of results in \cite{Lew2026} to accommodate controls $u\in L^1$. We describe these changes in this section.

Throughout this section, $C_a$ denotes a constant that depends on $a$ (e.g., $C_b$ depends on the drift $b$), and $p\in[2,3)$ unless specified otherwise. %

\subsection{Preliminaries on $p$-variation bounds and Gaussian rough paths}
To study solutions to RDEs and rough SDEs, we use bounds that quantify the regularity of solutions in terms of $p$-variation norms, which measure how fast increments of a path grow.  Importantly, the $p$-variation rough path norms of Gaussian rough paths   $\mbX=\mbB(\omega)$ have favorable integrability properties.

\begin{lemma}[Inequalities for  $p$-variations {\cite[Lemma 2.3 \ellone{(modified)}]{Lew2026}}]
\label{lem:pvariation:inequalities}
Let $p\geq 1$, and $X\in\C^p([0,T],\R^d)$. Then,
\begin{equation}
\label{eq:path_finite_var:infty_ineq}
\|X\|_\infty\leq
\|X_0\|+\|X\|_p.
\end{equation}
Let $X:[0,T]\to\R^d$, $Y^1,\dots,Y^n\in\C^p([0,T],\R^d)$,  $c\geq 0$ \ellone{and $f\in L^1([0,T],\R^d)$}. 
Then, 
\begin{equation}
\label{eq:path_finite_var:sum_pvars}
\hspace{-2mm}
\|X_{s,t}\|\leq \sum_{i=1}^n\|Y^i_{s,t}\|
+
c\ellone{\|f\|_{L^1([s,t],\R^d)}}
\, \forall s,t\in[0,T]
{\implies}
\|X\|_p\leq C_{n,p}
\bigg(\!
\sum_{i=1}^n\|Y^i\|_p+c\ellone{\|f\|_{L^1([0,T],\R^d)}}
\!\bigg)\!.
\end{equation}
Let $p\geq 2$, $X:[0,T]\to\R^d$, $Y^1,\widetilde{Y}^1,\dots,Y^n,\widetilde{Y}^n\in\C^p$, and $Z^1,\dots,Z^m\in\C^\frac{p}{2}$. %
Then, $\|X\|_p\leq\|X\|_\frac{p}{2}$ and 
\begin{align}\label{eq:path_finite_var:sum_p/2vars}
\|X_{s,t}\|
&\leq
\sum_{i=1}^n\|Y^i_{s,t}\|\|\widetilde{Y}^i_{s,t}\|
+\sum_{j=1}^m\|Z^j_{s,t}\|
+
c\ellone{\|f\|_{L^1([s,t],\R^d)}}
\ \forall s,t\in[0,T]
\\[-2mm]
\nonumber
&\implies 
\|X\|_\frac{p}{2}\leq C_{m,n,p}
\bigg(\sum_{i=1}^n\|Y^i\|_p\|\widetilde{Y}^i\|_p\,{+}\sum_{j=1}^m\|Z^j\|_\frac{p}{2}\,{+}\,c\ellone{\|f\|_{L^1([0,T],\R^d)}}
\hspace{-1pt}
\bigg).
\end{align}
Let $p\geq 1$, $\sigma\in C^1_b([0,T]\times\R^n,\R^m)$, and $X\in\C^p([0,T],\R^n)$.  Then,
\begin{equation}\label{eq:sigma(.,X):pvar}
\|\sigma(\cdot,X)\|_p\leq C_p\|\sigma\|_{C^1_b}(\|X\|_p+T).
\end{equation}
Moreover, if $\sigma\in C^2_b([0,T]\times\R^n,\R^m)$ and $X,\widetilde{X}\in\C^p([0,T],\R^n)$, then
\begin{equation}\label{eq:Delta_sigma(.,X):pvar}
\|\sigma(\cdot,X_\cdot)-\sigma(\cdot,\widetilde{X}_\cdot)\|_p\leq 
C_p\|\sigma\|_{C^2_b}
(
1 + \|X\|_p+\|\widetilde{X}\|_p
+
T
)
(\|X_0-\widetilde{X}_0\|+\|X-\widetilde{X}\|_p).
\end{equation}
 \end{lemma}
\begin{proof}[Proof of Lemma \ref{lem:pvariation:inequalities}]
\eqref{eq:path_finite_var:infty_ineq} is in \cite[Lemma 2.3]{Lew2026}. 
To show \eqref{eq:path_finite_var:sum_p/2vars}, for an arbitrary partition  $\pi$ of $[0,T]$,
 \begin{align*}
 \sum_{[s,t]\in\pi}\|X_{s,t}\|^\frac{p}{2}
 &\leq
 \sum_{[s,t]\in\pi}
 \Big(
 \sum_{i=1}^n\|Y^i_{s,t}\|\|\widetilde{Y}^i_{s,t}\|+
 \sum_{j=1}^m\|Z^j_{s,t}\|
 +
 c\ellone{\|f\|_{L^1([s,t],\R^d)}}
 \Big)^\frac{p}{2}
 \\[-2mm]
 &\hspace{-5mm}\leq
 (n+m+1)^\frac{p}{2}
 \Big(
 \sum_{i=1}^n
 \sum_{[s,t]\in\pi}
 \|Y^i_{s,t}\|^\frac{p}{2}\|\widetilde{Y}^i_{s,t}\|^\frac{p}{2}
 +
 \sum_{j=1}^m
 \sum_{[s,t]\in\pi}
 \|Z^j_{s,t}\|^\frac{p}{2}
 +
 c^\frac{p}{2}\sum_{[s,t]\in\pi}
 \ellone{\|f\|_{L^1([s,t],\R^d)}}^\frac{p}{2}
 \Big)
 \\[-1mm]
 &\hspace{-5mm}\leq
 C_{m,n,p}
 \Big(
 \sum_{i=1}^n
 \Big(\sum_{[s,t]\in\pi}
 \|Y^i_{s,t}\|^p
 \Big)^\frac{1}{2}
  \Big(\sum_{[s,t]\in\pi}
 \|\widetilde{Y}^i_{s,t}\|^p
 \Big)^\frac{1}{2}
 +
 \sum_{j=1}^m
 \|Z^j\|_\frac{p}{2}^\frac{p}{2}
 +
 (c\,\ellone{\|f\|_{L^1([0,T],\R^d)}})^\frac{p}{2}
 \Big)
 \\[-1mm]
 &\hspace{-5mm}\leq
 C_{m,n,p}
 \Big(
 \sum_{i=1}^n
 \|Y^i\|_p^\frac{p}{2}
 \|\widetilde{Y}^i\|_p^\frac{p}{2}
 +
 \sum_{j=1}^m
 \|Z^j\|_\frac{p}{2}^\frac{p}{2}
 +
 (c\,\ellone{\|f\|_{L^1([0,T],\R^d)}})^\frac{p}{2}
 \Big),
 \end{align*}
with the same arguments as in \cite{Lew2026}.  \eqref{eq:path_finite_var:sum_pvars} is shown similarly.   
\eqref{eq:sigma(.,X):pvar} and \eqref{eq:Delta_sigma(.,X):pvar} are in \cite[Lemma 2.3]{Lew2026}.
\end{proof}

\begin{lemma}[$p$-variations of Lebesgue integrals {\cite[Lemma 3.7 \ellone{(modified)}]{Lew2026}}]\label{lem:bounds_int_b_ds}
Let $p\in[2,3)$, $U\subseteq\R^m$, %
 $b$ satisfy Assumption \ref{assum:rde_drift},  
$Y,\widetilde{Y}\in C([0,T],\R^n)$, and $u,\tilde{u}\in \ellone{L^1([0,T],U)}$. Then,  
\begin{align}\label{eq:bounds_int_b_ds}
\left\|\int_0^\cdot b(s,Y_s,u_s)\dd s\right\|_\frac{p}{2} &\leq C_{p,b}T,
\  \ \text{and} \ \ \ 
\left\|\int_0^\cdot (b(s,Y_s,u_s)-b(s,\widetilde{Y}_s,u_s))\dd s\right\|_\frac{p}{2} \leq C_{p,b}T\|Y-\widetilde{Y}\|_\infty.
\end{align}
\ellone{
If $u$ and $\tilde{u}$ only differ on an interval $J\subseteq [0,T]$, i.e., $u_t=\tilde{u}_t$ for almost every $t\in [0,T]\setminus J$, then
\begin{align}\label{eq:bounds_int_b_ds:dy_du:u_subinterval}
\left\|\int_0^\cdot (b(s,Y_s,u_s)-b(s,\widetilde{Y}_s,\tilde{u}_s))\dd s\right\|_\frac{p}{2} 
\leq 
C_{p,b}
\big(
\ellone{T}\|Y-\widetilde{Y}\|_\infty
+
|J|
\big).
\end{align}
}
Moreover, if $b$ is also Lipschitz in $u$, then 
\begin{align}\label{eq:bounds_int_b_ds:dy_du}
\left\|\int_0^\cdot (b(s,Y_s,u_s)-b(s,\widetilde{Y}_s,\tilde{u}_s))\dd s\right\|_\frac{p}{2} 
\leq 
C_{p,b}
\big(
\ellone{T}\|Y-\widetilde{Y}\|_\infty
+
\|u-\tilde{u}\|_{\ellone{L^1([0,T],U)}}
\big).
\end{align}
\end{lemma}
\begin{proof}
\eqref{eq:bounds_int_b_ds} is in \cite[Lemma 3.7]{Lew2026}. 
\ellone{If $u$ and $\tilde{u}$ only differ on an interval $J$, then 
\begin{align*}
\bigg\|
\int_s^t (b(r,Y_r,u_r)-b(r,\widetilde{Y}_r,\widetilde{u}_r))\dd r
\bigg\|
&\leq
\bigg\|
\int_s^t (
(b(r,Y_r,u_r)-b(r,\widetilde{Y}_r,u_r))
+
(b(r,\widetilde{Y}_r,u_r)-b(r,\widetilde{Y}_r,\widetilde{u}_r))
)\dd r
\bigg\|
\\
&\leq
 C_b(\|Y-\widetilde{Y}\|_\infty\ellone{|t-s|}+|J\cap [s,t]|).
\end{align*}
Also,}
assuming that $b$ is Lipschitz in $u$, 
$\|\int_s^t (b(r,Y_r,u_r)-b(r,\widetilde{Y}_r,\tilde{u}_r))\dd r\|
\leq C_b(\|Y-\widetilde{Y}\|_\infty\ellone{|t-s|}+\|u-\tilde{u}\|_\ellone{L^1([s,t],U)})$.  
The desired inequalities then follow from \eqref{eq:path_finite_var:sum_p/2vars} in  Lemma \ref{lem:pvariation:inequalities}.
\end{proof}

\begin{definition}[Control and greedy partition {\cite[Definition 2.12]{Lew2026}}]%
\label{def:Nalpha}
A control is a continuous map $w:[0,T]^2\to\R$  such that $w(s,t)\geq 0$ and $w(s,t)+w(t,u)\leq w(s,u)$ for any $s,t,u\in[0,T]$.

Given a control $w$, a resolution $\alpha>0$ and $[s,t]\subseteq[0,T]$, we  define the sequence
$
\tau_0 =s$,  
$\tau_{i+1}=\inf\{u : w(\tau_i,u)\geq \alpha, \tau_i<u\leq t\}\wedge t$ for $i\in\N$, 
with the convention $\inf\emptyset=+\infty$, and
$$
N_{\alpha,[s,t]}(w):=\sup \{n\in\N\cup \{0\}:\tau_n< t\}.
$$

The \textit{greedy partition} of the interval $[s,t]$ is defined as the partition
$
\{\tau_i, i=0,1,\dots,N_{\alpha,[s,t]}(w)+1\}.
$ 
Given $p\in[2,3)$ and a rough path $\mbX\in\sC^p$, %
the control  $w_\mbX$ %
and $N_{\alpha,[s,t]}(\mbX)$  are defined %
as  $
w_\mbX(s,t)
:=
\|X\|_{p,[s,t]}^p
+
\smash{\|\bX\|_{\frac{p}{2},[s,t]}^{p/2}}$ and $N_{\alpha,[s,t]}(\mbX):=N_{\alpha,[s,t]}(w_\mbX)$.
\end{definition} 
The \textit{control} $w$  should not be confused with the \textit{control input} $u$: $w_\mbX$ bounds variations of the rough path $\mbX$, whereas $u$ steers the dynamical system.
Below, $p\in[2,3)$.

\begin{corollary}[$p$-variation rough path norm bounds {\cite[Corollary 2.15]{Lew2026}}]\label{cor:Nalpha:NX_NXtilde_NT:small_intervals}
Let $\mbX,\widetilde{\mbX}\in\sC^p$, $C_p=6^p$, and define 
$$
w:[0,T]^2\to\R,
\ 
w(s,t)=C_p(w_\mbX(s,t)+w_{\widetilde{\mbX}}(s,t)+|t-s|).
$$
Then,  for any $\alpha>0$ and  $[s,t]\subseteq[0,T]$,
\begin{equation}\label{eq:Nalpha<=3CpNalpha_X_and_time}
N_{\alpha,[s,t]}(w)\leq 5C_p(
N_{\alpha,[s,t]}(\mbX)+N_{\alpha,[s,t]}(\widetilde{\mbX})+|t-s|/\alpha+1
).
\end{equation}
Moreover, for any   $0<\alpha\leq 1$  and   any interval $[s,t]\subseteq[0,T]$ small-enough so that  $w(s,t)\leq \alpha$,
\begin{equation}\label{eq:|X|+|Xtilde|+|dt|<=alpha^p}
\|\mbX\|_{p,[s,t]}+\|\widetilde{\mbX}\|_{p,[s,t]}+|t-s|\leq
 \alpha^\frac{1}{p}.
\end{equation}
Finally, for any $0<\alpha\leq 1$ and any interval $[s,t]\subseteq[0,T]$, with $C_{p,\alpha}=66e\alpha^\frac{1}{p}$,
\begin{equation}\label{eq:||X||_p<=exp(N_alpha^p)}
\|\mbX\|_{p,[s,t]}+\|\widetilde{\mbX}\|_{p,[s,t]}+|t-s|
\leq 
C_{p,\alpha}\exp\left(N_{\alpha,[s,t]}(w)\right).
\end{equation}
\end{corollary}

\begin{theorem}[Enhanced Gaussian process and Gaussian rough paths {\cite[Theorem 2.17]{Lew2026}}] %
\label{thm:gaussian_rough_paths}
Define the enhanced Gaussian process $\mbB=(B,\bB)$ as in Section \ref{sec:background:gaussian_paths}, 
whose sample paths are rough paths $\mbB(\omega)=(B(\omega),\bB(\omega))\in\sC^p_g([0,T],\R^d)$ almost surely, and 
where $B$ is a Gaussian process satisfying the regular covariance condition \eqref{eq:gaussian_covariance_regular}. 
Then, for any $\alpha>0$ and $D\geq 0$, 
\begin{equation}\label{eq:gaussian_rough_paths:E[exp(Nalpha)]<infty}
\E\left[\exp\left(DN_{\alpha,[0,T]}(\mbB)\right)\right]<\infty.
\end{equation} 
\end{theorem}
\noindent This result follows from results in \cite{Friz2013}, see also \cite[Theorem 11]{Bayer2016},  and is key to obtaining integrable error bounds.

\subsection{Existence and uniqueness of solutions to RDEs}\label{sec:rdes:existence_uniqueness:nonlinear}\begin{restatement}
  [Nonlinear RDEs {\cite[Theorem 3.9 \ellone{(modified)}]{Lew2026}}]
  {Theorem}
  {thm:rdes:existence_uniqueness}
Let %
\ellone{$u\in L^1([0,T],U)$}, 
 $b$ satisfy Assumption \ref{assum:rde_drift}, 
$\sigma\in C_b^3$, $\bar y\in\R^n$, and 
$\mbX\in\sC_g^p$. 
Then, there exists a unique
$(Y,Y')\in\sD_X^p$ with $Y'_t=\sigma(t,Y_t)$ that solves the RDE
\begin{equation}
\tag{\ref{eq:nonlinear_rde}}
Y_t
=
\bar y
+
\int_0^tb(s,Y_s,u_s)\dd s
+
\int_0^t\sigma(s,Y_s)\dd\mbX_s.
\end{equation}
\end{restatement}
\begin{proof}[Proof of Theorem \ref{thm:rdes:existence_uniqueness}]
Every statement in the proof of \cite[Theorem 3.9]{Lew2026} remains valid if $u\in L^1$, since the drift $b$ is bounded and the control $u$ is fixed.
\end{proof}

\begin{restatement}
  [Linear RDEs: existence and uniqueness of solutions {\cite[Theorem 3.10]{Lew2026}}]
  {Theorem}
  {thm:rde:linear:existence_uniqueness}
Let %
$\bar v\in\R^n$, $A\in L^\infty$, 
$\mbX\in\sC^p$, 
and $(\Sigma,\Sigma')\in\sD_X^p$.   
Then, there exists a unique solution $(V,V')\in\sD^p_X$ with $V_t'=\Sigma_t V_t$ to the  RDE 
\begin{equation}
\tag{\ref{eq:linear_rde}}
V_t
=
\bar v
+
\int_0^tA_sV_s\dd s
+
\int_0^t\Sigma_sV_s\dd\mbX_s.
\end{equation}
\end{restatement}

\begin{assumption}[Stronger regularity of $b$ {\cite[Assumption 3.11]{Lew2026}}]\label{assumption:b:stronger:linear_rde}
Let $U\subseteq\R^m$, and  $b:[0,T]\times\R^n\times U\to\R^n$ satisfy
\begin{itemize} 
\item 
$b(\cdot, x, u):[0,T]\to\R^n$  is  measurable for all $(x,u)\in\R^n\times U$,  
\item  
$b(t,\cdot,\cdot):\R^n\times U\to\R^n$ 
is continuous for almost every $t\in[0,T]$,
\item 
$b(t,\cdot,u):\R^n\to\R^n$  is continuously differentiable  for almost every $t\in[0,T]$  and  all $u\in U$,  
\item %
$\big\|\nablaof{x}b(t,x,u)\big\|\leq C_b$ for almost every $t\in[0,T]$ and all $(x,u)\in\R^n\times U$ for some constant $C_b\geq 0$.
\end{itemize}
\end{assumption}
\begin{corollary}[Linearized RDE: existence and uniqueness of solutions {\cite[Corollary 3.12 \ellone{(modified)}]{Lew2026}}]\label{cor:rde:linear:existence_uniqueness}
Let %
$\bar v\in\R^n$, $U\subseteq\R^m$, 
$u\in \ellone{L^1([0,T],U)}$, 
$b$  satisfy  Assumption \ref{assumption:b:stronger:linear_rde}, 
$\sigma\in C_b^3$, 
$\mbX\in\sC^p$, and 
$(Y,Y')\in\sD^p_X$.
Then, there exists a unique solution $(V,V')\in\sD^p_X$ with $V'=\nablaof{x}\sigma(\cdot,Y_\cdot)V_\cdot$ to the linear RDE $V_t =\bar v +  
\int_0^t
\nablaof{x}b(s,Y_s,u_s)V_s\dd s
+  
\int_0^t
\nablaof{x}\sigma(s,Y_s)V_s\dd\mbX_s$. 
\end{corollary}
\begin{proof}
\ellone{The proof is identical to the proof of \cite[Corollary 3.12]{Lew2026}, since $\nablaof{x}b$ is bounded and $u$ is fixed.}
\end{proof}

\subsection{Bounds on solutions to RDEs}\label{sec:rdes:bounds}
\subsubsection{Bounds on solutions to nonlinear RDEs}
\begin{proposition}
[Error bound for solutions to RDEs on short intervals {\cite[Proposition 3.14 \ellone{(modified)}]{Lew2026}}]\label{prop:rdes:error_bound}
Let %
$y,\tilde{y}\in\R^n$, $U\subseteq\R^m$, 
$u,\tilde{u}\in \ellone{L^1([0,T],U)}$, 
$b$ satisfy Assumption \ref{assum:rde_drift} and be Lipschitz in $u$, 
$\sigma\in C_b^3$,  
$\mbX,\widetilde{\mbX}\in\sC^p$, 
 $(Y,Y')\in\sD^p_X$ and $(\widetilde{Y},\widetilde{Y}')\in\sD^p_{\widetilde{X}}$ 
 solve the RDEs 
\begin{align*}
Y_t=y+\int_0^tb(s,Y_s,u_s)\dd s+\int_0^t\sigma(s,Y_s)\dd\mbX_s,
\ \
\widetilde{Y}_t=\tilde{y}+\int_0^tb(s,\widetilde{Y}_s,\tilde{u}_s)\dd s+\int_0^t\sigma(s,\widetilde{Y}_s)\dd\widetilde{\mbX}_s,
\ \, t\in[0,T].
\end{align*}
Then, there exist two constants $C_{p,b,\sigma}\geq 1$ and $0<\alpha_{p,b,\sigma}<1$ such that
\begin{gather}\label{eq:bounds_pvars_solutions_RDEs:small_intervals:combined}
\|Y\|_{p,I},\, 
\|\widetilde{Y}\|_{p,I},\, 
\|R^Y\|_{\frac{p}{2},I},\, 
\|R^{\widetilde{Y}}\|_{\frac{p}{2},I},\, 
K_{Y,I},\, 
K_{\widetilde{Y},I}\leq
C_{p,b,\sigma},
\quad\text{and}
\\[2mm]
\label{eq:RDE:DY'+dRY:close}
\|Y'-\widetilde{Y}'\|_{p,I}+\|R^Y-R^{\widetilde{Y}}\|_{\frac{p}{2},I}
\leq
C_{p,b,\sigma}(\|Y_{t_0}-\widetilde{Y}_{t_0}\|
+
\|\Delta\mbX\|_{p,I}
+
\|u-\tilde{u}\|_{\ellone{L^1(I,U)}}),
\end{gather}
for any interval $I=[t_0,t_1]\subseteq[0,T]$ such that $\|\mbX\|_{p,I}+\|\widetilde{\mbX}\|_{p,I}+|I|\leq
 \alpha_{p,b,\sigma}^\frac{1}{p}$.

\ellone{
If $u$ and $\tilde{u}$ only differ on a subinterval $J\subseteq I$, i.e., $u_t=\tilde{u}_t$ for almost every $t\in I\setminus J$, then}
\begin{align}\label{eq:RDE:DY'+dRY:close:u_subinterval}
\ellone{
\|Y'-\widetilde{Y}'\|_{p,I}+\|R^Y-R^{\widetilde{Y}}\|_{\frac{p}{2},I}
\leq
C_{p,b,\sigma}
(\|\Delta Y_{t_0}\|
+
\|\Delta \mbX\|_{p,I}
+
|J|
).
}
\end{align}
\end{proposition}
\begin{proof}
\ellone{The proof follows  \cite[Proposition 3.14]{Lew2026}, using the modified inequalities \eqref{eq:bounds_int_b_ds} and  
\eqref{eq:bounds_int_b_ds:dy_du:u_subinterval} instead of \cite[Equation (3.13) in Lemma 3.7]{Lew2026} to obtain \eqref{eq:RDE:DY'+dRY:close} and \eqref{eq:RDE:DY'+dRY:close:u_subinterval}.
}
\end{proof}
 
\begin{lemma}[Bounds for solutions to RDEs on long intervals {\cite[Lemma 3.15]{Lew2026}}]\label{lem:rdes:bound:entire_interval}
Let %
$y,U,u,b,\sigma,\mbX,Y,Y'$ be as in Theorem \ref{thm:rdes:existence_uniqueness}, where $b$ satisfies Assumption \ref{assum:rde_drift} and $\sigma\in C_b^3$.   
Then, there exist constants $C_{p,T,b,\sigma}\geq 1$ and $0<\alpha_{p,b,\sigma}<1$ such that  
\begin{align}
\label{eq:RDE:Yp_Y'p_RYp/2:bound:full_interval}
\|Y\|_{p,[0,T]}+
\|Y'\|_{p,[0,T]}+
\|R^Y\|_{\frac{p}{2},[0,T]}
&\leq
C_{p,T,b,\sigma}\exp\left(
C_{p,T,b,\sigma}N_{\alpha_{p,b,\sigma},[0,T]}(\mbX)
\right),
\\
\label{eq:RDE:Y:bound:full_interval}
\|Y\|_{\infty,[0,T]}
&\leq
C_{p,T,b,\sigma}\exp\left(
C_{p,T,b,\sigma}N_{\alpha_{p,b,\sigma},[0,T]}(\mbX)
\right)+\|y\|.
\end{align}
\end{lemma} 

\begin{proposition}
[Error bound for solutions to RDEs on long intervals {\cite[Proposition 3.16 \ellone{(modified)}]{Lew2026}}]\label{prop:rdes:error_bound:entire_interval}
Define %
$y,\tilde{y},U,u,\tilde{u},b,\sigma$, $\mbX,\widetilde{\mbX}$
as in Proposition \ref{prop:rdes:error_bound}, 
where $b$ satisfies Assumption \ref{assum:rde_drift} and is Lipschitz in $u$. Let $(Y,Y')\in\sD^p_X$ and $(\widetilde{Y},\widetilde{Y}')\in\sD^p_{\widetilde{X}}$  
solve the RDEs
\begin{align*}
Y_t=y+\int_0^tb(s,Y_s,u_s)\dd s+\int_0^t\sigma(s,Y_s)\dd\mbX_s,
\ \
\widetilde{Y}_t=\tilde{y}+\int_0^tb(s,\widetilde{Y}_s,\tilde{u}_s)\dd s+\int_0^t\sigma(s,\widetilde{Y}_s)\dd\widetilde{\mbX}_s,
\ \, t\in[0,T].
\end{align*} 
Then, %
for any interval $I=[t_0,t_1]\subseteq[0,T]$, 
\begin{align}\label{eq:RDE:DY'+dRY:close:full_interval}
\hspace{-2mm}
\|Y'-\widetilde{Y}'\|_{p,I}+\|R^Y-R^{\widetilde{Y}}\|_{\frac{p}{2},I}
&\leq
C\exp\left(
CN_{\alpha,I}(w)
\right)(\|\Delta Y_{t_0}\|
+
\|\Delta \mbX\|_{p,I}
+
\|\Delta u\|_{\ellone{L^1(I,U)}}
),
\\
\label{eq:RDE:DY:close:full_interval}
\|Y-\widetilde{Y}\|_{\infty,I} 
&\leq
C\exp\left(
CN_{\alpha,I}(w)
\right)(\|\Delta Y_{t_0}\|
+
\|\Delta \mbX\|_{p,I}
+
\|\Delta u\|_{\ellone{L^1(I,U)}}
),
\end{align}
for some constants $C\geq 1$ and $0<\alpha<1$ that depend on $(p,b,\sigma)$.

\ellone{
If $u$ and $\tilde{u}$ only differ on a subinterval $J\subseteq I$, i.e., $u_t=\tilde{u}_t$ for almost every $t\in I\setminus J$, then
}
\begin{align}\label{eq:RDE:DY:close:full_interval:u_subinterval}
\ellone{
\|Y'-\widetilde{Y}'\|_{p,I}+\|R^Y-R^{\widetilde{Y}}\|_{\frac{p}{2},I}
+
\|Y-\widetilde{Y}\|_{\infty,I} 
\leq
C\exp\left(
CN_{\alpha,I}(w)
\right)(\|\Delta Y_{t_0}\|
+
\|\Delta \mbX\|_{p,I}
+
|J|
).
}
\end{align}
\end{proposition}

\begin{proof}
\ellone{The proof follows  \cite[Proposition 3.16]{Lew2026}, using the modified inequalities \eqref{eq:RDE:DY'+dRY:close} and  
\eqref{eq:RDE:DY'+dRY:close:u_subinterval} instead of \cite[Equation (3.20) in Proposition 3.14]{Lew2026} to obtain \eqref{eq:RDE:DY'+dRY:close:full_interval}-\eqref{eq:RDE:DY:close:full_interval:u_subinterval}.
}
\end{proof}

\subsubsection{Bounds on solutions to linear RDEs}
\begin{lemma}[Bounds on solutions to linear RDEs {\cite[Lemma 3.17]{Lew2026}}]\label{lem:rde:linear:bounded_solutions}
Let %
$y,A,\mbX,\Sigma,\Sigma'$ be as in Theorem \ref{thm:rde:linear:existence_uniqueness}. 
Assume that  there exist two constants $C_\Sigma\geq 1$ and $0<\alpha_\Sigma<1$ such that 
$\|\Sigma\|_{\infty,I} + \|\Sigma\|_{p,I} +\|\Sigma'\|_{p,I}+\|R^\Sigma\|_{\frac{p}{2},I}
\leq
C_\Sigma$ 
for any interval $I\subseteq[0,T]$ such that $\|\mbX\|_{p,I}+|I|\leq
 \alpha_\Sigma^{1/p}$.  
Then,  
\begin{align}
\label{eq:RDE:linear:Vp_V'p_RVp/2:bound}
\|V\|_{p,[0,T]}+
\|V'\|_{p,[0,T]}+
\|R^V\|_{\frac{p}{2},[0,T]}
&\leq 
C_{p,T,A,\Sigma}\exp\left(C_{p,T,A,\Sigma}N_{\alpha_{p,A,\Sigma},[0,T]}(\mbX)\right)\|v\|,
\\
\label{eq:RDE:linear:V:bound}
\|V\|_{\infty,[0,T]}
&\leq 
C_{p,T,A,\Sigma}\exp\left(C_{p,T,A,\Sigma}N_{\alpha_{p,A,\Sigma},[0,T]}(\mbX)\right)\|v\|,
\end{align}
where the constants $C_{p,T,A,\Sigma}\geq 1$ and $0<\alpha_{p,A,\Sigma}\leq\alpha_\Sigma$    only depend on $(p,T,\|A\|_\infty,C_\Sigma,\alpha_\Sigma)$. 
\end{lemma}

\begin{lemma}[Bounds on the Jacobian flow {\cite[Lemma 3.18 \ellone{(modified)}]{Lew2026}}]
\label{lem:rde:linearized:bounded_solutions}
Let %
$y,v\in\R^n$, 
$U\subseteq\R^m$, 
$u\in \ellone{L^1([0,T],U)}$, 
$b$  satisfy  Assumption \ref{assumption:b:stronger:linear_rde}, 
$\sigma\in C_b^3$,  
$\mbX\in\sC^p$, 
and $(Y,Y'),(V,V')\in\sD^p_X$ 
solve $Y_t=y+\int_0^tb(s,Y_s,u_s)\dd s+\int_0^t\sigma(s,Y_s)\dd\mbX_s$   and  $V_t =v +  
\int_0^t
\nablaof{x}b(s,Y_s,u_s)V_s\dd s
+  
\int_0^t
\nablaof{x}\sigma(s,Y_s)V_s\dd\mbX_s$. 
Then, 
\begin{align}
\label{eq:rde:linearized:Vp_V'p_RVp/2:bound}
\|V\|_{p,[0,T]}+
\|V'\|_{p,[0,T]}+
\|R^V\|_{\frac{p}{2},[0,T]}
&\leq 
C_{p,T,b,\sigma}\exp\left(C_{p,T,b,\sigma}N_{\alpha_{p,b,\sigma},[0,T]}(\mbX)\right)\|v\|,
\\
\label{eq:rde:linearized:bounded_solutions}
\|V\|_{\infty,[0,T]}
&\leq 
C_{p,T,b,\sigma}\exp\left(C_{p,T,b,\sigma}N_{\alpha_{p,b,\sigma},[0,T]}(\mbX)\right)\|v\|
\end{align}
for some constants $C_{p,T,b,\sigma}\geq 1$ and $0<\alpha_{p,b,\sigma}<1$.
\end{lemma} 
\begin{proof}
\ellone{The proof is identical to the proof of \cite[Lemma 3.18]{Lew2026}, since the control $u$ is fixed.}
\end{proof}

\subsection{Continuity of solutions to RDEs}
\ellone{In this section, we show that the  solution maps
$$
Y,V: 
\R^n\times L^1([0,T],U)\times \sC^p_g([0,T],\R^d)
\to C([0,T],\R^n)
$$
are continuous. This continuity is used in Section \ref{sec:rdes:random_integrable} to show the measurability of solutions of rough SDEs, which are RDEs with random initial conditions, random controls, and Gaussian rough paths.
}

\subsubsection{Continuity of solutions to nonlinear RDEs}

\begin{lemma}
[Continuity of solutions to RDEs {\cite[Lemma 3.19 \ellone{(modified)}]{Lew2026}}]\label{lem:rdes:continuity}
Let %
$y,\tilde{y},U,u,\tilde{u},b,\sigma,\mbX,\widetilde{\mbX}$
be as in Proposition \ref{prop:rdes:error_bound}, 
where $b$ satisfies Assumption \ref{assum:rde_drift} and is Lipschitz in $u$ and $\sigma\in C_b^3$, 
and $(Y,Y')\in\sD^p_X$ and $(\widetilde{Y},\widetilde{Y}')\in\sD^p_{\widetilde{X}}$  
solve the RDEs
\begin{align*}
Y_t=y+\int_0^tb(s,Y_s,u_s)\dd s+\int_0^t\sigma(s,Y_s)\dd\mbX_s,
\ 
\widetilde{Y}_t=\tilde{y}+\int_0^tb(s,\widetilde{Y}_s,\tilde{u}_s)\dd s+\int_0^t\sigma(s,\widetilde{Y}_s)\dd\widetilde{\mbX}_s,
\  t\in[0,T],
\end{align*} and assume that there exists a constant $M\geq 0$ such that   $\|\mbX\|_p,\|\widetilde{\mbX}\|_p\leq M$. Then,%
\begin{align}
\hspace{-2mm}
\|Y'-\widetilde{Y}'\|_{p,[0,T]}+
\|R^Y-R^{\widetilde{Y}}\|_{\frac{p}{2},[0,T]}
&\leq 
C_{p,T,b,\sigma,M}
\left(
\|y-\tilde{y}\| 
+
\|\Delta \mbX\|_{p,[0,T]}
 + 
 \|u-\tilde{u}\|_{\ellone{L^1([0,T],U)}}
\right),
\\
\|Y-\widetilde{Y}\|_{\infty,[0,T]}
&\leq 
C_{p,T,b,\sigma,M}
\left(
\|y-\tilde{y}\|
+
\|\Delta \mbX\|_{p,[0,T]}
 + 
 \|u-\tilde{u}\|_{\ellone{L^1([0,T],U)}}
\right)
\end{align}
for a constant $C_{p,T,b,\sigma,M}\geq 0$.
\end{lemma}
\begin{proof}
\ellone{It is identical to the proof of \cite[Lemma 3.19]{Lew2026} using the modified $L^1$ estimates \eqref{eq:RDE:DY'+dRY:close:full_interval}-\eqref{eq:RDE:DY:close:full_interval}.}
\end{proof}

The following result relaxes the Lipschitz condition of $b$ in $u$, and will be used to prove continuity of solutions to linear RDEs without assuming that $\nablaof{x}b$ is Lipschitz in $u$ in \ref{lem:rde:linearized:continuous}.
\begin{lemma}
[\ellone{Continuity of solutions to RDEs without Lipschitz continuity of $b$ in $u$}]\label{lem:rdes:continuity:sqrtholderinu}
Let %
$y,\tilde{y},U,u,\tilde{u},b,\\\sigma,\mbX,\widetilde{\mbX},Y,Y',\widetilde{Y},\widetilde{Y}'$ be as in Lemma \ref{lem:rdes:continuity}, 
where $b$ satisfies Assumption \ref{assum:rde_drift}. 
Instead of assuming that $b$ is Lipschitz in $u$, assume that $b$ is only $\frac12$-H\"older continuous in $u$:
\begin{equation}
\label{eq:b:holdercontinuous_in_u}
\left\|
b(t,x,u)-b(t,x,\widetilde u)
\right\|
\leq C_b
\sqrt{\|u-\widetilde u\|}
\end{equation}
for  almost every $t\in[0,T]$, all $x,\tilde{x}\in\R^n$, and all $u,\widetilde u\in U$.

Assume that there exists a constant $M\geq 0$ such that   $\|\mbX\|_p,\|\widetilde{\mbX}\|_p\leq M$. Then,
\begin{align}
\label{eq:RDE:DY'+dRY:close:full_interval:continuity}
\hspace{-2mm}
\|Y'-\widetilde{Y}'\|_{p,[0,T]}+
\|R^Y-R^{\widetilde{Y}}\|_{\frac{p}{2},[0,T]}
&\leq 
C_{p,T,b,\sigma,M}
\!\left(
\|y-\tilde{y}\| 
+
\|\Delta \mbX\|_{p,[0,T]}
 + 
 \ellone{
 \|u-\tilde{u}\|_{\ellone{L^1([0,T],U)}}
 ^{1/2}}
\right)\!,
\\
\label{eq:RDE:DY:close:full_interval:continuity}
\|Y-\widetilde{Y}\|_{\infty,[0,T]}
&\leq 
C_{p,T,b,\sigma,M}
\!\left(
\|y-\tilde{y}\|
+
\|\Delta \mbX\|_{p,[0,T]}
 + 
 \ellone{
 \|u-\tilde{u}\|_{\ellone{L^1([0,T],U)}}
 ^{1/2}}
\right)\!.
\end{align}
\end{lemma}
\begin{proof}
\ellone{First,  under the $\frac12$-H\"older continuity assumption in \eqref{eq:b:holdercontinuous_in_u}, the bound \eqref{eq:bounds_int_b_ds:dy_du} can be revised as
\begin{align}\label{eq:bounds_int_b_ds:dy_du:1/2holder}
\left\|\int_0^\cdot (b(s,Y_s,u_s)-b(s,\widetilde{Y}_s,\tilde{u}_s))\dd s\right\|_\frac{p}{2} 
\leq 
C_{p,b}
\big(
\ellone{T}\|Y-\widetilde{Y}\|_\infty
+
\sqrt{T}\|u-\tilde{u}\|_{\ellone{L^1([0,T],U)}}^{1/2}
\big).
\end{align}
The desired bounds then follow by adapting the proof of Lemma \ref{lem:rdes:continuity} \cite[Lemma 3.19]{Lew2026}.
}
\end{proof}

\subsubsection{Continuity of solutions to linear RDEs}

The continuity of solutions to  linear RDEs is proved under the following stronger assumption. 
\begin{assumption}[Stronger regularity of $b$]\label{assumption:b:stronger}
Let  $U\subseteq\R^m$, and $b$ satisfy Assumption \ref{assumption:b:stronger:linear_rde} and  
\begin{gather*}
\|b(t,x,u)\|+\big\|\nablaof{x}b(t,x,u)\big\| \leq C_b,
\quad
\|\nablaof{x}b(t,x,u)-\nablaof{x}b(t,\tilde{x},u)\|\leq C_b\|x-\tilde{x}\|,
\\[1mm]
\ellone{
\left\|
b(t,x,u)-b(t,x,\widetilde u)
\right\|
\leq C_b
\|u-\widetilde u\|,
}
\end{gather*} for  almost every $t\in[0,T]$, all $x,\tilde{x}\in\R^n$, and all $u,\widetilde u\in U$.
\end{assumption}
\begin{remark}[Interpolation inequality]\label{remark:interpolation}
\ellone{If $b$ satisfies  Assumption \ref{assumption:b:stronger}, then
\begin{equation}\label{eq:nablab:holder:interpolation}
\left\|
\nablaof{x}b(t,x,u)-\nablaof{x}b(t,x,\widetilde u)
\right\|
\leq C_b
\sqrt{\|u-\widetilde u\|},
\end{equation}
Indeed, denote $\Delta b(x)=b(x,u)-b(x,\widetilde{u})$, and apply Taylor's Theorem to get
\begin{align*}
\Delta b(x+he)-\Delta b(x)
&=
\int_0^1\nabla\Delta b(x+\theta he)\dd\theta he
=
\nablaof{x}\Delta b(x)he+
\int_0^1(\nablaof{x}\Delta b(x+\theta he)-\nablaof{x}\Delta b(x))\dd\theta he
\end{align*}
 for any $h\geq 0$ and unit-norm $e\in\R^n$.  
Then, using $\|
\Delta b(x+he)-\Delta b(x)
\|\leq C_b\|u-\widetilde u\|$ and $\|\nablaof{x}\Delta b(x)-\nablaof{x}\Delta b(\widetilde x)\|\leq C_b\|x-\widetilde x\|$, we obtain 
\begin{align*}
\|\nablaof{x}\Delta b(x)\|h
&\leq
\|
\Delta b(x+he)-\Delta b(x)
\|
+
\left\|
\int_0^1(\nablaof{x}\Delta b(x+\theta he)-\nablaof{x}\Delta b(x))\dd\theta he
\right\|
\leq C_b(\|u-\widetilde u\|+h^2).
\end{align*}
Choosing $h=\sqrt{\|u-\widetilde u\|}$, we obtain \eqref{eq:nablab:holder:interpolation}.
}
\end{remark}

\begin{lemma}[Continuity of the Jacobian flow {\cite[Lemma 3.21 \ellone{(modified)}]{Lew2026}}]\label{lem:rde:linearized:continuous}
Let %
$y,\tilde{y},U,u,\tilde{u},b,\sigma,\mbX,\widetilde{\mbX},Y,\\Y',\widetilde{Y},\widetilde{Y}',M$ be as in Lemma \ref{lem:rdes:continuity}, 
and \ellone{assume that $b$ satisfies  Assumption \ref{assumption:b:stronger}}, 
and 
$\sigma\in C_b^4$. 
Let $v\in C_b^1(\R^n,\R^n)$ and  $(V,V')\in\sD^p_X$ and $(\widetilde{V},\widetilde{V}')\in\sD^p_{\widetilde{X}}$ 
solve the linear RDEs 
\begin{align*}
\dd V_t &= 
\nablaof{x}b(t,Y_t,u_t)V_t\dd t +
\nablaof{x}\sigma(t,Y_t)V_t\dd\mbX_t
\ \text{ and }\   
\dd \widetilde{V}_t = 
\nablaof{x}b(t,\widetilde{Y}_t,\tilde{u}_t)\widetilde{V}_t\dd t +
\nablaof{x}\sigma(t,\widetilde{Y}_t)\widetilde{V}_t\dd\widetilde{\mbX}_t
\end{align*}
over $[0,T]$ and from $v(y)$ and $v(\tilde y)$, respectively. 
Then,  for a constant $C_{p,T,b,\sigma,v,M}\geq 0$,
\begin{align*}
\|V'-\widetilde{V}'\|_{p,[0,T]}+
\|R^V-R^{\widetilde{V}}\|_{\frac{p}{2},[0,T]}
+
\|V-\widetilde{V}\|_{\infty,[0,T]}
&\leq 
\\
&\hspace{-4cm}
C_{p,T,b,\sigma,v,M}
\left(
\|y-\tilde{y}\| 
+
\|\Delta \mbX\|_{p,[0,T]}
 + 
 \ellone{
 \|\Delta u\|_{\ellone{L^1([0,T],U)}}
}
 + 
 \ellone{
 \|\Delta u\|_{\ellone{L^1([0,T],U)}}
 ^{1/2}}
\right).
\end{align*}
\end{lemma}
\begin{proof}
The proof follows \cite[Lemma 3.21]{Lew2026}, \ellone{but does not assume that $\nablaof{x}b$ is Lipschitz in $u$}. Since $(V,V')$ is bounded,  it solves a nonlinear RDE with bounded coefficients $(\hat{b},\hat\sigma)$, \ellone{where $\hat{b}$ satisfies}
\begin{equation}
\tag{\ref{eq:b:holdercontinuous_in_u}}
\smash{
\ellone{
\big\|
\hat{b}(t,x,u)-\hat{b}(t,x,\widetilde u)
\big\|
\mathop{\leq}^{\eqref{eq:nablab:holder:interpolation}} 
C_b
\sqrt{\|u-\widetilde u\|}.
}
}
\end{equation}
 Thus, the pair $((Y,V),(Y',V'))$ solves a joint RDE with bounded coefficients, 
and similarly for $(\widetilde{Y},\widetilde{V})$. 
Then using Lemma \ref{lem:rdes:continuity} and \eqref{eq:RDE:DY'+dRY:close:full_interval:continuity} and  \eqref{eq:RDE:DY:close:full_interval:continuity} in Lemma \ref{lem:rdes:continuity:sqrtholderinu},
\begin{align*} 
\|(Y,V)-(\widetilde{Y},\widetilde{V})\|_{\infty} 
&\leq
C_{p,T,b,\sigma,v,M}\big(
\|(y,v(y))-(\tilde{y},v(\tilde{y}))\|
+
\|\Delta\mbX\|_{p}
+
\ellone{
\|\Delta u\|_{\ellone{L^1}}
+
 \|\Delta u\|_{\ellone{L^1}}
 ^{1/2}}
\big),
\end{align*}
and similarly for $\|(Y',V')-(\widetilde{Y}',\widetilde{V}')\|_{p}+\|R^{(Y,V)}-R^{(\widetilde{Y},\widetilde{V})}\|_{\frac{p}{2}}$,
from which we deduce the desired result.
\end{proof}

\subsection{Solutions to rough SDEs}\label{sec:rdes:random_integrable}
\begin{restatement}
  [a) (Rough SDEs {\cite[Theorem 3.22 \ellone{(modified and extended)}]{Lew2026}}]
  {Theorem}
  {thm:sdes:rough}
Let $\rho\in[1,3/2)$, $p\in(2\rho,3)$, $\mbB$ be an enhanced Gaussian process (see Section \ref{sec:background:gaussian_paths}),  $\ell\geq1$, 
$U\subseteq\R^m$, 
$u\in L^\ell([0,T]\times\Omega,U)$, 
$\bar x_0\in L^\ell(\Omega,\R^n)$, 
$b$ satisfy Assumption \ref{assum:rde_drift} and be Lipschitz in $u$, 
and $\sigma\in C_b^3$.

 For almost every $\omega\in\Omega$, let $(x(\omega),x'(\omega))\in\sD_{B(\omega)}^p([0,T],\R^n)$ 
solve the RDE
\begin{equation}
\tag{\ref{eq:random_gaussian_rde}}
\dd x_t(\omega)
=
b(t,x_t(\omega),u_t(\omega))\dd t
+
\sigma(t,x_t(\omega))\dd\mbB_t(\omega),
\quad t\in[0,T]
\end{equation}
from $\bar x_0(\omega)$. Then, $x\in L^\ell(\Omega,C([0,T],\R^n))$. 

If $\bar x_0\in L^\ell_{\F_0}$, $u\in L^\ell_\F$, and $\mbB|_{[0,t]}:\Omega\to\sC^p_g([0,t],\R^d)$ is $\F_t$-measurable  for every $t\in[0,T]$, then $x$ is also $\F$-adapted. 
\end{restatement}
\begin{proof}
Note that $u\in L^1([0,T]\times\Omega,U)$ since 
 $ 
\|u(\omega)\|_{L^1([0,T],U)}\leq T^{1-1/\ell}\|u(\omega)\|_{L^\ell([0,T],U)}
$ almost surely by H\"older's inequality, so we may view $\omega\mapsto u(\omega)$ as a measurable map with values in $L^1([0,T],U)$.   
The proof follows the proof of \cite[Theorem 3.22]{Lew2026}, accounting for $L^1$-valued random controls.

\textit{1) Measurability}: The solution map $x:\Omega\to C([0,T],\R^n)$ is the composition of the measurable map 
$$
(\bar x_0,u,\mbB):\Omega\to\R^n\times L^1([0,T],U)\times\sC^p_g([0,T],\R^d)
$$
and of the continuous map (thanks to Lemma \ref{lem:rdes:continuity})
$$
Y:
\R^n\times L^1([0,T],U)\times\sC^p_g([0,T],\R^d) \to C([0,T],\R^n)
$$
that solves the RDE $\dd Y_t=
b(t,Y_t,u_t)\dd t
+
\sigma(t,Y_t)\dd\mbX_t$ with initial condition $y\in\R^n$, control $u\in L^1([0,T],U)$, and driving signal $\mbX\in\sC^p_g([0,T],\R^d)$. 
Thus, $x$ is measurable. 

By restricting the maps above to  $[0,t]$, it is clear (see the proof of \cite[Theorem 9.1]{Friz2020}) that if $\bar x_0\in L^\ell_{\F_0}$, $u$ is progressively measurable, and $\mbB|_{[0,t]}$ is $\F_t$-measurable for every $t\in[0,T]$, then $x$ is also $\F$-adapted.  

\textit{2) Integrability}: 
By Lemma \ref{lem:rdes:bound:entire_interval},  there exist constants $C_{p,T,b,\sigma}\geq 1$ and $0<\alpha_{p,b,\sigma}<1$ such that  
\begin{align*}
\|x\|_{\infty,[0,T]}
&\smash{\mathop{\leq}^{\eqref{eq:RDE:Y:bound:full_interval}}}
C_{p,T,b,\sigma}\exp\left(
C_{p,T,b,\sigma}N_{\alpha_{p,b,\sigma},[0,T]}(\mbB)
\right)+\|\bar x_0\|
\end{align*}
almost surely.  
Since $\E\left[\exp\left(
CN_{\alpha,[0,T]}(\mbB)
\right)\right]<\infty$  by \eqref{eq:gaussian_rough_paths:E[exp(Nalpha)]<infty}, we obtain $\E[\|x\|_\infty^\ell]<\infty$, so we conclude that  $x\in L^\ell(\Omega,C([0,T],\R^n))$. 
\end{proof}

\begin{restatement}
  [b) (Rough SDEs {\cite[Theorem 3.22 \ellone{(modified and extended)}]{Lew2026}}]
  {Theorem}
  {thm:sdes:rough}

 In addition to the definitions and assumptions of Theorem \ref{thm:sdes:rough} (a), assume that $(b,\sigma)$ satisfy conditions (i)-(v) of Theorem \ref{thm:sdes:rough} (b). 
Let $t_1\in[0,T]$,  and 
$\bar v_{t_1},\bar p_T\in L^\ell(\Omega,\R^n)$. 
For almost every $\omega\in\Omega$, let $(v(\omega),v'(\omega)),(p(\omega),p'(\omega))\in\sD_{B(\omega)}^p([t_1,T],\R^n)$ with $v'_t(\omega)=\Sigma_tv_t$ and $p'_t(\omega)=-\Sigma_t^\top p_t$ solve the RDEs
\begin{align}
\tag{\ref{eq:tangent_rde_background}}
v_t(\omega)
&=
\bar v_{t_1}(\omega)
+
\int_{t_1}^t\nabla_xb(s,x_s,u_s)(\omega)v_s(\omega)\dd s
+
\int_{t_1}^t\nabla_x\sigma(s,x_s)(\omega)v_s(\omega)\dd\mbB_s(\omega),
\\
\tag{\ref{eq:adjoint_rde_background}}
p_t(\omega)
&=
\bar p_T(\omega)
+
\int_t^T
\nabla_xb(s,x_s,u_s)(\omega)^\top p_s(\omega)
\dd s
+
\int_t^T
\nabla_x\sigma(s,x_s)(\omega)^\top p_s(\omega)\dd\mbB_s(\omega),
\end{align}
on $[t_1,T]$. Then, $v,p\in L^{\ell'}(\Omega,C([t_1,T],\R^n))$ for any $\ell'<\ell$. 
If $\bar x_0\in L^\ell_{\F_0}$, $\bar v_{t_1}$ is $\F_{t_1}$-measurable, $u$ is progressively measurable, and $\mbB|_{[0,t]}:\Omega\to\sC^p_g([0,t],\R^d)$ is $\F_t$-measurable  for every $t\in[0,T]$, then $v$ is also $\F$-adapted.
\end{restatement}
\begin{proof}
The proof follows the proof of Theorem \ref{thm:sdes:rough} (a). The measurability of the solution maps follows from Lemma \ref{lem:rde:linearized:continuous}. 
The integrability of these maps follows from Lemma \ref{lem:rde:linearized:bounded_solutions} and H\"older's inequality (which explains the weaker integrability with $\ell'<\ell$, see also Remark \ref{remark:rough_adjoint_integrability}). 
\end{proof}

\subsection{Needle variations}
Finally, we derive the needle variation error bounds in Lemma \ref{lem:needle:rough} for the  rough PMP. The key differences with the results in \cite{Lew2026} are that the control $u$ is $L^1$-valued  and that the bounds do not depend on value of the optimal control $u$ nor on the needle variation values $\bar{u}_1,\dots,\bar{u}_q$. This is possible by using the boundedness of the drift $b$ and expressing error bounds in terms of the needle variation durations $\eta_i$ using \eqref{eq:RDE:DY:close:full_interval:u_subinterval}.  As a result, the error bounds are integrable even though the control is stochastic.
\begin{proposition}
[Needle variations and linearized RDEs {\cite[Proposition 4.2 \ellone{(modified)}]{Lew2026}}]
\label{prop:linear_variation}  
 Let %
$y\in\R^n$, $U\subseteq\R^m$, 
$u\in \ellone{L^1([0,T],U)}$, 
$b$   satisfy Assumption \ref{assumption:b:stronger}  and be Lipschitz in $u$, $\sigma\in C_b^4$, and 
$\mbX\in\sC^p$. 
Let $(Y,Y')\in\sD^p_X$  and $(Y^{\pi_1},(Y^{\pi_1})')\in\sD^p_X$ %
solve the RDEs 
\begin{align*}
Y_t&=y+\int_0^tb(s,Y_s,u_s)\dd s+\int_0^t\sigma(s,Y_s)\dd\mbX_s,
\ \ \ 
Y^{\pi_1}_t=y+\int_0^tb(s,Y^{\pi_1}_s,u^{\pi_1}_s)\dd s+\int_0^t\sigma(s,Y^{\pi_1}_s)\dd\mbX_s,
\end{align*}
on $[0,T]$, 
where %
$t_1\in[0,T]$ is a Lebesgue point, %
$\eta_1\in[0,T-t_1]$, %
$\bar{u}_1\in U$, %
and the needle-like variation $\pi_1=(t_1,\eta_1,\bar{u}_1)$ of $u$ is defined by 
 $$
 u^{\pi_1}_t=\begin{cases}
 \bar{u}_1\quad&\text{if }t\in[t_1,t_1+\eta_1],
 \\
 u_t&\text{otherwise}.
 \end{cases}
 $$ 
 Let 
$(V^{\pi_1},(V^{\pi_1})')\in\sD^p_X$ with $(V^{\pi_1})'=\nablaof{x}\sigma(\cdot,Y_\cdot)V^{\pi_1}_\cdot$ solve the linear RDE
\begin{subequations}\label{eq:rde:linear_variation}
\begin{align}
V^{\pi_1}_t&=b(t_1,Y_{t_1},\bar{u}_1)-b(t_1,Y_{t_1},u_{t_1}),
&&t\in[0,t_1],
\\
V^{\pi_1}_t &= 
V^{\pi_1}_{t_1}
+
\int_{t_1}^t
\nablaof{x}b(s,Y_s,u_s)V^{\pi_1}_s\dd s +
\int_{t_1}^t
\nablaof{x}\sigma(s,Y_s)V^{\pi_1}_s\dd\mbX_s,
 &&t\in[t_1,T].
\end{align}
\end{subequations}
Then, there exist constants $C>0$ and $0<\alpha<1$ that depend on \ellone{$(p,T,b,\sigma)$} such that 
\begin{align}
\label{eq:needle_like_deltasols}
\|Y^{\pi_1}-Y\|_{\infty,[0,T]}
&\leq 
C\exp\left(
CN_{\alpha,[0,T]}(\mbX)
\right)
\eta_1,
\\
\label{eq:needle_like_deltasols_variation}
\|Y^{\pi_1}-Y-\eta_1V^{\pi_1}\|_{\infty,[t_1+\eta_1,T]}&\leq C\exp\left(
CN_{\alpha,[0,T]}(\mbX)
\right)
o(\eta_1).
\end{align}
\end{proposition} 
\begin{proof}
\ellone{The proof follows the steps of \cite[Proposition 4.2]{Lew2026}, using the modified estimates in \eqref{eq:RDE:DY:close:full_interval:u_subinterval} 
that depend on the impulse duration $\eta_1$ and to avoid dependencies on the control values $u$ and $\bar u_1$.}
\end{proof}
Proposition \ref{prop:linear_variation}  can be extended to needle variations with multiple spikes by induction, see \cite{Lew2026}.

\begin{restatement}
  [Needle variations for rough SDEs {\cite[Lemma 4.3 \ellone{(modified)}]{Lew2026}}]
  {Lemma}
  {lem:needle:rough}
Let $\rho,p,\mbB,\ell,U,u,b,\sigma,x,x'$ be as in Theorem \ref{thm:sdes:rough}, 
where $b$ also satisfies Assumption \ref{assumption:b:stronger}  and is Lipschitz in $u$, $\sigma\in C_b^4$, 
 $\mbB=(B,\bB)$ is an enhanced Gaussian process and $(x(\omega),x'(\omega))\in\sD^p_{B(\omega)}$  solves 
\begin{align}
\tag{\ref{eq:random_gaussian_rde}}
x_t(\omega)&=
\bar x_0(\omega)
+\int_0^tb(s,x_s(\omega),u_s(\omega))\dd s+\int_0^t\sigma(s,x_s(\omega))\dd\mbB_s(\omega),
\quad\ \ t\in[0,T].
\end{align}  
Let $0<t_1<\dots<t_q<T$ be Lebesgue times \ellone{satisfying  \eqref{eq:lebesgue_time}, 
$\bar{u}_i\in L^\ell_{\F_{t_i}}(\Omega,U)$} for $i=1,\dots,q$, 
 $0\leq \eta_i< t_{i+1}-t_i$ for $i=1,\dots, q-1$ and $0\leq\eta_q< T-t_q$, and define the needle-like variation $\pi=\{t_1,\dots,t_q,\eta_1,\dots,\eta_q,\bar{u}_1,\dots,\bar{u}_q\}$ of $u$ as the control $u^\pi$ defined by%
 $$
  u^\pi_t=\begin{cases}
 \bar{u}_i\quad&\text{if }t\in[t_i,t_i+\eta_i],
 \\
 u_t&\text{otherwise}.
 \end{cases}
 $$
Let $(x^\pi(\omega),(x^\pi(\omega))')\in\sD^p_{B(\omega)}([0,T],\R^n)$ be the pathwise solution to the RDE
\begin{align*} 
x^\pi_t(\omega)=\bar x_0(\omega)+\int_0^tb(s,x^\pi_s(\omega),u^\pi_s(\omega))\dd s+\int_0^t\sigma(s,x^\pi_s(\omega))\dd\mbB_s(\omega),
\  t\in[0,T],
\end{align*}
and for $i=1,\dots,q$, let $(v^{\pi_i}(\omega),(v^{\pi_i}(\omega))')\in\sD^p_{B(\omega)}([0,T],\R^n)$ be the pathwise solutions to the RDEs 
\begin{subequations}
\begin{align*}
v^{\pi_i}_t(\omega) &= 
v^{\pi_i}_{t_i}(\omega)
+
\int_{t_i}^t
\nablaof{x}b(s,x_s(\omega),u_s(\omega))v^{\pi_i}_s(\omega)\dd s +
\int_{t_i}^t
\nablaof{x}\sigma(s,x_s(\omega))v^{\pi_i}_s(\omega)\dd\mbB_s(\omega),
 &&t\in[t_i,T],
\\
v^{\pi_i}_t(\omega)&=b(t_i,x_{t_i}(\omega),\bar{u}_i(\omega))-b(t_i,x_{t_i}(\omega),u_{t_i}(\omega)),
&&t\in[0,t_i],
\end{align*}
\end{subequations}
with $x,x^\pi,v^{\pi_1},\dots,v^{\pi_q}\in L^\ell(\Omega,C([0,T],\R^n))$.  
Let $\Phi:\R^n\to\R^k$ be continuously differentiable and satisfy $\big\|\nabla\Phi(x)\big\|\leq C_\Phi$ 
and 
$\big\|\nabla\Phi(x)-\nabla\Phi(\tilde{x})\big\|\leq C_\Phi\|x-\tilde{x}\|$ for all $x,\tilde{x}\in\R^n$ for some $C_\Phi<\infty$. 
 Then, 
\begin{align}
\label{eq:needle_like_deltasols:B}
\E\left[
\left\|
\Phi(x^\pi)-\Phi(x)
\right\|_{\infty,[0,T]}
\right]
&\leq 
C\,
\E\left[\exp(CN_{\alpha,[0,T]}(\mbB))\right]
\sum_{i=1}^q\eta_i,
\\[-1mm]
\label{eq:needle_like_deltasols_variation:B}
\E\bigg[
\bigg\|
\Phi(x^\pi)-\Phi(x)-\sum_{i=1}^q\eta_i\nabla\Phi(x)v^{\pi_i}
\bigg\|_{\infty,[t_q+\eta_q,T]}
\bigg]
&\leq 
C\,
\E\left[\exp(CN_{\alpha,[0,T]}(\mbB))\right]
o\bigg(\sum_{i=1}^q\eta_i\bigg),
\end{align}
where $0<C<\infty$ and $0<\alpha<1$ depend on \ellone{$(p,T,b,\sigma,\Phi)$}, and $\E\left[\exp(CN_{\alpha,[0,T]}(\mbB))\right]<\infty$.
\end{restatement} 

\begin{proof}[Proof of Lemma \ref{lem:needle:rough}]
\ellone{
The proof follows the proof of \cite[Lemma 4.3]{Lew2026}, using the bounds  \eqref{eq:needle_like_deltasols}-\eqref{eq:needle_like_deltasols_variation} that do not depend on the (now stochastic) control values $u(\omega),\bar u_1(\omega), \dots, \bar u_q(\omega)$. 
}
\end{proof}

\bibliographystyle{IEEEtran}
\bibliography{main}

@Preamble{"\newcommand{\noopsort}[1]{} " #
"\newcommand{\printfirst}[2]{#1} " #
"\newcommand{\singleletter}[1]{#1} " #
"\newcommand{\switchargs}[2]{#2#1} "}

@String { jrn_EDPS_ESAIMCOCV        = {{ESAIM: Control, Optimisation \& Calculus of Variations}} }

@String { jrn_SIAM_JCO              = {{SIAM Journal on Control and Optimization}} }

@String { jrn_SIAM_JNA              = {{SIAM Journal on Numerical Analysis}} }

@String { proc_ICLR                 = {{Int.\ Conf.\ on Learning Representations}} }

@String { pub_Springer              = {{Springer}} }

@String { pub_Springer_NY           = {{Springer New York}} }

@String { pub_SSBM                  = {{Springer Science \& Business Media}} }

@String { pub_Wiley                 = {{John Wiley \& Sons}} }

@Book{LeGall2016,
  Title                    = {Brownian Motion, Martingales, and Stochastic Calculus},
  Author                   = {Le Gall, J.~F.},
  Publisher                = pub_Springer,
  Year                     = {2016}
}

@book{Agrachev2004,
  year = {2004},
  publisher = {Springer Berlin Heidelberg},
  author = {Agrachev, A.~A. and Sachkov, Y.~L.},
  title = {Control Theory from the Geometric Viewpoint}
}

@article{Diehl2016,
  title = {Stochastic control with rough paths},
  volume = {75},
  number = {2},
  journal = {Applied Mathematics \&; Optimization},
  publisher = pub_SSBM,
  author = {Diehl, J. and Friz, P.~K. and Gassiat, P.},
  year = {2016},
  pages = {285-315}
}

@article{Allan2020,
  title = {Pathwise stochastic control with applications to robust filtering},
  volume = {30},
  number = {5},
  journal = {The Annals of Applied Probability},
  publisher = {Institute of Mathematical Statistics},
  author = {Allan, A.~L. and Cohen, S.~N.},
  year = {2020},
}

@book{Lyons2002,
  title = {System Control and Rough Paths},
  publisher = {Oxford University Press},
  author = {Lyons, T. and Qian, Z.},
  year = {2002},
}

@book{Friz2010,
  title = {Multidimensional Stochastic Processes as Rough Paths: Theory and Applications},
  publisher = {Cambridge University Press},
  author = {Friz, P.~K. and Victoir, N.~B.},
  year = {2010},
}

@book{Friz2020,
  title = {A course on rough paths: with an introduction to regularity structures},
  journal = {Universitext},
  publisher = {Springer International Publishing},
  author = {Friz, P.~K. and Hairer, M.},
  year = {2020}
}

@misc{Allan2021,
  author       = {Allan, A.},
  title        = {Lecture Notes on Rough Path Theory},
  howpublished = {ETH Z\"urich},
  year         = {2021}
}

@article{Friz2018,
  title = {Differential equations driven by rough paths with jumps},
  volume = {264},
  number = {10},
  journal = {Journal of Differential Equations},
  publisher = {Elsevier {BV}},
  author = {Friz, P.~K. and Zhang, H.},
  year = {2018},
  pages = {6226--6301}
}

@article{Friz2013,
  title = {Integrability of (Non-)Linear Rough Differential Equations and Integrals},
  volume = {31},
  number = {2},
  journal = {Stochastic Analysis and Applications},
  author = {Friz, P. and Riedel, S.},
  year = {2013},
  pages = {336--358}
}

@article{Bayer2016,
  title = {From Rough Path Estimates to Multilevel {Monte} {Carlo}},
  volume = {54},
  number = {3},
  journal = jrn_SIAM_JNA,
  author = {Bayer, C. and Friz, P.~K. and Riedel, S. and Schoenmakers, J.},
  year = {2016},
  month = jan,
  pages = {1449--1483}
}

@article{Berret2020,
  title = {Stochastic optimal open-loop control as a theory of force and impedance planning via muscle co-contraction},
  volume = {16},
  number = {2},
  journal = {{PLOS} Computational Biology},
  author = {Berret, B. and Jean, F.},
  year = {2020},
}

@book{Yong1999,
  title = {Stochastic controls: {H}amiltonian systems and {HJB} equations},
  publisher = pub_Springer_NY,
  author = {Yong, J. and Zhou, X.~Y.},
  year = {1999}
}

@article{BonalliLewESAIM2022,
  author    = {Bonalli, R. and Lew, T. and Pavone, M.},
  title     = {Sequential Convex Programming For Non-Linear Stochastic Optimal Control},
  journal   = jrn_EDPS_ESAIMCOCV,
  volume    = {28},
  number    = {64},
  year      = {2022},
}

@article{E2017,
  title = {Deep Learning-Based Numerical Methods for High-Dimensional Parabolic Partial Differential Equations and Backward Stochastic Differential Equations},
  volume = {5},
  number = {4},
  journal = {Communications in Mathematics and Statistics},
  publisher = {Springer Science and Business Media LLC},
  author = {E,  W. and Han, J. and Jentzen, A.},
  year = {2017},
  pages = {349--380}
}

@article{Rogers2007,
  title = {Pathwise Stochastic Optimal Control},
  volume = {46},
  number = {3},
  journal = jrn_SIAM_JCO,
  author = {Rogers,  L.~C.~G.},
  year = {2007},
  pages = {1116--1132}
}

@article{HADavis1992,
  title = {A Deterministic Approach To Stochastic Optimal Control With Application To Anticipative Control},
  volume = {40},
  journal = {Stochastics and Stochastic Reports},
  author = {Davis, H.~A.~M. and Burstein, G.},
  year = {1992},
  pages = {203--256}
}

@article{Lew2026,
  title = {Rough Stochastic {P}ontryagin Maximum Principle and an Indirect Shooting Method},
  volume = {64},
  number = {2},
  journal = jrn_SIAM_JCO,
  author = {Lew, T.},
  year = {2026},
  pages = {869--905}
}

@INPROCEEDINGS{Domingo2025,
author    = {Domingo-Enrich, C. and Drozdzal, M. and Karrer, B. and Chen, R.~T.~Q},
title     = {Adjoint Matching: Fine-tuning Flow and Diffusion Generative Models with Memoryless Stochastic Optimal Control},
booktitle = proc_ICLR,
year      = {2025},
}

@article{FrizControlled2024,
author = {Friz, P.~K. and L\^e, K. and Zhang, H.},
title = {Controlled rough {SDEs}, pathwise stochastic control and dynamic programming principles},
journal = {The Annals of Probability},
year = {2024},
}

@article{Domingo2026,
	title={Adjoint Matching through the Lens of the Stochastic Maximum Principle in Optimal Control},
	author={Domingo-Enrich, C. and Han, J.},
	journal={Transactions on Machine Learning Research},
	year={2026},
}

@article{Chernousko1982,
  title = {Method of successive approximations for solution of optimal control problems},
  volume = {3},
  journal = {Optimal Control Applications and Methods},
  author = {Chernousko, F.~L. and Lyubushin,  A.~A.},
  year = {1982},
  pages = {101--114}
}

@misc{Bank2026,
  author = {Bank, P. and Dause, J.~R. and de Feo, F. and Friz, P.~K.},
  title = {Duality for Stochastic Control with non-{Markovian} Random Coefficients},
  year = {2026},
	note      = {Available at \url{https://arxiv.org/abs/2609.05101}},
}

@misc{Ashkarian2026,
  author = {Ashkarian, E. and Chakraborty, P. and Honnappa, H. and Tindel, S.},
  title = {The {Pontryagin} maximum principle and {$Q$}-functions in rough environments},
  year = {2026},
	note      = {Available at \url{https://arxiv.org/abs/2601.05354}},
}

@article{Horst2026,
  author = {Horst, U. and Zhang, H.},
  title = {Pontryagin Maximum Principle for rough stochastic systems and pathwise stochastic control},
  year = {2026},
  volume = {64},
  number = {5},
  journal = jrn_SIAM_JCO,
}

@article{Buckdahn2007,
  title = {Pathwise Stochastic Control Problems and Stochastic {HJB} Equations},
  volume = {45},
  number = {6},
  journal = jrn_SIAM_JCO,
  author = {Buckdahn, R. and Ma, J.},
  year = {2007},
  pages = {2224--2256}
}

@book{EthierKurtz1986,
  author    = {Ethier, S.~N. and Kurtz, T.~G.},
  title     = {Markov Processes: Characterization and Convergence},
  publisher = pub_Wiley,
  year      = {1986}
}

\end{document}